\documentclass[11pt]{amsart}
\usepackage{amssymb,mathrsfs,graphicx,enumerate,color}
\usepackage{colortbl}
\definecolor{black}{rgb}{0.0, 0.0, 0.0}
\definecolor{red}{rgb}{1.0, 0.5, 0.5}

\title[   ]{Contraction property for large perturbations of shocks of the barotropic Navier-Stokes system}

\author[Kang]{Moon-Jin Kang}
\address[Moon-Jin Kang]{%\newline Laboratoire Jacques-Louis Lions, 
\newline Department of Mathematic \& Research Institute of Natural Sciences, \newline Sookmyung Women's University, Seoul 140-742, Korea}
\email{moonjinkang@sookmyung.ac.kr}

\author[Vasseur]{Alexis F. Vasseur}
\address[Alexis F. Vasseur]{\newline Department of Mathematics, \newline The University of Texas at Austin, Austin, TX 78712, USA}
\email{vasseur@math.utexas.edu}

\newtheorem{theorem}{Theorem}[section]
\newtheorem{lemma}{Lemma}[section]

\newtheorem{proposition}{Proposition}[section]
\newtheorem{remark}{Remark}[section]

\newcommand{\bbr}{\mathbb R}

\newcommand{\bbt} {\mathbb T}

\newcommand{\pt}{p(\tilde{v}_\eps)}
\newcommand{\deo}{\delta_0}
\newcommand{\deu}{\delta_0}

\numberwithin{figure}{section}
\newcommand{\beq}{\begin{equation}}
\newcommand{\eeq}{\end{equation}}
\newcommand{\bsp}{\begin{split}}
\newcommand{\esp}{\end{split}}

\def\eps{\varepsilon }

\newcommand\adots{\mathinner{\mkern2mu\raise1pt\hbox{.}
\mkern3mu\raise4pt\hbox{.}\mkern1mu\raise7pt\hbox{.}}}

\def\charf {\mbox{{\text 1}\kern-.30em {\text l}}}
\newcommand \zo {Z_1}
\newcommand \zt {Z_2}

\newcommand \deltat {\delta_2}

\newcommand{\Rd}{\mathcal{R}_\delta}

\newcommand{\vt}{\tilde{v}_\eps}

\begin{document}
%%%%%%%%%%%%%%%%
\bibliographystyle{plain}

\date{\today}


\thanks{\textbf{Acknowledgment.}  The first author was partially supported by the NRF-2017R1C1B5076510.
The second author was partially supported by the NSF grant: DMS 1614918. 
%and by the Foundation Sciences Math$\acute{\mbox{e}}$matiques de Paris under a postdoctoral fellowship. 
}

\begin{abstract}
This paper is dedicated to the construction of a pseudo-norm, for which small shock profiles of the barotropic Navier-Stokes equations have a contraction property. This contraction property holds in the class of any large solutions to the barotropic Navier-Stokes equations. It implies a stability condition which is independent of the strength of the viscosity. The proof is based on the relative entropy method, and is reminiscent to the notion of a-contraction first introduced by the authors in the hyperbolic case.  \end{abstract}
\maketitle \centerline{\date}

%Considering the negation of the necessary condition, we develop two kinds of criteria to identify the invalidity of contraction of intermediate entropic shocks. The first criterion is based on a geometric relation between the rarefaction wave of another genuinely nonlinear field and the $(n-1)$-dimensional surface determined by the intermediate entropic shock and the pseudo-distance. This criterion is still valid in the linearly degenerate case, i.e., for the intermediate contact discontinuity. On the other hand, the second criteria is related to a relation between the Hugoniot curve of another linearly degenerate field and the $(n-1)$-dimensional surface. As an application of the criteria to two-dimensional isentropic magnetohydrodynamics, we show that there is no contraction of certain intermediate shocks. As an application to full Euler system of gas dynamics, we find a certain range of weights for invalidity of the contraction property of 2-contact discontinuity. All results do not involve any smallness condition on the initial perturbation, nor on the size of the shock.

\tableofcontents

\section{Introduction}
\setcounter{equation}{0}

In this article, we consider  the one-dimensional barotropic Navier-Stokes equations in the Lagrangian coordinates:
\begin{align}
\begin{aligned}\label{main}
\left\{ \begin{array}{ll}
       v_t - u_x =0\\
       u_t+p(v)_x = \Big(\frac{\mu(v)}{v} u_x\Big)_x, \end{array} \right.
\end{aligned}
\end{align}
where $v$ denotes the specific volume, $u$ is the fluid velocity, and $p(v)$ is the pressure law. We consider the case of  polytropic perfect gas where the pressure verifies
\beq\label{pressure}
p(v)= v^{-\gamma},\quad \gamma> 1,
\eeq
with  $\gamma$ the adiabatic constant. The quantity  $\mu(v) = bv^{-\alpha}$ is the viscosity coefficient. 
Notice that if $\alpha>0$, $\mu(v)$ degenerates near the vacuum, i.e., near $v=+\infty$.
Very often, the viscosity coefficient is assumed to be constant, i.e., $\alpha=0$. However, in the  physical context  the viscosity of a gas depends on the temperature (see Chapman and Cowling \cite{CC}).  In the barotropic case,  the temperature depends directly on the density ($\rho=1/v$). %Therefore the viscosity coefficient depends on $v$ also. 
The viscosity  is expected to degenerate near the vacuum as a power of the density, which is translated into $\mu(v) = bv^{-\alpha}$ in terms of $v$ with $\alpha>0$. Global strong solutions  of the system \eqref{main} can be constructed for a large family of initial data without vacuum. These solutions are also unique  (see Constantin-Drivas-Nguyen-Pasqualotto \cite{CDNP},  Haspot \cite{H3} and \cite{MV_sima}).  For simplification, we will restrict in this paper to the case where $\alpha=\gamma$.

The system \eqref{main} admits viscous shock waves connecting two end states 
$(v_-,u_-)$ and $(v_+,u_+)$, provided the two end states satisfy the Rankine-Hugoniot condition and the Lax entropy condition (see Matsumura and Wang \cite{MW}): 
\begin{align}
\begin{aligned}\label{end-con} 
&\exists~\sigma\quad\mbox{s.t. }~\left\{ \begin{array}{ll}
       -\sigma (v_+-v_-) - (u_+-u_-) =0,\\
       -\sigma (u_+-u_-) +p(v_+)-p(v_-)=0, \end{array} \right. \\
&\mbox{and either $v_->v_+$ and $u_->u_+$ or $v_-<v_+$ and $u_->u_+$ holds.}        
\end{aligned}
\end{align} 
In other words, for given constant states $(v_-,u_-)$ and $(v_+,u_+)$ satisfying \eqref{end-con}, there exists a viscous shock wave $(\tilde v,\tilde u)(x-\sigma t)$ as a solution of
\begin{align}
\begin{aligned}\label{shock_0} 
\left\{ \begin{array}{ll}
       -\sigma \tilde v' - \tilde u' =0,\\
       -\sigma \tilde u'+p( \tilde v)'= \Big(\frac{\mu( \tilde v)}{ \tilde v}  \tilde u'\Big)'\\
       \lim_{\xi\to\pm\infty}(\tilde v,\tilde u)(\xi)=(v_{\pm}, u_\pm). \end{array} \right.
\end{aligned}
\end{align} 
Here, if $v_->v_+$, the solution of \eqref{shock_0} is a 1-shock wave with velocity $\sigma=- \sqrt{-\frac{p(v_+)-p(v_-)}{v_+-v_-}}$, whereas if $v_-<v_+$, that is a 2-shock wave with $\sigma= \sqrt{-\frac{p(v_+)-p(v_-)}{v_+-v_-}}$.  

The stability of the viscous shock waves for the compressible Navier-Stokes system is a very important issue in both mathematical and physical viewpoints. %It is worth mentioning some results on the stability of the viscous shocks.
In the case of constant viscosity ($\alpha=0$), Matsumura-Nishihara \cite{MN} showed the time-asymptotic stability for small initial perturbations with integral zero. Later on, the assumption on integral zero was removed by Mascia-Zumbrun \cite{MZ} and Liu-Zeng \cite{LZ}. We also refer to Barker-Humpherys-Laffite-Rudd-Zumbrun \cite{BHLRZ,HLZ} and the references therein for the spectral stability of small perturbations of large shocks. For the system \eqref{main} with degenerate viscosity ($\alpha>0$), Matsumura-Wang \cite{MW} showed the asymptotic stability for small initial perturbations with integral zero under the assumption $\alpha\ge\frac{1}{2}(\gamma-1)$. This assumption was recently removed by the second author and Yao \cite{VY}. 

To the best of our knowledge, up to now, there were  no result on stability, independent of the size of the  perturbation, for viscous shocks of  compressible Navier-Stokes system.%, no matter what assumption on the shocks or the systems is considered. 

The main contribution of this article is to show the existence of a  contraction property for viscous shocks, up to a shift,   for any possibly  large perturbations, in the case of  the Navier-Stokes system \eqref{main} with $\alpha=\gamma$ (see Theorem \ref{thm_general}). 
%More precisely, a large perturbation is contractive up to a Lipschitz shift, and the shift reflects the dynamics of the perturbation. 

This result reaches a new milestone in the study of contractions of shock waves of conservation laws  based on the relative entropy. In the inviscid case, the $L^2$ contraction of  shocks was first  obtained by Leger \cite{Leger}  for scalar conservation laws (see also Adimurthi,  Goshal, and Veerappa Gowda \cite{LP} for contraction in the $L^p$ norm). In \cite{Serre-Vasseur}, it was shown that this property is not true, for most systems, when considering homogenous norms. However it is true, at least for extremal shocks, if we consider an adapted non-homogenous  pseudo-norm \cite{LV, Vasseur-2013}. This was theorized with the notion of $a$-contraction in \cite{KVARMA}. There, the case of intermediate shocks was also considered. This situation is more delicate. The contraction works for some systems, as the Euler system with energy \cite{SV_16dcds, SV_16}, and can fail for others \cite{Kang}. 
In the viscous case, based on the $L^2$ norm  a first result was obtained for  viscous shocks in the case of the viscous Burgers equation \cite{Kang-V-1} (see also \cite{Kang19}). Our paper can be seen as a generalization of this result in the system case. Of course,  the system case is far more involved. Especially, since these results are independent of the size of the perturbations, by rescaling the equation, they are valid uniformly in the vanishing viscosity limit. Because of the negative result of \cite{Serre-Vasseur} for the Euler system,  the result cannot be true for the Navier-Stokes equations when considering a homogenous pseudo-norm.
This difficulty is compounded with the degenerate parabolic structure of Navier-Stokes, where the equation on $v$ is purely hyperbolic.\\
We also mention a first attempt to extend the theory to the multi-variables setting in the scalar case \cite{KVW}, and the application of the method for the study of asymptotic limits  \cite{CV,VW}. 

In an analytical viewpoint, handling the contraction property of the viscous shocks is pretty different  from the  inviscid situation. The main difficulty is due to the  balance between the hyperbolic and parabolic terms.

\subsection{Main result}
We first introduce a relative functional $E(\cdot|\cdot)$ defined as follows:
\begin{align}
\begin{aligned}\label{psedo}
&\mbox{for any functions } v_1,u_1,v_2,u_2,\\ 
&E((v_1,u_1)|(v_2,u_2)) :=\frac{1}{2}\big(u_1 +p(v_1)_x -u_2 -p(v_2)_x  \big)^2 +Q(v_1|v_2),
\end{aligned}
\end{align}
where $Q(v_1|v_2):=Q(v_1)-Q(v_2)-Q'(v_2)(v_1-v_2)$ is a relative functional associated with the strictly convex function $Q(v):=\frac{v^{-\gamma+1}}{\gamma-1}$. The functional $E$ is associated to the BD entropy (see Bresch-Desjardins \cite{BD_03,BD_06,BDL}). Since $Q(v_1|v_2)$ is positive definite, \eqref{psedo} is also positive definite, that is, for any functions $(v_1,u_1)$ and $(v_2,u_2)$ we have $E((v_1,u_1)|(v_2,u_2))\ge 0$, and 
\[
%,\quad\mbox{and}
\quad E((v_1,u_1)|(v_2,u_2))= 0~\mbox{a.e.} \quad\Leftrightarrow\quad (v_1,u_1)=(v_2,u_2)~\mbox{a.e.}
\]
Our main result shows  a  contraction property  measured by the relative functional \eqref{psedo}. 
Our result is stated for the system \eqref{main} with the viscosity $\mu(v)=\gamma v^{-\gamma}$, i.e., the exponent $\alpha$ is identical to the adiabatic constant $\gamma$. 
A new approach developed in this paper can be applied to the case of more general viscosity (see \cite{KV_vanishing}). 

\begin{theorem}\label{thm_general}
Consider the system \eqref{main}-\eqref{pressure} with the viscosity $\mu(v)=\gamma v^{-\gamma}$, $\gamma>1$.  For a given constant state $(v_-,u_-)\in\bbr^+\times\bbr$, 
there exists  constants  $\eps_0, \delta_0>0$ such that the following is true.\\
For any $\eps<\eps_0$, $\delta_0^{-1}\eps<\lambda<\delta_0$, and any $(v_+,u_+)\in\bbr^+\times\bbr$ satisfying \eqref{end-con} with $|p(v_-)-p(v_+)|=\eps$, there exists  a smooth monotone function $a:\bbr\to\bbr^+$ with $\lim_{x\to\pm\infty} a(x)=1+a_{\pm}$ for some  constants $a_-, a_+$ with $|a_+-a_-|=\lambda$ such that the following holds.\\
Let $\tilde U:=(\tilde v,\tilde u)$ be the viscous shock connecting $(v_-,u_-)$ and $(v_+,u_+)$ as a solution of \eqref{shock_0}.
For any solution $U:=(v,u)$ to \eqref{main} with initial data $U_0:=(v_0,u_0)$ satisfying $\int_{-\infty}^{\infty} E(U_0| \tilde U) dx<\infty$, there exists a shift $X\in W^{1,1}_{loc}(\bbr^+)$ such that 
\beq%\label{cont_main}
 \frac{d}{dt}\int_{-\infty}^{\infty} a(x) E\big(U(t,x+X(t))| \tilde U (x)\big) dx \le 0,
\eeq
and 
\begin{align}
\begin{aligned} \label{est-shift}
&|\dot X(t)|\le \frac{1}{\eps^2}(1 + f(t)),\quad t>0,\\
&\mbox{for some positive function $f$ satisfying}\quad\|f\|_{L^1(0,\infty)} \le\frac{2\lambda}{\delta_0\eps}\int_{-\infty}^{\infty} E(U_0| \tilde U) dx.
\end{aligned}
\end{align}
\end{theorem}

\begin{remark}
Theorem \ref{thm_general} provides a contraction property for viscous shocks with suitably small amplitude parametrized by $\eps=|p(v_-)-p(v_+)|$. This smallness together with \eqref{end-con} implies $|v_--v_+|=\mathcal{O}(\eps)$ and $|u_--u_+|=\mathcal{O}(\eps)$. 
However, for such a fixed small shock, the contraction holds for any weak solutions  to (\ref{main}), without any smallness condition imposed on $U_0$. This is important to study the  inviscid limit problem  ($\nu\to0$) of: 
\begin{align}
\begin{aligned} \label{inveq}
\left\{ \begin{array}{ll}
       v^{\nu}_t - u^{\nu}_x =0\\
       u^{\nu}_t+p(v^{\nu})_x = \nu\Big(\frac{\mu(v^{\nu})}{v^{\nu}} u^{\nu}_x\Big)_x. \end{array} \right.
\end{aligned}
\end{align}
By rescaling the result of Theorem \ref{thm_general} as $(t,x)\to (t/\nu,x/\nu)$ we obtain the exact same theorem for the system (\ref{inveq}). Therefore we obtain a stability result on viscous shocks of fixed strength which is independent of the strength of the viscosity $\nu$ (see \cite{KV_vanishing}). 
%We leave th to a future work.
\end{remark}

\begin{remark}
The contraction property is non-homogenous in $x$, as measured by the function $x\to a(x)$. This is consistant with the hyperbolic case (with $\nu=0$). In the  hyperbolic case, it was shown in \cite{Serre-Vasseur} that a homogenous contraction cannot hold  for the full Euler system. However, the contraction property is true if we consider a non-homogenous pseudo-distance \cite{Vasseur-2013} providing the so-called $a$-contraction \cite{KVARMA}. Our main result shows that the non-homogeneity of the pseudo-distance can be chosen of a similar size as the strength of the shock (as measured by the quantity $\lambda$).
\end{remark}

\subsection{Transformation of the system \eqref{main}}
We first introduce a new effective velocity $h:=u+p(v)_x$. The system \eqref{main} with $\mu(v)=\gamma v^{-\gamma}$ is then  transformed into
\begin{align}
\begin{aligned}\label{NS_1}
\left\{ \begin{array}{ll}
       v_t - h_x = -(p(v))_{xx}\\
       h_t+p(v)_x =0. \end{array} \right.
\end{aligned}
\end{align}
Notice that the above system has a parabolic regularization on the specific volume, contrary to the regularization on the velocity for the original system \eqref{main}. This is better for our analysis, since the hyperbolic part of the system is linear in $u$ (or $h$) but nonlinear in $v$ (via the pressure).
This effective velocity  was first introduced by Shelukhin \cite{Shel}  for $\alpha=0$, and in the general case (in Eulerian coordinates)  by Bresch-Desjardins \cite{BD_03,BD_06,BDL}, and Haspot \cite{H1,H2,H3}.  It was also used  in \cite{VY}.

As mentioned in Theorem \ref{thm_general}, we consider  shock waves with suitably small amplitude $\eps$. For that, let $(\tilde v_\eps,\tilde u_\eps)(x-\sigma_\eps t)$ denote a shock wave with amplitude $|p(v_-)-p(v_+)|=\eps$ as a solution of \eqref{shock_0} with $\mu(v)=\gamma v^{-\gamma}$. Then, setting $\tilde h_\eps:=\tilde u_\eps + (p(\tilde v_\eps))_x$, the shock wave $(\tilde v_\eps,\tilde h_\eps)(x-\sigma_\eps t)$ satisfies
\begin{align}
\begin{aligned}\label{small_shock1} 
\left\{ \begin{array}{ll}
       -\sigma_\eps \tilde v_{\eps}' - \tilde h_{\eps}' =-(p(\tilde v_\eps))''\\
       -\sigma_\eps \tilde h_{\eps}'+p( \tilde v_\eps)'=0\\
       \lim_{\xi\to\pm\infty}(\tilde v_\eps, \tilde h_\eps)(\xi)=(v_{\pm}, u_\pm). \end{array} \right.
\end{aligned}
\end{align}

For simplification of our analysis, we rewrite \eqref{NS_1} into the following system, based on the change of variable $(t,x)\mapsto (t, \xi=x-\sigma_\eps t)$: 
\begin{align}
\begin{aligned}\label{NS}
\left\{ \begin{array}{ll}
       v_t -\sigma_\eps v_{\xi} - h_{\xi} = -(p(v))_{\xi\xi}\\
       h_t-\sigma_\eps h_{\xi}+p(v)_{\xi} =0\\
       v|_{t=0}=v_0,\quad h|_{t=0}=u_0. \end{array} \right.
\end{aligned}
\end{align}

\begin{remark}
In \eqref{NS}, the dissipation is in $v$ and has the specific form $(-p(v))_{\xi\xi}$, whose structure is due to the fact that $\alpha=\gamma$. This simplifies our analysis a lot, since we consider the entropy $Q(v)$ with $Q'(v)=-p(v)$.
\end{remark}

Theorem \ref{thm_general} is a direct consequence of the following theorem on the contraction of shocks for the system \eqref{NS_1}. To measure the contraction, we use the relative entropy associated to the entropy of \eqref{NS_1} as
\[
\eta((v_1,h_1)|(v_2,h_2)) :=\frac{|h_1-h_2|^2}{2} +Q(v_1|v_2),
\]
where $Q(v_1|v_2):=Q(v_1)-Q(v_2)-Q'(v_2)(v_1-v_2)$ and $Q(v):=\frac{v^{-\gamma+1}}{\gamma-1}$.

\begin{theorem}\label{thm_main}
For a given constant state $(v_-,u_-)\in\bbr^+\times\bbr$, there exist constants $\eps_0,\delta_0>0$ such that the following holds.\\
For any $\eps<\eps_0$, $\delta_0^{-1}\eps<\lambda<\delta_0$, and any $(v_+,u_+)\in\bbr^+\times\bbr$ satisfying \eqref{end-con} with $|p(v_-)-p(v_+)|=\eps$, there exists a  smooth monotone function $a:\bbr\to\bbr^+$ with $\lim_{x\to\pm\infty} a(x)=1+a_{\pm}$ for some  constants $a_-, a_+$ with $|a_--a_+|=\lambda$ such that the following holds.\\
Let $\tilde U_\eps:=(\tilde v_\eps,\tilde h_\eps)$ be a viscous shock connecting $(v_-,u_-)$ and $(v_+,u_+)$ as a solution of \eqref{small_shock1}.
For any solution $U:=(v,h)$ to \eqref{NS} with initial data $U_0:=(v_0,u_0)$ satisfying $\int_{-\infty}^{\infty} \eta(U_0| \tilde U_\eps) dx<\infty$, there exists a shift function $X\in W^{1,1}_{loc}(\bbr^+)$ such that 
\beq\label{cont_main}
 \frac{d}{dt}\int_{-\infty}^{\infty} a(\xi) \eta\big(U(t,\xi+X(t))| \tilde U_\eps (\xi)\big) d\xi \le 0,
\eeq
and
\begin{align}
\begin{aligned} \label{est-shift1}
&|\dot X(t)|\le \frac{1}{\eps^2}(1 + f(t)),\quad t>0,\\
&\mbox{for some positive function $f$ satisfying}\quad\|f\|_{L^1(0,\infty)} \le \frac{2\lambda}{\delta_0\eps}\int_{-\infty}^{\infty} \eta(U_0| \tilde U_\eps) d\xi.
\end{aligned}
\end{align}
\end{theorem}

Notice that it is enough to prove Theorem \ref{thm_main}  for  1-shocks. Indeed, the result for 2-shocks is obtained by the change of variables $x\to -x$, $u\to -u$, $\sigma_\eps\to -\sigma_\eps$. \\
Therefore, from now on, we consider a 1-shock $(\tilde v_\eps,\tilde h_\eps)$, i.e., $v_->v_+$, $u_->u_+$, and
\beq\label{RH-con}
\sigma_\eps=- \sqrt{-\frac{p(v_+)-p(v_-)}{v_+-v_-}}.
\eeq

{\bf Notations }
$\bullet$  Throughout the paper, $C$ denotes a positive constant which may change from line to line, but which stays independent on $\eps$ (the shock strength) and $\lambda$ (the total variation of the function $a$). %and the small constant $\eps_0, \delta_i,~i=0,\cdots,4$. We sometimes use $f\lesssim g$ to denote $f\le C g$.
The paper will consider two smallness conditions, one on $\eps$, and the other on $\eps/\lambda$. In the argument, $\eps$ will be far smaller than $\eps/\lambda$ .\\
$\bullet$ To avoid confusion, for any function $F$ of $x$, we denote: $F'(v)=\frac{d}{dv}F(v)$, $F(v)'=\frac{d}{dx}F(v)$.

\subsection{Ideas of the proof}
In all the computations $\eps>0$ is the size of the fixed shock. We remind the reader that the perturbation $U_0-\tilde{U}_\eps=(v_0-\tilde{v}_\eps, h_0-\tilde{h}_\eps)$ can be unconditionally big. The non-homogeneity of the semi-norm comes through the function $a$. This function is decreasing in the case of a 1-shock, and increasing in the case of a 2-shock. The strength of this non-homogeneity is measured by the number $\lambda>0$, which is the difference between the values of $a$ at $-\infty$ and $+\infty$ (see (\ref{weight-a})).  Typically, $\lambda$ is small, but it can be far bigger than $\eps$. Actually, in the analysis, we will consider some smallness on both $\eps$ and $\eps/\lambda$, $\eps$ being much smaller that $\eps/\lambda$. Note that the velocity of the shock ${\sigma}_\eps$ has the same sign as $a'$, so the quantity ${\sigma}_\eps a'$ is positive. The relative entropy computation (see Lemma \ref{lem-rel}) gives that 
\begin{eqnarray*}
&& \frac{d}{dt}\int_{-\infty}^{\infty} a(\xi) \eta\big(U(t,\xi+X(t))| \tilde U_\eps (\xi)\big) d\xi\\
 &&\qquad=\dot X(t) Y(U(t,\cdot+X(t))) +\mathcal{B}(U(t,\cdot+X(t)))- \mathcal{G}(U(t,\cdot+X(t))).
\end{eqnarray*}
The functional  $\mathcal{G}(U)$ is non-negative  (good term) and can be split into three terms (see \eqref{ggd}):
$$
\mathcal{G}(U)=\mathcal{G}_1(U)+\mathcal{G}_2(U)+\mathcal{D}(U),
$$
where only $\mathcal{G}_1(U)$ depends on $h$. The term $\mathcal{D}(U)$ corresponds to the diffusive term (which depends on $v$ only, thanks to the transformation of the system). We are able to write this decomposition such that the functional  $\mathcal{B}(U)$  (bad terms) depends only on $v$.
This is  the main reason why we can consider a degenerate diffusion (the viscosity in $u$ only is  replaced by a diffusion in $v$ only, after transformation of the system).  The fact that the hyperbolic flux in the Navier-Stokes equations is only  linear in $h$ plays a particular role for this matter: the corresponding relative flux then vanishes. 

Because of the relative entropy structure, the quantity $\mathcal{G}(U)$ and $\mathcal{B}(U)$ are quadratic when the perturbation is small.  However, we have no uniform control on the size of $U(t,\cdot)$, therefore we will have also to carefully estimate what happens for large value of $U(t,x)$. 

The shift $X(t)$ introduces the term $\dot{X}(t) Y(U)$.  The key idea of the technique, is to take advantage of this term when $Y(U(t,\cdot))$ is not two small, by compensating all the other terms via the choice of the velocity of the shift (see (\ref{X-def})). Specifically, we ensure algebraically that the contraction holds as long as $|Y(U(t))|\geq\eps^2$. The rest of the analysis is to ensure that when $|Y(U(t))|\leq \eps^2$, the contraction still holds. 

The condition  $|Y(U(t))|\leq \eps^2$ ensures a smallness condition that we want to fully exploit. This is where the non-homogeneity of the semi-norm is crucial. In the  case where the function $a$ is constant,
$Y(U)$ is a linear functional in $U$. The smallness of $Y(U)$ gives only that a certain weighted mean value of $U$ is almost null. However, when $a$ is decreasing, $Y(U)$ becomes convex. The smallness $Y(U(t))\leq \eps^2$ implies, for this fixed time $t$ (See Lemma \ref{lemmeC2} with \eqref{d-weight} and \eqref{tail}):
\begin{equation}\label{small}
\int_{\bbr}\eps e^{-C\eps |\xi|}Q(v(t,\xi +X(t))|\tilde{v}_\eps(\xi))\,d\xi\leq C\left(\frac{\eps}{\lambda}\right)^2.
\end{equation}
This gives a control in $L^2$ for moderate values of $v$, and in  $L^1$ for big values of $v$,  in the layer region ($|\xi-X(t)|\lesssim 1/\eps$). 

The problem now looks, at first glance, as a typical problem of stability with a smallness condition.
There are, however, two major difficulties: We have some smallness only in $v$, for a very weak norm, and only localized in the layer region. More importantly, the smallness is measured with respect to the smallness of the shock. It basically says that, considering only the moderate values of $v$:
 the perturbation is not bigger than $\eps/\lambda$ (which is still very big with respect to the size of the shock  $\eps$). Actually, as we will see later, it is not possible to consider only the linearized problem: Third order terms appear in the  expansion using the smallness condition (the energy method involving the linearization would have only second order term in $\eps$).
 
In the argument, for the values of $t$ such that $|Y(U(t))|\leq\eps^2$,  we construct the shift as a solution to the ODE: $\dot X(t)=-Y(U(t,\cdot+X(t)))/\eps^4$. 
 From this point, we forget that $U=U(t,\xi)$ is a solution to \eqref{NS} and $X(t)$ is the shift. That is, we leave out $X(t)$ and the $t$-variable of $U$. Then we show that for any function $U$ satisfying $Y(U)\leq \eps^2$, we have 
 \begin{equation}\label{but}
 -\frac{1}{\eps^4}Y^2(U)+|\mathcal{B}(U)|-\mathcal{G}(U)\leq0.
 \end{equation}
This is the main Proposition \ref{prop:main} (actually, the proposition is slightly stronger to ensure the control of the shift). This implies clearly the contraction. There are several steps to prove this proposition.
\vskip0.3cm
\noindent{\it Step 1}:
 Using the smallness condition, we show that if the good diffusive term verifies 
$$
\mathcal{D}(U)\geq \frac{\eps^2}{\lambda},
$$ 
then (\ref{but}) holds true.
Note that if the values of $v$ were bounded from above  and bounded away from 0, we could control $\mathcal{B}(U)$ from (\ref{small}), since both expressions  would be  quadratic in $v-\tilde{v}_\eps$. The main difficulty in this step is to obtain the control where  the values of $v$ are small. Indeed for such small $v$, the worst term in $\mathcal{B}(U)$ behaves  like $p(v)^2=1/v^{2\gamma}$, while $Q(v|\tilde{v}_\eps)$ behaves like $1/v^{\gamma-1}$.  So we need to use a little bit of $\mathcal{D}(U)$ as a Poincar\'e type inequality  (Remember that $a\geq 1-\lambda>0$) from:
$$
\mathcal{D}(U)=\int_\bbr a |\partial_\xi (p(v)-p(\tilde v_\eps))|^2 d\xi
$$
(See \eqref{n2} from Lemma \ref{lemma3}). 
We can now restrict ourselves to  the case where both $|Y(U)|\leq \eps^2$, and $\mathcal{D}(U)\leq \eps^2/\lambda$. 

\vskip0.3cm
\noindent{\it Step 2}: To be able to perform an expansion in $\eps$ later, we want to show that it is enough to consider only values of  $v$ such that $v-\tilde{v}_\eps$ is bounded (smaller than a $\delta$ small enough, but not dependent on $\eps$ nor on $\eps/\lambda$). We need also use only the part $Y_g(v)$ of $Y(U)$ which contains only terms in $v$ (and not the terms in $h$). We do not have enough estimates on  $U$ to show that $U$ is uniformly bounded on $\bbr$. But we can show that the big values of $|v-\tilde{v}_\eps|$ (which can occur only for big values of $\xi$) do not change much the estimate  (see Section \ref{section-finale}).  It involves a careful study of the contribution of the tails ($U(\xi)$ for $|\xi|\ge 1/\eps$). This is the only part where $\mathcal{G}_1$ is used in order to control $Y_b(U)=Y(U)-Y_g(v)$, the part of $Y(U)$ which depends also on $h$ (see Lemma \ref{lemma3}).
More precisely, this step shows that it is enough to prove  that for any functions $v$ such that $|v-\tilde{v}_\eps|\leq \delta$ and $|Y_g(v)|\leq \eps^2/\lambda$, we have
$$
-\frac{1}{\eps\delta}|Y_g(v)|^2+(1+\delta)|\mathcal{B}(v)|-(1-\delta)\mathcal{G}_2(v)-(1-\delta)\mathcal{D}(v)  \leq 0.
$$
All the terms in this inequality depends on $U$ only through $v$. Therefore, with a slight abuse of notations, we will write these functions as functions of $v$. 
This corresponds to   Proposition \ref{prop:main3}. The $\delta$ terms are still needed because we lose a bit when truncating the tails, to obtain  (\ref{but}). The terms depending on $h$ are not present anymore. So it is now an estimate on  scalar functions $v$. The good term in $Y_g(v)$ involves a smaller power of $1/\eps$, since we had to control the corresponding $Y_b(U)$ with the same power of $1/\eps$. 

\vskip0.3cm
\noindent{\it Step 3}: To show  Proposition \ref{prop:main3}, we now perform a expansion in $\eps$ uniformly  in $v$ (but for a fixed $\delta$). Note that the expansion has to be performed up to the third order. Indeed, because of the function $a$, terms involving the function $a$ or the functions $a'$ do not have the same power in $\eps/\lambda$. Interestingly, the term $\mathcal{G}_2(v)$ cancels exactly the term of order $\lambda/\eps$ of $\mathcal{B}(v)$. This step shows that, thanks to some rescaling,  it is enough to prove  that for any 
$W\in L^2(0,1)$:
\begin{eqnarray*}
&&-\frac{1}{\delta}\left(\int_0^1W^2\,dy+2\int_0^1 W\,dy\right)^2+(1+\delta)\int_0^1 W^2\,dy\\
&&\qquad\qquad+\frac{2}{3}\int_0^1 W^3\,dy +\delta \int_0^1 |W|^3\,dy  -(1-\delta)\int_0^1 y(1-y)|\partial_y W|^2\,dy\leq 0.
\end{eqnarray*}
We need to show this for some  $\delta>0$ possibly very small. So it looks very similar to a nonlinear Poincar\'e inequality with constraint. The constraint (the term in $1/\delta$)  came from the term with $Y_g(v)$ through the asymptotic.
This result on $W$ is the Proposition \ref{prop:W}. 

\vskip0.3cm
\noindent{\it Step 4}: To prove Proposition \ref{prop:W} , we first reduce the problem to the minimization problem for a polynomial of two variables with a constraint.  For this we use two lemmas. 
Lemma \ref{lem-sup} provides a sharp $L^\infty$ control using the dissipation term. Lemma \ref{lem-poin} is a well known sharp Poincar\'e inequality that was already used in \cite{Kang-V-1}.
This reduces the problem to a minimization of a polynomial with variables:
$$
Z_1=\int_0^1 W(y)\,dy, \qquad Z_2=\left(\int_0^1(W-Z_1)^2\,dy\right)^{1/2}.
$$
Because of the constraint, we can reduce this minimization problem to the minimization problem of a polynomial of only one variable (see Lemma \ref{lem-alge}).  

\vskip0.1cm
It is easier to present the proofs of the propositions and lemmas in reverse. Therefore the rest of the paper is as follows. Section \ref{section_pre} is dedicated to the proofs of preliminaries. It includes some useful estimates on small shock waves, the computation of the time derivative of the  relative entropy, the construction of the function $a$, some global estimates on the relative quantities (for small or big values of $v$), and the minimization problem  for the  polynomial functional with one variable. Section \ref{section_theo} is dedicated to the proof of the main Theorem. First we give the construction of the shift, and state the main Proposition \ref{prop:main}, and then show how the Proposition implies the Theorem. To prove Proposition \ref{prop:main}, we first show the minimization problem with two variables, then the nonlinear Poincar\'e type of inequality, and continue backward up to the general situation where we have only the constraint on $Y(U)$.  

The range of $\eps$ will be reduced from one Lemma to the next, with the same notation on the restriction $\eps_0$. The restriction on $\eps/\lambda$ is more subtle. To ensure that there is no loop in the argument, we will carefully track the smallness needed on this quantity from one lemma to the next. The smallness on $\eps/\lambda$ will be denoted with $\delta$ notations. The results in the preliminaries will consider a generic smallness $\delta_*$. they can be safely replace by the same constant $\delta_*$ (taking the smallest of all). However, the constant $\delta_3$ will play a crucial role  to control the strength of the typical perturbations.  Later on, constants will be build that may blow up when $\delta_3$ is very small. It will be important to make sure that $\delta_3$ can be fixed before hand. The following restrictions on $\eps/\lambda$ are less sensitive. Therefore we will just reduced them from one lemma to the next keeping the generic notation $\delta_0$.

%%%%%%%%%%%%%%%%%%%%%%%%%%%%%%%%%%%%%%%%%%%%%
\section{Preliminaries}\label{section_pre}
\setcounter{equation}{0}
\subsection{Small shock waves}
In this subsection, we  present  useful properties of the 1-shock waves $(\tilde v_\eps,\tilde h_\eps)$ with small amplitude $\eps$. In the sequel, without loss of generality, we consider the 1-shock wave $(\tilde v_\eps,\tilde h_\eps)$ satisfying $\tilde v_\eps(0)=\frac{v_-+v_+}{2}$. Notice that the estimates in the following lemma also hold for $\tilde h_\eps$ since we have  $\tilde h_\eps'=\frac{p'(\tilde v_\eps)}{\sigma_\eps} \tilde v_\eps'$ and $C^{-1}\le\frac{p'(\tilde v_\eps)}{\sigma_\eps}\le C$. But, since the below estimates for $\tilde v_\eps$ are enough in our analysis, we give the estimates only for $\tilde v_\eps$.

\begin{lemma}
We fix $v_->0$ and $h_-\in \bbr$. Then there exists $\eps_0>0$, such that for any $0<\eps<\eps_0$ the following is true. 
Let $\tilde v_{\eps}$ be the 1-shock wave with amplitude $|p(v_-) -p(v_+)|=\eps$ and such that $\tilde v_\eps(0)=\frac{v_-+v_+}{2}$. Then, there exist constants $C, C_1, C_2>0$ such that
\beq\label{tail}
-C^{-1}\eps^2 e^{-C_1 \eps |\xi|} \le \tilde v_\eps'(\xi) \le -C\eps^2 e^{-C_2 \eps |\xi|},\quad \forall\xi\in\bbr.
\eeq
Therefore, as a consequence, we have
\beq\label{lower-v}
%\|\tilde v_{\eps}'\|_{L^{\infty}[-\frac{1}{\eps},\frac{1}{\eps}]}\ge C\eps^2.
\inf_{\left[-\frac{1}{\eps},\frac{1}{\eps}\right]}| v'_{\eps}|\ge C\eps^2.
\eeq
\end{lemma}
\begin{proof} We multiply the first equation of \eqref{small_shock1} by $\sigma_\eps$ and eliminate the dependence on $h_\eps$ using the second equation. After integration in $\xi$, we find:
% First of all, it follows from \eqref{small_shock1} that
\beq\label{s-11}
\sigma_\eps(p(\tilde v_\eps))'=\sigma_\eps^2 (\tilde v_{\eps}-v_+) + p(\tilde v_\eps)-p(v_+).
\eeq
Dividing by $\tilde v_{\eps}-v_+$ and using \eqref{RH-con} we get
\[
\frac{\sigma_\eps(p(\tilde v_\eps))'}{\tilde v_{\eps}-v_+}= -\frac{p(v_-)-p(v_+)}{v_--v_+} +\frac{p(\tilde v_\eps)-p(v_+)}{\tilde v_{\eps}-v_+}.
\]
Consider the smooth function $\varphi:\bbr^+\to\bbr$ defined by 
\[
\varphi(v):=\frac{p(v)-p(v_+)}{v-v_+}.
\]
Then, the above equality can be written as
\beq\label{temm}
\frac{\sigma_\eps(p(\tilde v_\eps))'}{\tilde v_{\eps}-v_+}= \varphi(\tilde v_\eps)-\varphi(v_-).
\eeq
To estimate the above r.h.s., we apply the Taylor theorem to the function $\varphi$ about $v_-$, so that for any $v\in\bbr^+$ with $|v-v_-|<\frac{v_-}{2}$, there exists a constant $C>0$ (depending only on $v_-$) such that
\beq\label{apply-v}
|\varphi(v)-\varphi(v_-)-\varphi'(v_-)(v-v_-)|\le C(v-v_-)^2.
\eeq
It can be shown that (see \cite{MW})
\beq\label{mono-v}
\tilde v_{\eps}'<0,\quad \mbox{and}\quad v_+<\tilde v_{\eps}< v_-.
\eeq
Therefore, for $\eps_0$ small enough:
\[
0\le v_- -\tilde v_\eps \le v_--v_+\le C\eps <\frac{v_-}{2}.
\]
Using \eqref{apply-v} with $v=\tilde v_\eps$, we have
\[
|\varphi(\tilde v_\eps)-\varphi(v_-)-\varphi'(v_-)(\tilde v_\eps -v_-)|\le C\eps(v_--\tilde v_\eps).
\]
Moreover, since 
\[
\varphi'(v_-)=\frac{p'(v_-)(v_--v_+)-\big(p(v_-)-p(v_+)\big)}{(v_--v_+)^2}=\frac{p''(v_*)}{2},\quad\mbox{for some } v_*\in (v_+,v_-),
\]
we take $\eps_0$ small enough such that $p''(v_-)\ge \varphi'(v_-)\geq p''(v_-)/2 >0$.\\
Thus, for $\eps_0$ small enough, we have
\[
2p''(v_-) (\tilde v_\eps -v_-) \le \varphi(\tilde v_\eps) -\varphi(v_-) \le \frac{p''(v_-)}{8} (\tilde v_\eps -v_-).
\]
Then, it follows from \eqref{temm} that
\[
2p''(v_-) (\tilde v_\eps -v_-)(\tilde v_\eps -v_+) \le \sigma_\eps(p(\tilde v_\eps))' \le \frac{p''(v_-)}{8} (\tilde v_\eps -v_-)(\tilde v_\eps -v_+). 
\]
Since
\beq\label{inst-0}
 -\sqrt{-p'(v_-/2)}\le \sigma_\eps \le -\sqrt{-p'(v_-)}\quad\mbox{and} \quad p'(v_-/2)\le p'(\tilde v_\eps) \le p'(v_-)<0,
\eeq
the quantity $\sigma_\eps p'(\tilde v_\eps)$ is bounded from below and above  uniformly in $\eps$.\\
Therefore
\beq\label{gein}
C^{-1}(\tilde v_\eps -v_-)(\tilde v_\eps -v_+) \le \tilde v_\eps' \le C(\tilde v_\eps -v_-)(\tilde v_\eps -v_+). 
\eeq
To prove the estimate \eqref{tail}, we first observe that $\tilde v_\eps'<0$ and $\tilde v_\eps(0)=\frac{v_-+v_+}{2}$ imply
\begin{align}
\begin{aligned}\label{bothcase}
&\xi\le0 ~\Rightarrow~v_--v_+ \ge \tilde v_{\eps}(\xi)-v_+\ge \tilde v_{\eps}(0)-v_+=\frac{v_--v_+}{2},\\
&\xi\ge0 ~\Rightarrow~v_--v_+ \ge v_--\tilde v_{\eps}(\xi)\ge v_-- \tilde v_{\eps}(0)=\frac{v_--v_+}{2}.
\end{aligned}
\end{align} 
Then, using \eqref{gein} and \eqref{bothcase} with $|v_--v_+|\leq C\eps$, we have 
\begin{align*}
\begin{aligned}
&\xi\le0 ~\Rightarrow~-C^{-1}\eps (v_- - \tilde v_\eps)\le \tilde v_{\eps}' \le -C\eps (v_- - \tilde v_\eps),\\
&\xi\ge0 ~\Rightarrow~-C^{-1}\eps (\tilde v_\eps -v_+)\le \tilde v_{\eps}' \le -C\eps (\tilde v_\eps -v_+).
\end{aligned}
\end{align*} 
Thus,
\begin{align*}
\begin{aligned}
&\xi\le0 ~\Rightarrow~-C^{-1}\eps (v_- - \tilde v_\eps)\ge (v_- -\tilde v_{\eps})' \ge -C\eps (v_- - \tilde v_\eps),\\
&\xi\ge0 ~\Rightarrow~-C^{-1}\eps (\tilde v_\eps -v_+) \le (\tilde v_{\eps}-v_+)' \le -C\eps (\tilde v_\eps -v_+).
\end{aligned}
\end{align*} 
These together with $\tilde v_\eps(0)=\frac{v_-+v_+}{2}$ imply
\begin{align*}
\begin{aligned}
&\xi\le0 ~\Rightarrow~C^{-1}\eps e^{-C_2\eps|\xi|}\le v_- -\tilde v_{\eps} \le C\eps e^{-C_1\eps|\xi|},\\
&\xi\ge0 ~\Rightarrow~C^{-1}\eps e^{-C_2\eps\xi}\le \tilde v_{\eps}-v_+ \le C\eps e^{-C_1\eps\xi}.
\end{aligned}
\end{align*} 
Finally, applying the above estimate together with $|\tilde v_{\eps}-v_\pm|\le C\eps$ to \eqref{gein}, gives
%\[
%-C^{-1}\eps^2 e^{-C_1 \eps |\xi|} \le \tilde v_\eps'(\xi) \le -C\eps^2 e^{-C_2 \eps |\xi|},\quad \forall\xi\in\bbr,
%\]
 \eqref{tail}.\\
Estimate \eqref{lower-v} follows directly  from the upper bound on $ \tilde v_\eps'(\xi)$ in \eqref{tail}.
\end{proof}
We finish this subsection with an estimate based on the inverse of the pressure function.
\begin{lemma}\label{lemma_pression}
Let us fixed $p_->0$. Then, there exists $\eps_0>0$ and $C>0$ such that for any $p_+, p>0$ such that
$0<\eps= p_+-p_-\leq \eps_0$, $p_-\leq p\leq p_+$, and $v, v_-,v_+$ such that 
$p(v)=p, p(v_\pm)=p_\pm$, we have
$$
\left|\frac{v-v_-}{p-p_-}+\frac{v-v_+}{p_+-p}+\frac{1}{2}\frac{p''(v_-)}{p'(v_-)^2}(v_--v_+)\right|\leq C\eps^2.
$$
\end{lemma}
\begin{proof}
Consider the function $v(p)=p^{-1/\gamma}$. Then, using a Taylor expansion at $p_-$, we find that there exists $\eps_0$ such that for any $|p-p_-|\leq \eps_0$ and $|p-p_+|\leq \eps_0$ we have
\begin{eqnarray}
\label{alpha}
\left| {v-v_-}-\frac{dv}{dp}(p_-)(p-p_-)-\frac{1}{2}\frac{d^2v}{dp^2}(p_-)(p-p_-)^2\right|\leq C |p-p_-|^3,\\
\label{beta}
\left| {v-v_+}-\frac{dv}{dp}(p_+)(p-p_+)-\frac{1}{2}\frac{d^2v}{dp^2}(p_+)(p-p_+)^2\right|\leq C |p-p_+|^3.
\end{eqnarray} 
Since 
$$
\frac{d^2v}{dp^2}=\frac{d}{dp}\left(\frac{1}{p'(v)}\right)=-\frac{p''(v)}{p'(v)^2}\frac{dv}{dp},
$$
we get
\begin{align}
\begin{aligned}\label{gamma}
&\left|\frac{1}{2}\frac{p''(v_-)}{p'(v_-)^2}(v_--v_+)+\frac{1}{2}\frac{d^2v}{dp^2}(p_-)(p_--p_+)\right|\\
&\qquad \le\frac{p''(v_-)}{2p'(v_-)^2}\Big|v_+ -v_- - \frac{dv}{dp}(p_-)(p_+ -p_-) \Big|\leq C\eps^2.
\end{aligned}
\end{align} 
Since
\begin{equation}\label{delta}
\begin{array}{ll}
&\displaystyle{\left| \frac{1}{2}\frac{d^2v}{dp^2}(p_+)(p-p_+) -\frac{1}{2}\frac{d^2v}{dp^2}(p_-)(p-p_-) +\frac{1}{2}\frac{d^2v}{dp^2}(p_-)(p_+-p_-)\right|} \\[0.3cm]
&\displaystyle{\qquad\qquad=\frac{1}{2} \left| \left(\frac{d^2v}{dp^2}(p_+)-\frac{d^2v}{dp^2}(p_-)\right)(p-p_+)\right|\leq C\eps^2.}
\end{array}
\end{equation}
dividing \eqref{alpha} by $p-p_-$, \eqref{beta} by $p_+-p$, and adding both terms together with the terms estimated in \eqref{gamma} and \eqref{delta}, we obtain
\begin{eqnarray*}
&&\left|\frac{v-v_-}{p-p_-}+\frac{v-v_+}{p_+-p} +\frac{1}{2}\frac{p''(v_-)}{p'(v_-)^2}(v_--v_+)\right.\\
&&\qquad \qquad\qquad \left.-\left(\frac{dv}{dp}(p_-)-\frac{dv}{dp}(p_+)-\frac{d^2v}{dp^2}(p_-)(p_--p_+)\right)\right|\leq C\eps^2.
\end{eqnarray*}
This  gives the result, since the second line term is itself of order $\eps^2$.
\end{proof}

\subsection{Relative entropy method}
Our analysis is  based on the relative entropy. The method  is purely nonlinear, and allows to handle rough and large perturbations. The relative entropy method was first introduced by Dafermos \cite{Dafermos1} and Diperna \cite{DiPerna} to prove the $L^2$ stability and uniqueness of Lipschitz solutions to the hyperbolic conservation laws endowed with a convex entropy.

To use the relative entropy method, we rewrite \eqref{NS} into the following general system of viscous conservation laws:
\beq\label{system-0}
\partial_t U +\partial_\xi A(U)= { -\partial_{\xi\xi}p(v) \choose 0},
\eeq
where 
\[
U:={v \choose h},\quad A(U):={-\sigma_\eps v-h \choose -\sigma_\eps h+p(v)}.
\]
The system \eqref{system-0} has a convex entropy $\eta(U):=\frac{h^2}{2}+Q(v)$, where $Q(v)=\frac{v^{-\gamma+1}}{\gamma-1}$, i.e., $Q'(v)=-p(v)$.\\
Using the derivative of the entropy as 
\beq\label{nablae}
\nabla\eta(U)={-p(v)\choose h},
\eeq
the above system \eqref{system-0} can be rewritten as
\beq\label{system}
\partial_t U +\partial_\xi A(U)= \partial_\xi\Big(M\partial_\xi\nabla\eta(U) \Big),
\eeq
where $M={1\quad 0 \choose 0\quad 0}$, and \eqref{small_shock1} can be rewritten as 
\beq\label{re_shock}
\partial_\xi A(\tilde U_\eps)= \partial_\xi\Big(M\partial_\xi\nabla\eta(\tilde U_\eps) \Big).
\eeq
Consider the relative entropy function defined by
\[
\eta(U|V)=\eta(U)-\eta(V) -\nabla\eta(V)\cdot (U-V),
\]
and the relative flux defined by
\[
A(U|V)=A(U)-A(V) -\nabla A(V) (U-V).
\] 
Let $G(\cdot;\cdot)$ be the flux of the relative entropy defined by
\[
G(U;V) = G(U)-G(V) -\nabla \eta(V) (A(U)-A(V)),
\]
where $G$ is the entropy flux of $\eta$, i.e., $\partial_{i}  G (U) = \sum_{k=1}^{2}\partial_{k} \eta(U) \partial_{i}  A_{k} (U),\quad 1\le i\le 2$.\\
Then, for our system \eqref{system-0}, we have
\begin{align}
\begin{aligned}\label{relative_e}
&\eta(U|\tilde U_\eps)=\frac{|h-\tilde h_\eps |^2}{2} + Q(v|\tilde v_\eps),\\
& A(U|\tilde U_\eps)={0 \choose p(v|\tilde v_\eps)},\\
&G(U;\tilde U_\eps)=(p(v)-p(\tilde v_\eps)) (h-\tilde h_\eps)-\sigma_\eps \eta(U|\tilde U_\eps),
\end{aligned}
\end{align}
Note that the relative pressure is defined as
\begin{equation}\label{pressure-relative}
p(v|w)=p(v)-p(w)-p'(w)(v-w).
\end{equation}

We consider a weighted relative entropy between the solution $U$ of \eqref{system} and the viscous shock $\tilde U_\eps:={\tilde v_\eps \choose \tilde h_\eps}$ in \eqref{small_shock1} up to a shift $X(t)$ :
\[
a(\xi)\eta\big(U(t,\xi+X(t))|\tilde U_\eps(\xi)\big),
\]
where $a$ is a smooth weight function.\\
The following Lemma provides a quadratic structure on $\frac{d}{dt}\int_{\bbr} a(\xi)\eta\big(U(t,\xi+X(t))|\tilde U_\eps(\xi)\big) d\xi$. We introduce the following notation: for any function $f : \bbr^+\times \bbr\to \bbr$ and the shift $X(t)$, 
\[
f^{\pm X}(t, \xi):=f(t,\xi\pm X(t)).
\]
We also introduce the functional space
\[
\mathcal{H}:=\{(v,h)\in\bbr^+\times\bbr~|~  v^{-1}, v,h \in L^{\infty}(\bbr),~ \partial_\xi \big(p(v)-p(\tilde v_\eps) \big)\in L^2( \bbr) \},
\]
on which the below functionals $Y, \mathcal{B}, \mathcal{G}$ in \eqref{badgood} are well-defined. \\

In this paper we assume that the solution lies in $C(0,T; \mathcal{H})$ for any $T>0$.

\begin{remark}\label{rem-sol}
The recent result of Constantin-Drivas-Nguyen-Pasqualotto \cite{CDNP} provides the global existence and uniqueness of smooth solutions to \eqref{main} for the case of $\alpha>1/2$ and periodic boundary condition. Note that the system \eqref{main} is equivalent to the one in the Eulerian coordinates for smooth solutions. More precisely, it follows from \cite[Theorem 1.5 and Remark 1.6]{CDNP} that \eqref{main} on the torus $\bbt$ admits a unique smooth solution $v, u$ such that for any $T>0$, $0<C(T)^{-1}\le v \le C(T)$, $\partial_x v\in L^\infty(\bbt)$ and $u\in C(0,T; H^k(\bbt))$ in the case of $\gamma=\alpha>1$, as long as the initial datum satisfies $v_0, u_0\in H^k$ and $\partial_x u_0\le 1$ for $k\ge4$. As a consequence, since $h=u+p'(v)\partial_x v$, this result guarantees the existence of solutions $v, h$ in $C(0,T; \mathcal{H})$ on the torus.
For an extension of their result for the case where solutions connecting two different states on the whole space as a perturbation of a shock, we leave it as a future work. 
Let us refer to the previous result \cite{Haspot} of Haspot (see also \cite{MV_sima}) for existence of solutions connecting two different states on the whole space in the case of $\alpha\le 1$. Note however that our result needs the case of $\alpha>1$.

\end{remark}

\begin{lemma}\label{lem-rel}
Let $a:\bbr\to\bbr^+$ be a smooth bounded function such that $a', a''$ are integrable. Let $X$ be a differentiable function, and $\tilde U_\eps:={\tilde v_\eps \choose \tilde h_\eps}$ be the viscous shock in \eqref{small_shock1}. For any solution $U\in\mathcal{H}$ to \eqref{system}, we have 
\begin{align}
\begin{aligned}\label{ineq-1}
\frac{d}{dt}\int_{\bbr} a(\xi)\eta(U^X(t,\xi)|\tilde U_\eps(\xi)) d\xi =\dot X(t) Y(U^X) +\mathcal{B}(U^X)- \mathcal{G}(U^X),
\end{aligned}
\end{align}
where
\begin{align}
\begin{aligned}\label{badgood}
&Y(U):= -\int_\bbr a'\eta(U|\tilde U_\eps) d\xi +\int_\bbr a\left(\partial_\xi\nabla\eta(\tilde U_\eps)\right) \cdot  (U-\tilde U_\eps) d\xi,\\
&\mathcal{B}(U):= \frac{1}{2\sigma_\eps} \int_\bbr a' |p(v)-p(\tilde v_\eps)|^2d\xi+ \sigma_\eps\int_\bbr a \partial_\xi \tilde v_\eps p(v|\tilde v_\eps) d\xi + \frac{1}{2}\int_\bbr a'' |p(v)-p(\tilde v_\eps)|^2 d\xi,\\
&\mathcal{G}(U):=\frac{\sigma_\eps}{2}\int_\bbr a'\Big(h-\tilde h_\eps -\frac{p(v)-p(\tilde v_\eps)}{\sigma_\eps}\Big)^2 d\xi +\sigma_\eps  \int_\bbr  a' Q(v|\tilde v_\eps) d\xi \\
&\qquad\qquad\qquad\qquad\qquad\qquad\qquad\qquad\qquad\qquad+\int_\bbr a |\partial_\xi (p(v)-p(\tilde v_\eps))|^2 d\xi.
\end{aligned}
\end{align}
\end{lemma}
\begin{proof}
To derive the desired structure, we  use here a change of variable $\xi\mapsto \xi-X(t)$ as
\beq\label{move-X}
\int_{\bbr} a(\xi)\eta(U^X(t,\xi)|\tilde U_\eps(\xi)) d\xi=\int_{\bbr} a^{-X}(\xi)\eta(U(t,\xi)|\tilde U_\eps^{-X}(\xi)) d\xi.
\eeq
Then, by a straightforward computation together with \cite[Lemma 4]{Vasseur_Book} and the identity $G(U;V)=G(U|V)-\nabla\eta(V)A(U|V)$, we have
\begin{align*}
\begin{aligned}
&\frac{d}{dt}\int_{\bbr} a^{-X}(\xi)\eta(U(t,\xi)|\tilde U_\eps^{-X}(\xi)) d\xi\\
&=-\dot X \int_{\bbr} \!a'^{-X} \eta(U|\tilde U_\eps^{-X} ) d\xi +\int_\bbr \!\!a^{-X}\bigg[\Big(\nabla\eta(U)-\nabla\eta(\tilde U_\eps^{-X})\Big)\!\Big(\!\!\!-\partial_\xi A(U)+ \partial_\xi\Big(M\partial_\xi\nabla\eta(U) \Big) \Big)\\
&\qquad -\nabla^2\eta(\tilde U_\eps^{-X}) (U-\tilde U_\eps^{-X}) \Big(-\dot X \partial_\xi\tilde U_\eps^{-X} -\partial_\xi A(\tilde U_\eps^{-X})+ \partial_\xi\Big(M\partial_\xi\nabla\eta(\tilde U_\eps^{-X}) \Big)\Big)  \bigg] d\xi\\
&\quad =\dot X \Big( -\int_\bbr a'^{-X}\eta(U|\tilde U_\eps^{-X}) d\xi +\int_\bbr a^{-X}\left(\partial_\xi\nabla\eta(\tilde U_\eps^{-X})\right) \cdot (U-\tilde U_\eps^{-X})  \Big) +I_1+I_2+I_3+I_4,
\end{aligned}
\end{align*}
where 
\begin{align*}
\begin{aligned}
&I_1:=-\int_\bbr a^{-X} \partial_\xi G(U;\tilde U_\eps^{-X}) d\xi,\\
&I_2:=- \int_\bbr a^{-X} \partial_\xi \nabla\eta(\tilde U_\eps^{-X}) A(U|\tilde U_\eps^{-X}) d\xi,\\
&I_3:=\int_\bbr a^{-X} \Big( \nabla\eta(U)-\nabla\eta(\tilde U_\eps^{-X})\Big) \partial_\xi \Big(M \partial_\xi (\nabla\eta(U)-\nabla\eta(\tilde U_\eps^{-X})) \Big)  d\xi \\
&I_4:=\int_\bbr a^{-X}(\nabla\eta)(U|\tilde U_\eps^{-X})\partial_\xi \Big(M \partial_\xi \nabla\eta(\tilde U_\eps^{-X}) \Big)  d\xi.
\end{aligned}
\end{align*}
Using \eqref{relative_e} and \eqref{nablae}, we have
\begin{align*}
\begin{aligned}
I_1&=\int_\bbr a'^{-X} G(U;\tilde U_\eps^{-X}) d\xi = \int_\bbr a'^{-X} \Big(\big(p(v)-p(\tilde v_\eps^{-X})\big) (h-\tilde h_\eps^{-X})  -\sigma_\eps \eta(U|\tilde U_\eps^{-X})\Big)  d\xi,\\
I_2&=-\int_\bbr a^{-X} \partial_\xi \tilde h_\eps^{-X} p(v|\tilde v_\eps^{-X}) d\xi,
\end{aligned}
\end{align*}
and 
\begin{align*}
\begin{aligned}
I_3&=\int_\bbr a^{-X} \big(p(v)-p(\tilde v_\eps^{-X})\big)\partial_{\xi\xi}\big(p(v)-p(\tilde v_\eps^{-X})\big) d\xi \\
&=-\int_\bbr a^{-X} |\partial_\xi (p(v)-p(\tilde v_\eps^{-X}))|^2 d\xi + \frac{1}{2}\int_\bbr a''^{-X} |p(v)-p(\tilde v_\eps^{-X})|^2 d\xi.
\end{aligned}
\end{align*}
Since it follows from \eqref{re_shock} and \eqref{nablae} that
\begin{align*}
\begin{aligned}
I_4=\int_\bbr a^{-X}(\nabla\eta)(U|\tilde U_\eps^{-X})\partial_\xi A(\tilde U_\eps^{-X})  d\xi =\int_\bbr a^{-X} p(v|\tilde v_\eps^{-X}) \Big(\partial_\xi \tilde h_\eps^{-X} + \sigma_\eps \partial_\xi \tilde v_\eps^{-X} \Big) d\xi,
\end{aligned}
\end{align*}
we have some cancellation
\[
I_2+I_4=\sigma_\eps\int_\bbr a^{-X} \partial_\xi \tilde v_\eps^{-X} p(v|\tilde v_\eps^{-X}) d\xi.
\]
Therefore, we have
\begin{align*}
\begin{aligned}
&\frac{d}{dt}\int_{\bbr} a^{-X}\eta(U|\tilde U_\eps^{-X}) d\xi\\
&\quad =\dot X \Big( -\int_\bbr a'^{-X}\eta(U|\tilde U_\eps^{-X}) d\xi +\int_\bbr a^{-X}\partial_\xi\nabla\eta(\tilde U_\eps^{-X}) (U-\tilde U_\eps^{-X}) d\xi  \Big)\\
&\qquad + \int_\bbr a'^{-X} \Big(\big(p(v)-p(\tilde v_\eps^{-X})\big) (h-\tilde h_\eps^{-X})  -\sigma_\eps \eta(U|\tilde U_\eps^{-X})\Big)  d\xi \\
&\qquad + \sigma_\eps\int_\bbr a^{-X} \partial_\xi \tilde v_\eps^{-X} p(v|\tilde v_\eps^{-X}) d\xi + \frac{1}{2}\int_\bbr a''^{-X} |p(v)-p(\tilde v_\eps^{-X})|^2 d\xi\\
&\qquad\qquad\qquad-\int_\bbr a^{-X} |\partial_\xi (p(v)-p(\tilde v_\eps^{-X}))|^2 d\xi.
\end{aligned}
\end{align*}
Again, we use a change of variable $\xi\mapsto \xi+X(t)$ to have
\begin{align*}
\begin{aligned}
&\frac{d}{dt}\int_{\bbr} a\eta(U^X|\tilde U_\eps) d\xi\\
&\quad =\dot X \Big( -\int_\bbr a'\eta(U^X|\tilde U_\eps) d\xi +\int_\bbr a\partial_\xi\nabla\eta(\tilde U_\eps) (U^X-\tilde U_\eps) d\xi  \Big)\\
&\qquad + \int_\bbr a' \Big( \underbrace{\big(p(v^X)-p(\tilde v_\eps)\big) (h^X-\tilde h_\eps)  -\sigma_\eps \eta(U^X|\tilde U_\eps)}_{=:I}\Big)  d\xi \\
&\qquad + \sigma_\eps\int_\bbr a \partial_\xi \tilde v_\eps p(v^X|\tilde v_\eps) d\xi + \frac{1}{2}\int_\bbr a'' |p(v^X)-p(\tilde v_\eps)|^2 d\xi
-\int_\bbr a |\partial_\xi (p(v^X)-p(\tilde v_\eps))|^2 d\xi.
\end{aligned}
\end{align*}

To extract a quadratic term on $p(v^X)-p(\tilde v_\eps)$ from the above hyperbolic part, we rewrite $I$ as
\begin{align*}
\begin{aligned}
I&=(p(v^X)-p(\tilde v_\eps)) (h^X-\tilde h_\eps)-\sigma_\eps\frac{|h^X-\tilde h_\eps|^2}{2} -\sigma_\eps Q(v^X|\tilde v_\eps)\\
&=\frac{|p(v^X)-p(\tilde v_\eps)|^2}{2\sigma_\eps} -\frac{\sigma_\eps}{2} \Big(h^X-\tilde h_\eps -\frac{p(v^X)-p(\tilde v_\eps)}{\sigma_\eps}\Big)^2 -\sigma_\eps Q(v^X|\tilde v_\eps).
\end{aligned}
\end{align*}
Hence we have the desired representation \eqref{ineq-1}-\eqref{badgood}.
\end{proof}

\begin{remark}\label{rem:0}
Notice that since $\sigma_\eps, a' <0$, the three terms in  $\mathcal{G}$ are non-negative. Therefore, $\mathcal{G}$ consists of good terms, while $\mathcal{B}$ consists of bad terms. 
\end{remark}

\subsection{Construction of the weight function}
We define the weight function $a$ by
\beq\label{weight-a}
a(\xi)=1-\lambda \frac{p(\tilde v_\eps(\xi))-p(v_-)}{[p]},
\eeq
where $[p]:=p(v_+)-p(v_-)$.%and the coefficient $\lambda$ is sufficiently small, but far bigger than $\eps$, that is chosen as
%\beq\label{lambda}
%\lambda=\eps^{\frac{3}{8}}.
%\eeq
%The choice \eqref{lambda} is not optimal, but crucially used in our analysis.\\
We briefly present some useful properties on the weight $a$.\\
First of all, the weight function $a$ is positive and decreasing, and satisfies $1-\lambda\le a\le 1$.\\
Since $[p]=\eps$, $p'(v_-/2)\le p'(\tilde v_\eps)\le p'(v_-)$ and
\beq\label{der-a}
a'=-\lambda \frac{\partial_\xi p(\tilde v_\eps)}{[p]},
\eeq
we have 
\beq\label{d-weight}
|a'|\sim \frac{\lambda}{\eps}|\tilde v_\eps'|.
\eeq
For $a''=-\lambda \frac{\partial_{\xi\xi} p(\tilde v_\eps)}{[p]}$, we use the following relation from \eqref{small_shock1}:
\beq\label{dd-a}
\partial_{\xi\xi} p(\tilde v_\eps)=\sigma_\eps \partial_{\xi} \tilde v_\eps +\partial_{\xi}\tilde h_\eps =\Big(\frac{\sigma_\eps^2}{p'(\tilde v_\eps)}+1\Big)\frac{\partial_\xi p(\tilde v_\eps)}{\sigma_\eps}.
\eeq
Notice that $|v_--v_+|=C\eps$ and \eqref{RH-con} together with the Taylor theorem imply
\beq\label{shock_speed}
\sigma_\eps= -\sqrt{-p'(v_-)}+\mathcal{O}(\eps).
\eeq
Moreover, since $p'(\tilde v_\eps)^{-1}= p'(v_-)^{-1} +\mathcal{O}(\eps)$, we have 
\beq\label{ddp}
|\partial_{\xi\xi} p(\tilde v_\eps)|\le C\eps |\partial_\xi p(\tilde v_\eps)|.
\eeq
Thus,
\[
|a''|\lesssim \lambda |\tilde v_\eps'|.
\]
which together with \eqref{d-weight} implies 
\beq\label{dd-weight}
|a''|\lesssim \eps |a'|.
\eeq

\begin{remark}
The definition \eqref{weight-a} can be more generally written by
\beq\label{weight}
a(\xi)=1-\lambda \frac{\int_{-\infty}^{\xi} |\partial_s\nabla\eta(\tilde U_\eps(s))|ds}{\int_{-\infty}^{\infty} |\partial_s\nabla\eta(\tilde U_\eps(s))|ds}.
\eeq
Indeed, since it follows from \eqref{small_shock1} that $p( \tilde v_\eps)'=\sigma_\eps \tilde h_{\eps}'$, we find that
\[
 |\partial_\xi\nabla\eta(\tilde U_\eps(\xi))|=\Big|\partial_\xi{-p(\tilde v_\eps(\xi))\choose \tilde h_\eps(\xi)}\Big|=|\partial_\xi p(\tilde v_\eps(\xi))| |(-1, \sigma_\eps^{-1}) |.
\]
Moreover, since $\partial_\xi p(\tilde v_\eps(\xi))>0$,
\beq\label{w-real}
 |\partial_\xi\nabla\eta(\tilde U_\eps(\xi))|=\partial_\xi p(\tilde v_\eps(\xi)) |(-1, \sigma_\eps^{-1}) |,
\eeq
which implies \eqref{weight-a}.
\end{remark}

\subsection{Global and local estimates on the relative quantities}
We here present useful inequalities on $Q$ and $p$ that are crucially for the proof of Theorem \ref{thm_main}. 
\subsubsection{Global inequalities on $Q$ and $p$}
Lemma \ref{lem-pro} provides some global inequalities on the relative function $Q(\cdot|\cdot)$ corresponding to the convex function $Q(v)=\frac{v^{-\gamma+1}}{\gamma-1}$, $v>0$, $\gamma>1$. 
\begin{lemma}\label{lem-pro}
For given constants $\gamma>1$, and $v_->0$ There exists constants $c_1, c_2>0$ such that  the following inequalities hold.\\
1)  For any $w\in (0,v_-)$,
\begin{align}
\begin{aligned}\label{rel_Q}
& Q(v|w)\ge c_1 |v-w|^2,\quad \mbox{for all } 0<v\le 3v_-,\\
 & Q(v|w)\ge  c_2 |v-w|,\quad  \mbox{for all } v\ge 3v_-.
\end{aligned}
\end{align}
%where $c_2$ can be particularly chosen as 
%\beq\label{c_2}
%c_2= (v^*-v_-)\int_0^1\int_0^1 Q''(v_- + st(v^*-v_-)) t dsdt.
%\eeq
2) Moreover if $0<w\leq u\leq v$ or $0<v\leq u\leq w$ then 
\beq\label{Q-sim}
Q(v|w)\geq Q(u|w),
\eeq
and for any $\delta_*>0$ there exists a constant $C>0$ such that if, in addition, 
$v_->w>v_--\delta_*/2$ and $|w-u|>\delta_*$, we have
\beq\label{rel_Q1}
Q(v|w)-Q(u|w)\geq C|u-v|. 
\eeq
%2) For a given $v_*$ satisfying $0<v_*<v_+$, there exists a positive constant $c_3$ depending on $v_*$, $v_+$, $\gamma$ such that
%$\forall w\in (v_+,v_-)$,
%\beq\label{rel_Q1}
%Q(v|w)\ge c_3 |v-w|,\quad \mbox{for all } 0<v\le v_*,
%\eeq
%where $c_3$ can be particularly chosen as 
%\beq\label{c_3}
%c_3= (v_+-v_*)\int_0^1\int_0^1 Q''(v_- + st(v_*-v_+)) t dsdt.
%\eeq
%3) For a given $v_*$ satisfying $v_*<v_+$, there exists a constant $c$ depending on $\gamma, v_*$ such that
%$\forall w\in (v_+,v_-)$,
%\beq\label{pQ}
%|v^{-\gamma+1}-w^{-\gamma+1}|\le cQ(v|w),\quad \forall v\le v_*.
%\eeq
\end{lemma}
\begin{proof}
$\bullet$ {\it proof of \eqref{rel_Q}} : We denote $v^*=3v_-$. First, for the case of $v\ge v^*$, we rewrite $Q(v|w)$ as
\beq\label{qvw0}
Q(v|w)=\int_0^1 \Big(Q'(w + t(v-w))  -Q'(w)\Big) dt (v-w)
\eeq
Since $w<v_-<v^*\le v$ and $Q'$ is increasing, we have
\[
Q'(w + t(v-w))\ge Q'(w+ t(v^*-v_-)).
\]
Thus,
\begin{align*}
\begin{aligned}
Q(v|w) &\ge  \int_0^1\Big(Q'(w+ t(v^*-v_-))  -Q'(w)\Big) dt (v-w)\\
&=\int_0^1\int_0^1 Q''(w + st(v^*-v_-)) t dsdt (v^*-v_-)(v-w).
\end{aligned}
\end{align*}
Moreover, since $Q''$ is decreasing, we have
\[
Q(v|w) \ge \int_0^1\int_0^1 Q''(v_- + st(v^*-v_-)) t dsdt (v^*-v_-)(v-w),
\]
which provides the second inequality in \eqref{rel_Q}.

On the other hand, for the case of $v\le v^*$, we use 
\[
Q(v|w)=(v-w)^2\int_0^1\int_0^1 Q''(w + st(v-w)) t dsdt.
\]
Observe that for all $v\le v^*$,
\[
Q''(w+ st(v-w))=\gamma (st v+(1-st)w )^{-\gamma-1} \ge \gamma (st v^*+(1-st)v^*)^{-\gamma-1}= \gamma \Big(\frac{1}{v^*}\Big)^{\gamma+1},
\]
where we have used $w<v^*$.\\
Therefore, we have
\[
Q(v|w)\ge \frac{\gamma}{2} \Big(\frac{1}{v^*}\Big)^{\gamma+1}(v-w)^2.
\]

$\bullet$ {\it proof of \eqref{rel_Q1}} : 
Note that $z\mapsto Q(z|y)$ is convex so $\partial_z Q(z|y)$ is increasing in $z$ and zero at $z=y$. Therefore $z\mapsto Q(z|y)$ is increasing in $|z-y|$, which implies 
$$
Q(v|w)\geq Q(u|w).
$$
Moreover, if $v_->w>v_--\delta_*/2$ and $|w-u|>\delta_*$, using the facts that $Q'$ is increasing and
\begin{eqnarray*}
Q(v|w)-Q(u|w)&=&Q(v)-Q(u)-Q'(w)(v-u)\\
&=&\int_u^v [Q'(y)-Q'(w)]\,dy,
\end{eqnarray*}
we have the following:\\
If $w<u<v$, then 
\begin{eqnarray*}
Q(v|w)-Q(u|w)\geq [Q'(v_-+\delta_*/2)-Q'(v_-)](v-u).
\end{eqnarray*}
If $v<u<w$, then 
\begin{eqnarray*}
Q(v|w)-Q(u|w)&\geq& [Q'(v_--\delta_*/2)-Q'(v_--\delta_*)](u-v).
\end{eqnarray*}
Hence we have \eqref{rel_Q1}.
\end{proof}
%
%
%Since $w>v_+>v_*\ge v$, thus
%\[
%Q'(w + t(v-w))\le Q'(w+ t(v_*-v_+)),
%\]
%it follows from \eqref{qvw0} that for all $v\le v_*$, 
%\begin{align*}
%\begin{aligned}
%Q(v|w) &\ge  \int_0^1\Big(Q'(w + t(v_*-v_+))  -Q'(w)\Big) dt (v-w)\\
%&=\int_0^1\int_0^1 Q''(w+ st(v_*-v_+)) t dsdt (v_*-v_+)(v-w).
%\end{aligned}
%\end{align*}
%Then, by $w\le v_-$, we have
%\begin{align*}
%\begin{aligned}
%Q(v|w) &\ge \int_0^1\int_0^1 Q''(v_- + st(v_*-v_+)) t dsdt (v_*-v_+)(v-w)\\
%& =\int_0^1\int_0^1 Q''(v_- + st(v_*-v_+)) t dsdt (v_+-v_*)|v-w|.
%\end{aligned}
%\end{align*}

%$\bullet$ {\it proof of \eqref{pQ}} :
%Consider two functions $f, g$ defined by 
%\[
%f(v)=v^{-\gamma+1}-w^{-\gamma+1},\quad g(v)=\frac{v^{-\gamma+1}}{\gamma-1}-\frac{w^{-\gamma+1}}{\gamma-1}
%+w^{-\gamma} (v-w).
%\]
%Since 
%\[
%f'(v)=(-\gamma+1) v^{-\gamma},\quad g'(v)=-\gamma v^{-\gamma} + \gamma w^{-\gamma},
%\]
%putting $c:=\gamma \Big(1- v_+^{-\gamma} v_*^{\gamma}\Big)^{-1}$, for all $v\le v_*< v_+\le w$,
%\[
%f'(v)-cg'(v)= v^{-\gamma} \Big[-\gamma+1 +c\Big(1-w^{-\gamma} v^{\gamma} \Big) \Big]
%\le v^{-\gamma} \Big[-\gamma+1 +c \Big(1- v_+^{-\gamma} v_*^{\gamma}\Big) \Big]<0.
%\]
%Therefore, $f(v)\le cg(v),~\forall v\le v_*$, which concludes \eqref{pQ}.\\

%\subsubsection{Global inequalities on $p$}
The following lemma provides some global inequalities on the pressure $p(v)=v^{-\gamma}$, $v>0$, $\gamma>1$,  and on the associated relative function $p(\cdot|\cdot)$.

\begin{lemma}\label{lem-pro2}
For given constants $\gamma>1$, and $v_->0$, there exist constants $c_3, C>0$ such that  the following inequalities hold.\\
1)For any $w\in (v_-/4, v_-)$
\beq\label{pressure2}
|p(v)-p(w)| \le c_3 |v-w|,\quad \mbox{for all } v\ge v_-/2,
\eeq
\beq\label{pressure0}
p(v|w) \le C |v-w|^2,\quad \mbox{for all } v\ge v_-/2.
\eeq
2) For any $w\in (v_-/4, v_-)$, and all $v> 0$ 
\beq\label{pressure4}
p(v|w)\leq C(|v-w|+|p(v)-p(w)|).
\eeq
\end{lemma}
\begin{proof}
$\bullet$ {\it proof of \eqref{pressure2}} : Since $v,w \ge v_-/2$, \eqref{pressure2} follows from using the mean value theorem:
\[
|p(v)-p(w)| \le |p'(v_-/2)| |v-w|.
\]

$\bullet$ {\it proof of \eqref{pressure0}} : Since $v,w \ge v_-/2$, \eqref{pressure2} follows from
\begin{align*}
\begin{aligned}
 p(v|w)&=(v-w)^2\int_0^1\int_0^1 p''(stv + (1-st)w)) t dsdt\\
&\le (v-w)^2\int_0^1\int_0^1 p''(v_-/2) t dsdt =\frac{p''(v_-/2)}{2}(v-w)^2.
\end{aligned}
\end{align*}

$\bullet$ {\it proof of \eqref{pressure4}} :
For every $v>v_-/2$
$$
0\leq p(v|w)=p(v)-p(w)-p'(w)(v-w)\leq 2|p'(v_-/2)| |v-w|.
$$
And for every $v\leq v_-/2$:
$$
|p(v)-p(w)|=\int_v^w|p'(y)|\,dy\geq |p'(w)| |v-w|\geq |p'(v_-)| |v-w|.
$$
Hence
$$
0\leq p(v|w)=p(v)-p(w)-p'(w)(v-w)\leq (1+|p'(v_-/2)|^{-1})|p(v)-p(w)|.
$$
\end{proof}

\subsubsection{Local inequalities on $Q$ and $p$}
We  present now  some local estimates on $p(v|w)$ and $Q(v|w)$ for $|v- w|\ll 1$, based on  Taylor expansions. The specific  coefficients of the estimates will be crucially used in our local analysis on a suitably small truncation $|p(v)-p(\tilde v_\eps)|\ll1$.  
\begin{lemma}\label{lem:local}
For given constants $\gamma>1$ and $v_->0$ %with $0<v_+<v_-$, 
there exist positive constants $C$ and $\delta_*$ such that for any $0<\delta<\delta_*$, the following is true.\\
1) For any $(v, w)\in \bbr_+^2$ %and $w\in (v_+,v_-)$ 
satisfying $|p(v)-p(w)|<\delta$, and  $|p(w)-p(v_-)|<\delta$ the following estimates \eqref{p-est1}-\eqref{Q-est1} hold:
\begin{align}
\begin{aligned}\label{p-est1}
p(v|w)&\le \bigg(\frac{\gamma+1}{2\gamma} \frac{1}{p(w)} + C\delta \bigg) |p(v)-p(w)|^2,
\end{aligned}
\end{align}
\beq\label{Q-est11}
Q(v|w)\ge \frac{p(w)^{-\frac{1}{\gamma}-1}}{2\gamma}|p(v)-p(w)|^2 -\frac{1+\gamma}{3\gamma^2} p(w)^{-\frac{1}{\gamma}-2}(p(v)-
p(w))^3,
\eeq
\beq\label{Q-est1}
Q(v|w)\le \bigg( \frac{p(w)^{-\frac{1}{\gamma}-1}}{2\gamma} +C\delta  \bigg)|p(v)-p(w)|^2.
\eeq
2) For any $(v, w)\in \bbr_+^2$ such that  $|p(w)-p(v_-)|\leq \delta$,  and satisfying either $Q(v|w)<\delta$ or $|p(v)-p(w)|<\delta$,
\beq\label{pQ-equi0}
|p(v)-p(w)|^2 \le C Q(v|w).
\eeq
\end{lemma}
\begin{proof}
We consider $\delta_*\leq p(v_-)/4$.

$\bullet$ {\bf proof of \eqref{p-est1}}
From the hypothesis, we have both $|p(v)-p(v_-)|\leq  p(v_-)/2$ and $|p(w)-p(v_-)|\leq  p(v_-)/2$.
First, we rewrite $p(v|w)$ in terms of $p(v)-p(w)$ as
\begin{align*}
\begin{aligned}
p(v|w)&=p(v)-p(w)+\gamma w^{-\gamma-1} (v-w)\\
&=p(v)-p(w)+\gamma p(w)^{\frac{\gamma+1}{\gamma}} \big(p(v)^{-\frac{1}{\gamma}}-p(w)^{-\frac{1}{\gamma}}\big).
\end{aligned}
\end{align*}
Setting $F_1(p):=p-\tilde p+\gamma\tilde p^{\frac{\gamma+1}{\gamma}} (p^{-\frac{1}{\gamma}}-\tilde p^{-\frac{1}{\gamma}})$ where $p:=p(v)$, $\tilde p:=p(w)$, we apply the Taylor theorem to $F_1$ about $\tilde p$. That is, using
\[
F_1'(p)=1-\tilde p^{\frac{\gamma+1}{\gamma}} p^{-\frac{\gamma+1}{\gamma}},\quad F_1''(p)=\frac{\gamma+1}{\gamma}
\tilde p^{\frac{\gamma+1}{\gamma}} p^{-\frac{2\gamma+1}{\gamma}},
\]
since $F_1(\tilde p)=0, F_1'(\tilde p)=0,$ and $F_1''(\tilde p)=\frac{\gamma+1}{\gamma \tilde p}$, we have
\begin{align*}
\begin{aligned}
p(v|w)&=F_1(p)= \frac{\gamma+1}{\gamma \tilde p}\frac{|p-\tilde p|^2}{2} +\frac{F_1'''(p_*)}{6} |p-\tilde p|^3, 
\end{aligned}
\end{align*}
where $p_*$ lies between $p$ and $\tilde p$. %Thus, for any $p, \tilde p$ satisfying $|p-\tilde p|<\delta$ and $p(v_-)<\tilde p<p(v_+)$, since 
Therefore $\frac{p(v_-)}{2}<p_*<2p(v_-)$.  Taking  $\delta\leq \delta_*$, we have 
\[
p(v|w) \le \frac{\gamma+1}{\gamma \tilde p}\frac{|p-\tilde p|^2}{2} + C\delta |p-\tilde p|^2.
\]
Therefore, we have \eqref{p-est1}.

$\bullet$ {\bf proof of \eqref{Q-est11} and \eqref{Q-est1}}
Likewise, since
\begin{align*}
\begin{aligned}
Q(v|w)&=Q(v)-Q(w)+p(w) (v-w)\\
&=\frac{p(v)^{\frac{\gamma-1}{\gamma}}}{\gamma-1}-\frac{p(w)^{\frac{\gamma-1}{\gamma}}}{\gamma-1}+p(w) (p(v)^{-\frac{1}{\gamma}}-p(w)^{-\frac{1}{\gamma}}).
\end{aligned}
\end{align*}
setting $F_2(p):=\frac{p^{\frac{\gamma-1}{\gamma}}}{\gamma-1}-\frac{\tilde p^{\frac{\gamma-1}{\gamma}}}{\gamma-1}+\tilde p (p^{-\frac{1}{\gamma}}-\tilde p^{-\frac{1}{\gamma}})$ where $p:=p(v)$, $\tilde p:=p(w)$, we apply the Taylor theorem to $F_2$ about $\tilde p$. That is,
 using
\begin{align*}
\begin{aligned}
&F_2'(p)=\frac{1}{\gamma} p^{-\frac{1}{\gamma}}\Big(1- \tilde pp^{-1}\Big),\quad F_2''(p)=-\frac{1}{\gamma^2} p^{-\frac{1+\gamma}{\gamma}}\Big(1- (1+\gamma)\tilde pp^{-1}\Big),\\
&F_2'''(p)=\frac{1+\gamma}{\gamma^3} p^{-\frac{1+2\gamma}{\gamma}}\Big(1- (1+2\gamma)\tilde pp^{-1}\Big),\\
& F_2''''(p)=-\frac{(1+\gamma)(1+2\gamma)}{\gamma^4} p^{-\frac{1+3\gamma}{\gamma}}\Big(1- (1+3\gamma)\tilde pp^{-1}\Big),
\end{aligned}
\end{align*}
and then
\begin{align*}
\begin{aligned}
&F_2(\tilde p)=0,\quad F_2'(\tilde p)=0,\quad F_2''(\tilde p)=\frac{1}{\gamma} \tilde p^{-\frac{1}{\gamma}-1},\\
&F_2'''(\tilde p)=-\frac{2(1+\gamma)}{\gamma^2} \tilde p^{-\frac{1}{\gamma}-2},\quad F_2''''(\tilde p)=\frac{3(1+\gamma)(1+2\gamma)}{\gamma^3}\tilde p^{-\frac{1+3\gamma}{\gamma}},
\end{aligned}
\end{align*}
we have
\begin{align*}
\begin{aligned}
Q(v|w)&=F_2''(\tilde p) \frac{(p-\tilde p)^2}{2}+F_2'''(\tilde p) \frac{(p-\tilde p)^3}{6} +F_2''''(\tilde p)\frac{(p-\tilde p)^4}{24} +F_2^{(5)}(p_*)\frac{(p-\tilde p)^5}{5!}. 
\end{aligned}
\end{align*} 
Since $F_2''''(\tilde p)\ge \frac{3(1+\gamma)(1+2\gamma)}{\gamma^3} [p(v-)/2]^{-\frac{1+3\gamma}{\gamma}}>0$, taking $\delta_*$ smaller if needed,  we have for every $\delta<\delta_*$
\[
Q(v|w)\ge F_2''(\tilde p) \frac{|p-\tilde p|^2}{2}+F_2'''(\tilde p) \frac{(p-\tilde p)^3}{6},
\]
which completes \eqref{Q-est11}.
The estimate \eqref{Q-est1} follows by considering the 2nd order Taylor polynomial as done in \eqref{p-est1}.

$\bullet$ {\bf proof of \eqref{pQ-equi0}}
Since it follows from \eqref{rel_Q} that $\min\{c_1|v-w|^2,c_2|v-w|\}\le Q(v|w)$, if $Q(v|w)<\delta<\delta_*\ll1$, then $|v-w|\ll1$ and thus $\frac{v_-}{2}<v<2v_-$ and $c_1|v-w|^2\le Q(v|w)$.\\
Therefore, 
\beq\label{pqtem}
|p(v)-p(w)|^2\le |p'(\frac{v_-}{2})|^2 |v-w|^2 \le c_1^{-1}|p'(\frac{v_-}{2})|^2 Q(v|w).
\eeq
If $|p(v)-p(w)|<\delta$, then it follows from \eqref{Q-est1} that 
\[
Q(v|w) \le C|p(v)-p(w)|^2 <\delta,
\]
which gives \eqref{pqtem}.
\end{proof}

\subsection{Some functional inequalities}

We state in this section some standard functional inequalities. Some of the proofs will be postponed to the appendix.
The first result is a simple inequality on a specific polynomial functional.
\begin{lemma}\label{lem-alge}
For all $x\in [-2,0)$,
\[
2x-2x^2-\frac{4}{3}x^3 +\frac{4\theta}{3} \Big(-x^2 -2x \Big)^{3/2}< 0,
\]
where $\theta =\sqrt{5-\frac{\pi^2}{3}}$.
\end{lemma}
The proof of Lemma \ref{lem-alge} is given in Appendix \ref{app-1}.

The second result is a sharp point-wise estimate.

\begin{lemma}\label{lem-sup}
Let $f\in C^1(0,1)$. Then, for all $x\in[0,1)$,
\[
\Big|f(x)-\int_0^1 f(y) dy\Big| \le \sqrt{L(x)+L(1-x)} \sqrt{\int_0^1 y(1-y)|f'|^2 dy},
\]
where $L(x):=-x-\ln(1-x)$.
Moreover 
\[
\Big(\int_0^1(L(y)+L(1-y))^{2}dy\Big)^{1/2} =\sqrt{5-\frac{\pi^2}{3}}=\theta.
\]
\end{lemma}
\begin{proof}
First, since
\begin{align*}
\begin{aligned}
f(x)-\int_0^1 f(y) dy = \int_0^1\int_y^x f'(z)dzdy =\int_0^x\int_y^x f'(z)dzdy + \int_x^1\int_y^x f'(z)dzdy,
\end{aligned}
\end{align*}
we have
\begin{align*}
\begin{aligned}
\Big|f(x)-\int_0^1 f(y) dy \Big| &\le \int_0^x\int_y^x |f'(z)|dzdy + \int_x^1\int_x^y |f'(z)|dzdy\\
&\le \underbrace{\bigg(\int_0^x\int_y^x \frac{1}{1-z}dzdy \bigg)^{\frac{1}{2}}\bigg(\int_0^x\int_y^x (1-z)|f'(z)|^2dzdy \bigg)^{\frac{1}{2}}}_{=:I_1}\\
&\quad+ \underbrace{\bigg(\int_x^1\int_x^y \frac{1}{z}dzdy \bigg)^{\frac{1}{2}}\bigg(\int_x^1\int_x^y z|f'(z)|^2dzdy \bigg)^{\frac{1}{2}}}_{=:I_2}.
\end{aligned}
\end{align*}
Using Fubini's theorem as $\int_0^x\int_y^x g dzdy=\int_0^x\int_0^z gdydz$, we have
\begin{align*}
\begin{aligned}
I_1&=\bigg(\int_0^x \frac{z}{1-z}dz \bigg)^{\frac{1}{2}}\bigg(\int_0^x z(1-z)|f'(z)|^2dz \bigg)^{\frac{1}{2}}\\
&=\big(-x-\ln(1-x)\big)^{\frac{1}{2}}\bigg(\int_0^x z(1-z)|f'(z)|^2dz \bigg)^{\frac{1}{2}},
\end{aligned}
\end{align*}
and likewise,
\begin{align*}
\begin{aligned}
I_2&=\bigg(\int_x^1 \frac{1-z}{z}dz \bigg)^{\frac{1}{2}}\bigg(\int_x^1 z(1-z)|f'(z)|^2dz \bigg)^{\frac{1}{2}}\\
&=\big(-(1-x)-\ln x\big)^{\frac{1}{2}}\bigg(\int_x^1 z(1-z)|f'(z)|^2dz \bigg)^{\frac{1}{2}}.
\end{aligned}
\end{align*}
Let $L(x):=-x-\ln(1-x)$ and 
\[
X:=\int_0^x z(1-z)|f'(z)|^2dz,\quad D:=\int_0^1 z(1-z)|f'(z)|^2dz.
\]
Then,
\[
I_1+I_2=\sqrt{L(x)}\sqrt{X}+\sqrt{L(1-x)}\sqrt{D-X}
\]
If we consider a function $F(X):=\sqrt{L(x)}\sqrt{X}+\sqrt{L(1-x)}\sqrt{D-X}$ for $X\in[0,D]$, we see that the function $F$ has a maximum at $\bar X:=\frac{L(x)D}{L(x)+L(1-x)}$. \\
Thus, we have
\[
I_1+I_2\le F(\bar X)=\sqrt{L(x)+L(1-x)}\sqrt{D},
\]
which completes the desired inequality.
We now  compute the value of $\theta$.          % \eqref{value0}. 
 We have
\begin{align*}
\begin{aligned}
\int_0^1\big(L(x)+L(1-x)\big)^{2}dx &= \int_0^1\big(1+\ln(1-x)+\ln x \big)^{2}dx\\
&=1+\int_0^1 \big( \ln(1-x) \big)^2 \, dx +\int_0^1 \big( \ln x \big)^2 \, dx + 2\int_0^1  \ln (1-x) \,dx\\
&\quad +2\int_0^1  \ln x \, dx +2\int_0^1  \ln (1-x) \ln x \, dx.
\end{aligned}
\end{align*}
Since $\int_0^1  \ln (1-x) \, dx=\int_0^1  \ln x \,dx=-1$, we have
\[
\int_0^1  \big( \ln(1-x) \big)^2 dx=  \int_0^1  \big( \ln x \big)^2 dx = \Big[ x\big( \ln x \big)^2 \Big]_0^1-2 \int_0^1  \ln x \, dx = -2\int_0^1  \ln x \, dx=2.
\]
Thus, 
\[
\int_0^1\big(L(x)+L(1-x)\big)^{2}dx= 1+2\int_0^1  \ln (1-x) \ln x \, dx.
\]
To compute the last integral, we find
\begin{align*}
\begin{aligned}
\int_0^1  \ln (1-x) \ln x \, dx &= \Big[\Big(x\ln(1-x)-x-\ln(1-x) \Big) \ln x \Big]_0^1 \\
&\qquad\qquad\qquad-\int_0^1 \frac{x\ln(1-x)-x-\ln(1-x)}{x} dx\\
&=-\int_0^1\ln(1-x)dx +1 +\int_0^1 \frac{\ln(1-x)}{x} \, dx=2+\int_0^1 \frac{\ln(1-x)}{x} \, dx.
\end{aligned}
\end{align*}
Since 
\[
\int\frac{\ln(1-x)}{x}\, dx=-\sum_{n=1}^{\infty}\frac{x^{n}}{n^2},\quad |x|\le1,
\]
we have
\[
\int_0^1 \frac{\ln(1-x)}{x} dx =-\sum_{n=1}^{\infty}\frac{1}{n^2}=-\frac{\pi^2}{6}.
\]
This gives the result.
\end{proof}

\begin{lemma}\label{lem-poin}
For any $f:[0,1]\to\bbr$ satisfying $\int_0^1 y(1-y)|f'|^2 dy<\infty$, 
\beq\label{poincare}
\int_0^1\Big|f-\int_0^1 f dy \Big|^2 dy\le \frac{1}{2}\int_0^1 y(1-y)|f'|^2 dy.
\eeq
\end{lemma}
The proof of this lemma is given in Appendix \ref{app-2}.

%%%%%%%%%%%%%%%%%%%%%%%%%%%%%%%%%%%%%%%%%%%%%%%%

\vspace{1cm}

\section{Proof of Theorem \ref{thm_main}}\label{section_theo}
\setcounter{equation}{0}

\subsection{Construction of the shift $X$ and the main proposition}

For any fixed $\eps>0$, we consider a continuous function $\Phi_\eps$ defined by
\beq\label{Phi-d}
\Phi_\eps (y)=
\left\{ \begin{array}{ll}
      \frac{1}{\eps^2},\quad \mbox{if}~ y\le -\eps^2, \\
      -\frac{1}{\eps^4}y,\quad \mbox{if} ~ |y|\le \eps^2, \\
       -\frac{1}{\eps^2},\quad \mbox{if}  ~y\ge \eps^2. \end{array} \right.
\eeq
We define a shift function $X(t)$ as a solution of the nonlinear ODE:
\beq\label{X-def}
\left\{ \begin{array}{ll}
       \dot X(t) = \Phi_\eps (Y(U^X)) \Big(2|\mathcal{B}(U^X)|+1 \Big),\\
       X(0)=0, \end{array} \right.
\eeq
where $Y$ and $\mathcal{B}$ are as in \eqref{badgood}.
%Then, the shift $X(t)$ is Lipschitz for all time $t$. Indeed, we first have a $L^{\infty}$-bound of $|H_1(t)+P_1(t)+P_2(t)|$ up to a short time $t_0$ by a rough energy estimate, thus $X(t)$ is Lipschitz for $t\le t_0$. Then, the contraction estimate \eqref{cont_main} holds for $t\le t_0$ thanks to \eqref{X-def}. Therefore, the shift $X$ is Lipschitz for all time by the continuation argument.
Therefore, for the solution $U\in C(0,T;\mathcal{H})$, the shift $X$ exists and is unique at least locally by the Cauchy-Lipschitz theorem. 
Indeed, since $\tilde v_\eps', \tilde h_\eps', a', a''$ are bounded smooth and integrable, using $U\in C(0,T;\mathcal{H})$ together with the change of variables $\xi\mapsto \xi-X(t)$ as in \eqref{move-X}, we find that the right-hand side of the ODE \eqref{X-def} is uniformly Lipschitz continuous in $X$, and is continuous in $t$.

Moreover, the global-in-time existence and uniqueness of the shift are verified by the a priori estimate \eqref{112}.

The main proposition as a corner stone of proof of the Theorem is the following.
\begin{proposition}\label{prop:main}
There exist $\eps_0,\delta_0>0$, such that for any $\eps<\eps_0$ and $\delta_0^{-1}\eps<\lambda<\delta_0<1/2$, the following is true.\\
For any $U\in\mathcal{H} \cap \{U~|~|Y(U)|\le\eps^2 \}$,
\beq\label{prop:est}
\mathcal{R}(U):= -\frac{1}{\eps^4}Y^2(U) +(1+\delta_0(\eps/\lambda))|\mathcal{B}(U)|- \mathcal{G}(U) \le 0.
\eeq
\end{proposition}
 
 Most of the rest of the paper will be dedicated to the proof of this result. We will first show how this proposition implies Theorem~\ref{thm_main}.

\subsection{Proof of Theorem~\ref{thm_main} from Proposition \ref{prop:main}}

Based on \eqref{ineq-1} and \eqref{X-def}, to get the contraction estimate, it is enough to prove that for almost every time $t>0$ 
\beq\label{contem0}
\Phi_\eps (Y(U^X)) \Big(2|\mathcal{B}(U^X)|+1 \Big) Y(U^X) +\mathcal{B}(U^X)-\mathcal{G}(U^X)\le0.
\eeq
For every $U\in \mathcal{H}$ we define 
\beq\label{rhs}
\mathcal{F}(U):=\Phi_\eps (Y(U)) \Big(2|\mathcal{B}(U)|+1 \Big)Y(U) +|\mathcal{B}(U)|-\mathcal{G}(U).
\eeq
From \eqref{Phi-d}, we have 
\beq\label{XY}
\Phi_\eps (Y) \Big(2|\mathcal{B}|+1 \Big)Y\le
\left\{ \begin{array}{ll}
     -2|\mathcal{B}|,\quad \mbox{if}~  |Y|\ge \eps^2,\\
     -\frac{1}{\eps^4}Y^2,\quad  \mbox{if}~ |Y|\le \eps^2. \end{array} \right.
\eeq
Hence,   for all $U\in \mathcal{H}$ satisfying  $|Y(U)|\ge \eps^2 $, we have 
$$
\mathcal{F}(U) \le -|\mathcal{B}(U)|-\mathcal{G}(U) \le 0.
$$
Using both (\ref{XY}) and Proposition \ref{prop:main}, we find that for all $U\in \mathcal{H}$ satisfying  $|Y(U)|\le \eps^2 $, 
$$
\mathcal{F}(U) \le -\delta_0\left(\frac{\eps}{\lambda}\right)|\mathcal{B}(U)| \le 0.
$$
 Since $\delta_0 <1/2$,   these two estimates show that for every $U\in \mathcal{H}$ we have 
 $$
\mathcal{F}(U) \le -\delta_0\left(\frac{\eps}{\lambda}\right)|\mathcal{B}(U)|.
$$
For every fixed $t>0$, using this estimate with $U=U^X(t,\cdot)$, together with  \eqref{ineq-1}, and \eqref{contem0} gives 
\beq\label{111}
\frac{d}{dt}\int_{\bbr} a\eta(U^X|\tilde U_\eps) d\xi \le \mathcal{F}(U^X) \le -\delta_0\left(\frac{\eps}{\lambda}\right)|\mathcal{B}(U^X)|.
\eeq
Thus, $\frac{d}{dt}\int_{\bbr} a\eta(U^X|\tilde U_\eps) d\xi \le0$, which completes \eqref{cont_main}.\\
Moreover, since it follows from \eqref{111} that
\[
\delta_0\left(\frac{\eps}{\lambda}\right)\int_0^{\infty}|\mathcal{B}(U^X)|dt \le  \int_{\bbr} \eta(U_0|\tilde U_\eps) d\xi <\infty \quad\mbox{by the initial condition},
\]
using \eqref{X-def} and $\|\Phi_\eps\|_{L^{\infty}(\bbr)}\le \frac{1}{\eps^2}$ by \eqref{Phi-d}, we have
\beq\label{112}
|\dot X|\le \frac{1}{\eps^2} + \frac{2}{\eps^2}|\mathcal{B}|, \quad  \|\mathcal{B}\|_{L^1(0,\infty)}\le \frac{1}{\delta_0}\frac{\lambda}{\eps}\int_{\bbr} \eta(U_0|\tilde U_\eps) d\xi.
\eeq
This provides the global-in-time estimate \eqref{est-shift1}, thus $X\in W^{1,1}_{loc}(\bbr^+)$. This completes the proof of Theorem \ref{thm_main}.
\vskip0.1cm The rest of the paper is dedicated to the proof of Proposition \ref{prop:main}.

\subsection{An estimate on specific polynomials}

Let $\theta:=\sqrt{5-\frac{\pi^2}{3}}$, and $\delta>0$ be any constant.
We consider the following polynomial functionals:
\begin{eqnarray*}
&& E(\zo,\zt):=\zo^2+\zt^2+2\zo,\\
&&P_\delta(\zo,\zt):=(1+\delta)( Z_1^2 +Z_2^2) +2Z_1Z_2^2 +\frac{2}{3}Z_1^3  +6\delta (|Z_1|Z_2^2 +|Z_1|^3) \\
&&\qquad\qquad\qquad -2\left(1-\delta-\Big(\frac{2}{3}+\delta\Big)\theta Z_2\right)Z_2^2.
\end{eqnarray*}

This section is dedicated to the proof of the following proposition.
\begin{proposition}\label{prop:algebra}
There exist  $\delta_0, \delta_1>0$ such that for any $0<\delta<\delta_0$, the following is true.\\
If $(\zo,\zt)\in \bbr^2$ satisfies $ |E(\zo,\zt)|\leq \delta_1$, then
\beq\label{Pclaim}
 P_\delta(\zo,\zt)-|E(\zo,\zt)|^2\leq 0.
\eeq
\end{proposition}
This proposition will be used when a smallness condition on the perturbation, due to the shift, will be available. It should be noticed that the expansion leading to this polynomial is not merely a linearization. We end up with a polynomial $P_\delta$ which is of order 3. 
\vskip0.3cm
\begin{proof}
We split the proof into three steps.
\vskip0.1cm \noindent{\bf Step 1.} For $r>0$, we denote $B_{r}(0)$ the open ball centered at the origin with radius $r$.  We show the following claim: There exist $r>0$ and  $\delta_0>0$ such that for any $\delta\leq \delta_0$,
\begin{equation}\label{step1}
P_\delta(\zo,\zt)-|E(\zo,\zt)|^2\leq 0, \mathrm{ \ whenever \ }(Z_1,Z_2)\in B_{r}(0).
\end{equation}

To prove the claim, notice first that  $|Z_1|, |Z_2|\le r$ on $B_{r}(0)$. So we have
\[
|2Z_1|^2=(E-(Z_1^2+Z_2^2))^2\le 2|E|^2 + 2|Z_1^2+Z_2^2|^2 \le 2|E|^2 + 2r^2 (Z_1^2+Z_2^2),
\]
which implies
\[
-|E|^2 \le -2Z_1^2 +r^2 (Z_1^2+Z_2^2).
\]
Thus, for any $(Z_1,Z_2)\in B_{r}(0)$,
\begin{eqnarray*}
&& P_\delta-|E|^2 \le -2Z_1^2 +(1+\delta)\Big(Z_1^2+Z_2^2 +\frac{r^2}{1+\delta}(Z_1^2+Z_2^2)+ \frac{(2+6\delta)r}{1+\delta}Z_2^2 + \frac{((2/3)+6\delta)r}{1+\delta} Z_1^2\Big) \\
&&\qquad\qquad\qquad  -2\left(1-\delta-\frac{2}{3}(1+\delta)\theta r \right)\zt^2.
\end{eqnarray*}

%\[
%P_\delta(\zo,\zt)-|E|^2 \le -2Z_1^2 +(1+\delta) \Big( Z_1^2+Z_2^2+\frac{r^2}{1+\delta}(Z_1^2+Z_2^2) + 2rZ_2^2 +\frac{2}{3}rZ_1^2 \Big) -2(1-\delta)\left(1-\frac{2}{3}\theta r \right)\zt^2.
%\]
Taking $\delta_0$ and $r$ small enough, we can ensure that for any $\delta<\delta_0$,
\[
P_\delta(\zo,\zt)-|E|^2 \le 0, \qquad \mathrm{on} \ B_r(0).
\]
This proves the claim (\ref{step1}).

\vskip0.1cm\noindent {\bf Step 2.}
The second step is dedicated to the proof of the following claim.
There exists $\delta_0>0$ (possibly smaller that in the step 1), and $\delta_1>0$ such that for any $0<\delta\leq \delta_0$  we have :%if $(\zo,\zt)\in \bbr^2$ satisfies $ |E(\zo,\zt)|\leq \delta_1$ and $(Z_1,Z_2)\notin B_{r}((0))$, then
\begin{equation}\label{step2}
P_\delta(\zo,\zt)< 0, \qquad \mathrm{whenever \ }  |E(\zo,\zt)|\leq \delta_1, \qquad \mathrm{and \ }(Z_1,Z_2)\notin B_{r}(0).
\end{equation}

To show this claim, we first observe the limiting case: if $\delta=0$ and $E(Z_1,Z_2)=0$, we have
\[
P_0(Z_1,Z_2)=2Z_1-2Z_1^2-\frac{4}{3}Z_1^3 +\frac{4\theta}{3} \Big(-Z_1^2 -2Z_1 \Big)^{3/2},\quad Z_1^2+Z_2^2+2Z_1=0.
\]
Since $(Z_1+1)^2+Z_2^2=1$ by $E=0$, we have $-2\le Z_1\le 0$. Then by the algebraic inequality in Lemma \ref{lem-alge}, we have
\[
P_0(Z_1,Z_2)<0,\quad Z_1^2+Z_2^2+2Z_1=0,~ Z_1\neq 0.
\]
Since $P_0$ is continuous, it attains its maximum $-c<0$ on the compact set $\{E(Z_1,Z_2)=0\} \setminus B_r(0)$. In addition, $P_0$ is uniformly continuous on the compact set $\{|E(Z_1,Z_2)|\leq 1\}\setminus B_r(0)$, so there exist $0<\delta_1<1$ such that 
\[
P_0(Z_1,Z_2)<-c/2,\quad \mathrm{whenever \ } |E(Z_1,Z_2)|\leq \delta_1,~\mathrm{and \ } (Z_1,Z_2)\notin B_r(0).
\]
Taking $\delta_0$ small enough we still have for $\delta<\delta_0$: 
\[
P_\delta(Z_1,Z_2)<0,\quad \mathrm{whenever \ } |E(Z_1,Z_2)|\leq \delta_1,~ \mathrm{and \ }(Z_1,Z_2)\notin B_r(0).
\]
This proves the claim (\ref{step2}).
\vskip0.1cm\noindent{\bf Step 3.} The proved claims (\ref{step1}) and  (\ref{step2}) together give the Proposition  \ref{prop:algebra}.
\end{proof}

\subsection{A nonlinear Poincar\'e type inequality}

For any $\delta>0$, and any function $W\in L^2(0,1)$ such that 
$\sqrt{y(1-y)}\partial_yW\in L^2(0,1)$, we define
\begin{eqnarray*}
&&\Rd(W)=-\frac{1}{\delta}\left(\int_0^1W^2\,dy+2\int_0^1 W\,dy\right)^2+(1+\delta)\int_0^1 W^2\,dy\\
&&\qquad\qquad+\frac{2}{3}\int_0^1 W^3\,dy +\delta \int_0^1 |W|^3\,dy  -(1-\delta)\int_0^1 y(1-y)|\partial_y W|^2\,dy.
\end{eqnarray*}
This section is dedicated to the proof of the following proposition.
\begin{proposition}\label{prop:W}
For a given $C_1>0$, there exists $\deltat>0$, such that for any $\delta<\deltat$ the following is true.\\ For any $W\in L^2(0,1)$ such that 
$\sqrt{y(1-y)}\partial_yW\in L^2(0,1)$, if $\int_0^1 |W(y)|^2\,dy\leq C_1$, then
\beq\label{pWc}
 \Rd(W)\leq0.
\eeq
\end{proposition}
Note that the constant $C_1$ may not be small. Therefore we cannot discard the cubic term in $\mathcal{R}_\delta(W)$.

\begin{proof}
Let $\overline W =\int_0^1 W dy$. We first separate the first cubic term in $\mathcal{R}_\delta$ into the three parts:
\begin{align}
\begin{aligned}\label{cubic-form}
\int_0^1 W^3 dy &= \int_0^1 \Big((W-\overline W)+\overline W\Big)^3 dy\\
&=\int_0^1 (W-\overline W)^3 dy +3\overline W\int_0^1 (W-\overline W)^2 dy + \int_0^1\overline W^3 dy\\
&=\int_0^1 (W-\overline W)^3 dy + 2\overline W\int_0^1 (W-\overline W)^2 dy + \overline W \int_0^1 W^2 dy.
\end{aligned}
\end{align}
Thus, we have
\begin{align}
\begin{aligned}\label{rw0}
\Rd(W)&=-\frac{1}{\delta}\left(\int_0^1W^2\,dy+2\int_0^1 W\,dy\right)^2 +(1+\delta) \int_0^1 W^2 dy+ \frac{4}{3} \overline W\int_0^1 (W-\overline W)^2 dy\\
&+  \frac{2}{3}  \overline W \int_0^1 W^2 dy  +  \frac{2}{3}\int_0^1 (W-\overline W)^3 dy +\delta \int_0^1 |W|^3 dy
-(1-\delta)\int_0^1 y(1-y) |\partial_y W|^2 dy.
\end{aligned}
\end{align}
Let
\[
Z_1:=\overline W,\quad Z_2:= \Big(\int_0^1 (W-\overline W)^2 dy \Big)^{\frac{1}{2}},\quad E(Z_1,Z_2)=Z_1^2+Z_2^2+2Z_1.
\]
In what follows, we rewrite $\mathcal{R}_\delta$ in terms of the new variables $Z_1$ and $Z_2$.\\
Since 
\[
 \int_0^1 W^2 dy=Z_1^2+Z_2^2,
\]
and
\begin{align*}
\begin{aligned}
\int_0^1 |W|^3 dy &\le \int_0^1 \Big(|W-\overline W|+|\overline W|\Big)^3 dy\\
&\le\int_0^1 |W-\overline W|^3 dy +3|\overline W|\int_0^1 |W-\overline W|^2 dy+3|\overline W|^2 \int_0^1 |W-\overline W| dy + |\overline W|^3 \\
&\le\int_0^1 |W-\overline W|^3 dy  + 3|Z_1|Z_2^2 + 3|Z_1|^{3/2}|Z_1|^{1/2}Z_2 +|Z_1|^3\\
&\le\int_0^1 |W-\overline W|^3 dy  + 6|Z_1|Z_2^2 +4|Z_1|^3,
\end{aligned}
\end{align*}
we have 
\begin{align}
\begin{aligned}\label{fr2}
\mathcal{R}_\delta&=-\frac{1}{\delta} |E(Z_1,Z_2)|^2 + (1+\delta)( Z_1^2 +Z_2^2) +2Z_1Z_2^2 +\frac{2}{3}Z_1^3  +6\delta (|Z_1|Z_2^2 +|Z_1|^3) + \mathcal{P},
\end{aligned}
\end{align}
where
\beq\label{pterm}
\mathcal{P}:= \Big(\frac{2}{3}+\delta\Big)\int_0^1 |W-\overline W|^3 dy -(1-\delta)\int_0^1 y(1-y) |\partial_y W|^2 dy.
\eeq
For the cubic term in $\mathcal{P}$, we use Lemma \ref{lem-sup} to estimate
\begin{align}
\begin{aligned}\label{K_1}
& \int_{0}^1 \Big|W-\int_0^1 W\Big|^3 dy \\
&\le \int_0^1 z(1-z)|\partial_z W|^2 dz  \int_0^1|L(y)+L(1-y)| \Big|W-\int_0^1 W\Big|dy \\
&\le \int_0^1 z(1-z)|\partial_z W|^2 dz\Big(\int_0^1(L(y)+L(1-y))^{2}dy\Big)^{1/2}\Big(\int_0^1 \Big|W-\int_0^1 W\Big|^{2} dy\Big)^{1/2} \\
&= \theta Z_2\int_0^1 y(1-y)|\partial_y W|^2 dy.
\end{aligned}
\end{align}
%where
%\[
%\Big(\int_0^1(L(y)+L(1-y))^{2}dy\Big)^{1/2} =\sqrt{5-\frac{\pi^2}{3}}=\theta.
%\]
Thus, we have
\[
\mathcal{P}\le -\left(1-\delta-\Big(\frac{2}{3}+\delta\Big)\theta Z_2\right)\int_0^1 y(1-y)|\partial_y W|^2 dy.
\]
Since $(Z_1+1)^2  +Z_2^2 = 1+E(Z_1,Z_2)$, we have
\[
Z_2\le \sqrt{1+|E(Z_1,Z_2)|}.
\]
Since $\frac{2}{3}\theta=\frac{2}{3}\sqrt{5-\frac{\pi^2}{3}}\approx 0.88<1$, there exists a positive constant $\delta_\theta<1$ such that
\[
\frac{2}{3}\theta\sqrt{1+\delta_\theta} <1.
\]
Then, we take $\delta_2<1$ such that $\forall \delta<\delta_2$,
\[
1-\delta -\Big(\frac{2}{3}+\delta\Big) \theta\sqrt{1+\delta_\theta} >0.
\] 
We consider now two cases, whether $|E(Z_1,Z_2)|\le \min\{\delta_\theta, \delta_1\}$, or $|E(Z_1,Z_2)|\ge \min\{\delta_\theta, \delta_1\}$, where $\delta_1$ is the constant as in Proposition \ref{prop:algebra}. 
\vskip0.2cm
\noindent{\bf Case 1:}  Assume that 
\beq\label{ass1}
|E(Z_1,Z_2)|\le \min\{\delta_\theta, \delta_1\}.
\eeq
Then we find that $\forall \delta<\delta_2$,
\begin{eqnarray*}
1-\delta-\Big(\frac{2}{3}+\delta\Big)\theta Z_2 &\ge&1-\delta- \Big(\frac{2}{3}+\delta\Big)\theta \sqrt{1+\min\{\delta_\theta, \delta_1\}} \\
& \ge&1-\delta- \Big(\frac{2}{3}+\delta\Big)\theta \sqrt{1+\delta_\theta} >0.
\end{eqnarray*}
Therefore, we use the weighted Poincar\'e inequality \eqref{poincare} to have
\[
\mathcal{P}\le -2\left(1-\delta-\Big(\frac{2}{3}+\delta\Big)\theta Z_2\right)Z_2^2.
\]

Therefore, we have
\begin{align*}
\begin{aligned}
\mathcal{R}_\delta&\le -\frac{1}{\delta} |E(Z_1,Z_2)|^2 + (1+\delta)( Z_1^2 +Z_2^2) +2Z_1Z_2^2 +\frac{2}{3}Z_1^3  +6\delta (|Z_1|Z_2^2 +|Z_1|^3) \\
&\qquad -2\left(1-\delta-\Big(\frac{2}{3}+\delta\Big)\theta Z_2\right)Z_2^2\\
&=-\frac{1}{\delta} |E(Z_1,Z_2)|^2 +P_\delta(Z_1,Z_2).
\end{aligned}
\end{align*}
Hence, taking $\delta_2<\min\{\delta_0,1\}$ where $\delta_0$ is the constant as in Proposition \ref{prop:algebra}, and using Proposition \ref{prop:algebra} with \eqref{ass1}, we have $\mathcal{R}_\delta\le0$ for all $\delta<\delta_2$ under the assumption \eqref{ass1}.

\vskip0.2cm
\noindent{\bf Case 2.}  Assume now that 
\[
|E(Z_1,Z_2)|\ge \min\{\delta_\theta, \delta_1\}.
\]
We now use the assumption
\[
\int_0^1 |W(y)|^2\,dy\leq C_1,
\]
from which, all bad terms except for $\int_0^1(W-\overline W)^3dy$ in \eqref{rw0} are bounded by some constant $\tilde C_1$ depending on $C_1$. Therefore, we have
\begin{align*}
\begin{aligned}
\Rd \le -\frac{1}{\delta}\min\{\delta_\theta, \delta_1\}^2 +\tilde C_1  +  \frac{2}{3}(1+\delta)\int_0^1 (W-\overline W)^3 dy-(1-\delta)\int_0^1 y(1-y) |\partial_y W|^2 dy.
\end{aligned}
\end{align*}
For the remaining cubic term, we use Lemma \ref{lem-sup} to have
\begin{align*}
\begin{aligned}
& \int_{0}^1 (W-\overline W)^3 dy \\
&\quad\le \Big(\int_0^1 y(1-y)|\partial_y W_1|^2 dy \Big)^{3/4}  \int_0^1\Big|L(y)+L(1-y)\Big|^{3/4} \Big|W_1-\int_0^1 W_1\Big|^{3/2} dy\\
&\quad\le \Big(\int_0^1 y(1-y)|\partial_y W_1|^2 dy \Big)^{3/4}  \Big(\int_0^1\Big|L(y)+L(1-y)\Big|^{3}dy\Big)^{1/4}\Big(\int_0^1 \Big|W_1-\int_0^1 W_1\Big|^{2} dy\Big)^{3/4}.
\end{aligned}
\end{align*}
Then, using Young's inequality, we have
\begin{align*}
\begin{aligned}
\frac{2}{3}\int_0^1 (W-\overline W)^3 dy &\le \frac{1}{2} \int_0^1 y(1-y)|\partial_y W_1|^2 dy +C \Big(\int_0^1 \Big|W_1-\int_0^1 W_1\Big|^{2} dy\Big)^{3}\\
 &\le \frac{1}{2} \int_0^1 y(1-y)|\partial_y W_1|^2 dy +\tilde C_1,
\end{aligned}
\end{align*}
Therefore, 
\[
\mathcal{R}_\delta \le -\frac{1}{\delta}\min\{\delta_\theta, \delta_1\}^2 +2\tilde C_1.
\]
Hence, choosing $\delta_2<\min\{\delta_0,1\}$ small enough such that $-\frac{1}{\delta_2}\min\{\delta_\theta, \delta_1\}^2 +2\tilde C_1< 0$, we have $\mathcal{R}_\delta < 0$.\\
This completes the proof of Proposition \ref{prop:algebra}.
\end{proof}

\subsection{Expansion in the size of the shock}\label{section-expansion}

We define the following functionals:
\begin{align*}
\begin{aligned}
&Y_g(v):=-\frac{1}{2\sigma_\eps^2}\int_\bbr a' |p(v)-p(\tilde v_\eps)|^2 d\xi -\int_\bbr a' Q(v|\tilde v_\eps) d\xi -\int_\bbr a \partial_\xi p(\tilde v_\eps)(v-\tilde v_\eps)d\xi\\
&\quad+\frac{1}{\sigma_\eps}\int_\bbr a \partial_\xi \tilde h_\eps\big(p(v)-p(\tilde v_\eps)\big)d\xi,\\
&\mathcal{B}_1(v):= \sigma_\eps\int_\bbr a \partial_\xi \tilde v_\eps p(v|\tilde v_\eps) d\xi,\\
&\mathcal{B}_2(v):= \frac{1}{2\sigma_\eps} \int_\bbr a' |p(v)-p(\tilde v_\eps)|^2d\xi+ \frac{1}{2}\int_\bbr a'' |p(v)-p(\tilde v_\eps)|^2 d\xi,\\
&\mathcal{G}_2(v):=\sigma_\eps  \int_\bbr  a' Q(v|\tilde v_\eps) d\xi, \\
&\mathcal{D}(v):=\int_\bbr a |\partial_\xi (p(v)-p(\tilde v_\eps))|^2 d\xi.
\end{aligned}
\end{align*}

Note that all these quantities depend only on $v$ (not on $h$).
This section is dedicated to the proof of the following proposition.

\begin{proposition}\label{prop:main3}
%There exists a universal $c^*>0$ such that the following is true.
For any $C_2>0$, there exist $\eps_0, \delta_3>0$, such that for any $\eps\in(0,\eps_0)$, and any $\lambda, \delta\in(0,\delta_3)$ such that $\eps\leq \lambda$,
%verifying 
%$$
%\delta \leq c_0(\eps/\lambda),
%$$
 the following is true.\\
For any function $v:\bbr\to \bbr^+$ such that $\mathcal{D}(v)+\mathcal{G}_2(v)$%+|\mathcal{B}(v)|$
 is finite, if
\beq\label{assYp}
|Y_g(v)|\leq C_2 \frac{\eps^2}{\lambda},\qquad  \|p(v)-p(\vt)\|_{L^\infty(\bbr)}\leq \delta_3,
\eeq
then
\begin{align}
\begin{aligned}\label{redelta}
&\mathcal{R}_{\eps,\delta}(v):=-\frac{1}{\eps\delta}|Y_g(v)|^2 +(1+\delta)|\mathcal{B}_1(v)|\\
&\qquad\qquad\qquad+\left(1+\delta\left(\frac{\eps}{\lambda}\right)\right)|\mathcal{B}_2(v)|-\left(1-\delta\left(\frac{\eps}{\lambda}\right)\right)\mathcal{G}_2(v)-(1-\delta)\mathcal{D}(v)\le 0.
\end{aligned}
\end{align}

\end{proposition}

This proposition shows that we can afford an error of order 1 on $\mathcal{D}(v)$ and $\mathcal{B}_1(v)$ (up to $\delta$), but only of order $\eps/\lambda$ on $\mathcal{G}_2(v)$ and $\mathcal{B}_2(v)$.

\begin{proof}
We first impose on $(\delta_3, \eps_0)$ that 
$$
\delta_3\leq \min(\delta_*,1/2),\qquad \eps_0\leq \min (\delta_*, p(v_-)/2),
$$
where $\delta_*$ is defined by the Lemma \ref{lem:local}.  That way, the function $a$ is positive, 
the function $p(\tilde{v}_\eps)$ is uniformly bounded, and we can apply the results of Lemma \ref{lem:local} on  $v$ and $w=\tilde{v}_\eps$.

To simplify the notations, we set $\sigma=\sqrt{-p'(v_-)}>0$. This is a fixed quantity which does not depend on $\eps$ and $\lambda$. 
Note that from  \eqref{shock_speed}  we have 
\begin{equation}\label{sigma-f}
|\sigma+\sigma_\eps|\leq C\eps.%,\qquad \mathrm{and \ \ }|\sigma_\eps^2+p'(v_-)|\leq C \eps.
\end{equation}
But, since $|\tilde v_\eps-v_-|\leq C \eps$,  and $\sigma^2=-p'(v_-)=\gamma p(v_-)^{\frac{1}{\gamma}+1}$ we have actually:
\beq\label{sigma-p}
\sup_{\xi\in\bbr}|\sigma^2+p'(\tilde{v}_\eps(\xi))|\leq C \eps,
\qquad \mathrm{and \ \ }
\sup_{\xi\in\bbr}\left|\frac{1}{\sigma^2}-\frac{p(\tilde{v}_\eps(\xi))^{-\frac{1}{\gamma}-1}}{\gamma}\right|\leq C\eps.
\eeq

\vskip0.5cm
We now rewrite the above functionals $Y_g, \mathcal{B}, \mathcal{G}_2, \mathcal{D}$ w.r.t. the following variables 
\beq\label{w1w2}
w:=p(v)-p(\tilde v_\eps),\quad y:=\frac{p(\tilde v_\eps(\xi))-p(v_-)}{[p]}.
\eeq
Notice that since $p(\tilde v_\eps(\xi))$ is increasing in $\xi$, we use a change of variable $\xi\in\bbr\mapsto y\in[0,1]$.\\
Then it follows from \eqref{weight-a} that $a=1-\lambda y$ and
\beq\label{ach}
a'(\xi)=-\lambda \frac{p(\tilde v_\eps)'}{[p]},
%\quad\mbox{and}
\qquad \frac{dy}{d\xi}= \frac{p(\tilde v_\eps)'}{[p]}, \qquad |a-1|\leq \delta_3.
\eeq

%AAAAAAAAAAAAAAAAAAAA

$\bullet$ {\bf Change of variable for $Y_g$:} 
We decompose the $Y_g$ term as follows.
\begin{align*}
\begin{aligned}
Y_g &= \underbrace{-\frac{1}{2\sigma_\eps^2}\int_{\bbr} a' |p(v)-p(\tilde v_\eps)|^2 d\xi}_{=:Y_1} \underbrace{-\int_{\bbr} a' Q(v|\tilde v_\eps) d\xi}_{=:Y_2} \underbrace{-\int_{\bbr} a \partial_\xi p(\tilde v_\eps)(v-\tilde v_\eps)d\xi}_{=:Y_3}\\
&\quad\underbrace{+\frac{1}{\sigma_\eps}\int_{\bbr} a \partial_\xi \tilde h_\eps\big(p(v)-p(\tilde v_\eps)\big)d\xi}_{=:Y_4}.
\end{aligned}
\end{align*}
Using \eqref{ach}, we have
$$
Y_1 =\frac{\lambda}{2\sigma_\eps^2}\int_0^1 w^2 dy.
$$
Using (\ref{sigma-f}), we get
\begin{align}\label{alpha1}
\begin{aligned}
\left|Y_1 -\frac{\lambda}{2\sigma^2}\int_0^1 w^2 dy\right|\leq C\eps_0\lambda \int_0^1 w^2 dy.
\end{aligned}
\end{align}
%We take  $\delta_3\leq \delta_*$ where $\delta_*$ is defined in  in Lemma \ref{lem:local}. 
Using   \eqref{Q-est1} and  \eqref{Q-est11} from Lemma \ref{lem:local},  and $\|p(v)-p(\vt)\|_{L^\infty(\bbr)}\leq \delta_3 $ we find
\[
\left|Y_2-\frac{\lambda}{2\gamma} \int_0^1 p(\tilde v_\eps)^{-\frac{1}{\gamma}-1} w^2 dy\right|\leq C\delta_3\lambda \int_0^1 w^2 dy. 
\]
Moreover, using (\ref{sigma-p}), we find
%\beq\label{ps1}
%\left|p(\tilde v_\eps)^{-1}-p(v_-)^{-1}\right| \leq C\eps, \quad\mbox{(since $|\tilde v_\eps-v_-|\lesssim \eps$)},
%\eeq 
%we have
\begin{equation}\label{alpha2}
\left| Y_2-\frac{\lambda}{2\sigma^2}\int_0^1  w^2 dy\right|\leq C\lambda(\eps_0+\delta_3)\int_0^1  w^2 dy.
\end{equation}
For $Y_3$, we first  write 
\[
v-\tilde v_\eps =p(v)^{-\frac{1}{\gamma}}-p(\tilde v_\eps)^{-\frac{1}{\gamma}}.
\]
Using the  Taylor expansion, we find that uniformly in $\xi$ and $\eps$:
\[
\left|(v-\tilde v_\eps) +\frac{p(\tilde v_\eps)^{-\frac{1}{\gamma}-1}}{\gamma}\big(p(v)-p(\tilde v_\eps)\big) \right|\leq C|p(v)-p(\tilde v_\eps)|^2\leq C\delta_3|p(v)-p(\tilde v_\eps)|.
\]
Using \eqref{sigma-p}, we get
$$
\left|(v-\tilde v_\eps) +\frac{1}{\sigma^2}\big(p(v)-p(\tilde v_\eps)\big) \right|\leq C(\eps_0+\delta_3)|p(v)-p(\tilde v_\eps)|.
$$
Then, using $\partial_\xi p(\tilde v_\eps)=\eps\frac{dy}{d\xi}$ (since $[p]=\eps$), and $|a-1|\leq \delta_3$, we have

\begin{align}\label{alpha3}
\begin{aligned}
\left|Y_3 -\frac{\eps}{\sigma^2}\int_0^1w\,dy\right|\leq C\eps(\eps_0+\delta_3)\int_0^1  |w| dy.
\end{aligned}
\end{align}
%Moreover, we have
%\[
%Y_3 =\frac{[p]}{\gamma} \Big(p(v_-)^{-\frac{1}{\gamma}-1}+\mathcal{O}(\eps_0+\delta_3)\Big) \int_0^1  w dy.
%\]
Using $\partial_\xi \tilde h_\eps =\frac{\partial_\xi p(\tilde v_\eps)}{\sigma_\eps}$, we have
\[
Y_4 =\frac{\eps}{\sigma_\eps^2}\int_0^1 (1-\lambda y)w dy,%=\frac{[p]}{\sigma_\eps^2}(1+\lambda)\int_0^1 w dy.
\]
and so 
\begin{equation}\label{alpha4}
\left| Y_4 -\frac{\eps}{\sigma^2}\int_0^1 w dy\right|\leq C\eps(\delta_3+\eps_0)\int_0^1  |w| dy.
\end{equation}
We combine all the terms of $Y_g$, and write the result on the renormalized quantity:
\beq\label{changeW}
W:=\frac{\lambda}{\eps} w.
\eeq
From \eqref{alpha1}, \eqref{alpha2}, \eqref{alpha3}, and \eqref{alpha4}, we obtain:
\begin{align}
\begin{aligned}\label{Ygc}
\left|\sigma^2 \frac{\lambda}{\eps^2 }Y_g-\int_0^1W^2\,dy-2\int_0^1 W\,dy\right|\leq C(\eps_0+\delta_3)\left(\int_0^1W^2\,dy+\int_0^1|W|\,dy\right).
% \frac{1}{\sigma_\eps^2}\bigg( \lambda(1+\mathcal{O}(\eps_0+\delta_3))\int_0^1 w^2 dy + 2[p](1+\mathcal{O}(\eps_0+\delta_3))\int_0^1 w dy\bigg).
\end{aligned}
\end{align}
\vskip0.3cm

$\bullet$ {\bf Change of variable for  $\mathcal{B}_1$ and $\mathcal{B}_2$:} 
We decompose the $\mathcal{B}_2$ term as follows.
\begin{align*}
\begin{aligned}
\mathcal{B} _2=  \underbrace{\frac{1}{2\sigma_\eps} \int_\bbr a' |p(v)-p(\tilde v_\eps)|^2d\xi}_{=:\mathcal{B}_{21}} +\underbrace{ \frac{1}{2}\int_\bbr a'' |p(v)-p(\tilde v_\eps)|^2 d\xi}_{=:\mathcal{ B}_{22}} .
\end{aligned}
\end{align*}
We first have
%\beq\label{b1-est}
$$
\mathcal{B}_{21}=-\frac{\lambda}{2\sigma_\eps} \int_0^1  w^2dy=\frac{\lambda}{2|\sigma_\eps|} \int_0^1  w^2dy.
$$
So
$$
\left|\mathcal{B}_{21}-\frac{\lambda}{2\sigma} \int_0^1  w^2dy\right|\leq\lambda \eps \int_0^1  w^2dy\leq \eps\delta_3 \int_0^1  w^2dy.
$$
Using  \eqref{ddp}, we get
\[
|\mathcal{B}_{22}|\leq C \eps\lambda\int_0^1  w^2dy\le C \eps\delta_3\int_0^1  w^2dy.
\]
So, finally:
\beq\label{b1-est}
\left|\mathcal{B}_{2}-\frac{\lambda}{2\sigma} \int_0^1  w^2dy\right|\leq C\eps\delta_3 \int_0^1  w^2dy.
\eeq

For $\mathcal{B}_1$, using $\partial_\xi \tilde v_\eps=\frac{\partial_\xi p(\tilde v_\eps)}{p'(\tilde v_\eps)}$, we first have
\[
\mathcal{B}_1= \sigma_\eps[p] \int_0^1 (1-\lambda y) \frac{1}{p'(\tilde v_\eps)} p(v|\tilde v_\eps) dy.
\]
Then, using $[p]=\eps$, \eqref{p-est1}, $\lambda\leq \delta_3$, and \eqref{sigma-p}, 
we have
\begin{align*}
\begin{aligned}
|\mathcal{B}_1| &\le  \eps|\sigma_\eps| \int_0^1 \Big(\frac{\gamma+1}{2\gamma}  |p'(\tilde{v}_\eps)|^{-1}p(\tilde{v}_\eps)^{-1} +C \delta_3\Big) (1-\lambda y) w^2\,dy\\
&\leq \eps |\sigma_\eps| \int_0^1 \Big( \frac{\gamma+1}{2\gamma} |p'(\tilde{v}_\eps)|^{-1}p(\tilde{v}_\eps)^{-1} +C \delta_3\Big)  w^2\,dy\\
&\leq \eps |\sigma_\eps|\Big( \frac{\gamma+1}{2\gamma} |p'(v_-)|^{-1}p(v_-)^{-1} +C (\eps_0+\delta_3)\Big)  \int_0^1  w^2\,dy\\
\end{aligned}
\end{align*}
Therefore
\beq\label{b3}
|\mathcal{B}_1| \leq \eps\frac{\gamma+1}{2\gamma \sigma p(v_-)} (1+C(\eps_0+\delta_3)) \int_0^1  w^2\,dy.
\eeq
%Moreover, using $p'(v_-)=-\gamma p(v_-)^{\frac{1}{\gamma}+1}$ and \eqref{sigma-p},
%we have
%\beq\label{b2}
%|\mathcal{\tilde B}_2|\le  [p] \Big( p(v_-)^{-\frac{3\gamma +1}{2\gamma}}\frac{\gamma+1}{2\gamma\sqrt{\gamma}} +\mathcal{O}(\eps_0+\delta_3)\Big) \int_0^1 w_1^2dy.
%\eeq
%therefore, taking $\delta_3\ll1$, $|\mathcal{\tilde B}_3|$ is negligible by the r.h.s. in \eqref{b2}.
Note that the right hand side of (\ref{b1-est}) is small compared to $\mathcal{B}_1$. But the main part of $\mathcal{ B}_2$ is big compared to $\mathcal{B}_1$ (as $\lambda/\eps$). It will be compensated with the first order term in $\mathcal{G}_2$.
We denote 
$$
\alpha_\gamma = \frac{\gamma \sigma p(v_-)} {\gamma+1}.
$$
This number depends only on $v_-$ and  $\gamma$, but not on $\eps$ nor $\lambda$. 
We gather  all the terms  of $\mathcal{B}_1$ and $\mathcal{B}_2$, and write  the result on the renormalized quantity (\ref{changeW}).  Thanks to \eqref{b1-est} and \eqref{b3}  we find
\beq\label{newb2}
%\frac{\sigma}
{2\alpha_\gamma} \frac{\lambda^2}{\eps^3}|\mathcal{B}_2|\leq \left(\frac{\alpha_\gamma}{\sigma}\left(\frac{\lambda}{\eps}\right)+C(\eps_0+\delta_3)\right)\int_0^1 W^2\,dy,
\eeq
\beq\label{newb1}
%\frac{\sigma}
{2\alpha_\gamma} \frac{\lambda^2}{\eps^3}|\mathcal{B}_1|\leq \left(1 +C(\eps_0+\delta_3)\right)\int_0^1 W^2\,dy.
\eeq

\vskip0.3cm

$\bullet$ {\bf Change of variable for $\mathcal{G}_2$:}
We use \eqref{ach}, \eqref{Q-est11},  and \eqref{sigma-p} to get
\begin{align*}
\begin{aligned}
\mathcal{G}_2&=- \sigma_\eps\lambda  \int_0^1   Q(v|\tilde v_\eps) dy \\
&\ge-\frac{\sigma_\eps\lambda}{2\gamma} \int_0^1 p(\tilde v_\eps)^{-\frac{1}{\gamma}-1} w^2 dy +\sigma_\eps\lambda\frac{1+\gamma}{3\gamma^2}\int_0^1 p(\tilde v_\eps)^{-\frac{1}{\gamma}-2} w^3 dy\\
&\ge\left(\frac{\lambda}{2\sigma}-C\eps\delta_3 \right)\int_0^1  w^2 dy -\frac{\lambda}{3\alpha_\gamma}\int_0^1 w^3 dy-C\frac{\eps_0\lambda}{\alpha_\gamma}\int_0^1 |w|^3 dy.
\end{aligned}
\end{align*}
When renormalizing with \eqref{changeW}, we obtain:
\begin{equation}\label{H_2}
-2\alpha_\gamma \frac{\lambda^2}{\eps^3}\mathcal{G}_2\leq  \left(-\frac{\alpha_\gamma}{\sigma}\left(\frac{\lambda}{\eps}\right) +C\delta_3     \right)\int_0^1W^2\,dy +\frac{2}{3}\int_0^1W^3\,dy+C\eps_0\int_0^1|W|^3\,dy.
\end{equation}
Note that the very first term in the inequality \eqref{H_2} will exactly cancel the divergent term of $\mathcal{B}_2$. This is why an expansion to the  order three is needed.

\vskip0.3cm

$\bullet$ {\bf Change of variable on $\mathcal{D}$:} 
To deal with the diffusion term $\mathcal{D}$, we first need  a uniform in $y$ estimate on  $\frac{dy}{d\xi}$. This is provided by the following lemma.
\begin{lemma}
There exists a constant $C>0$ such that 
for any $\eps\leq \eps_0$, and any $y\in [0,1]$:
$$
\left|\frac{dy/d\xi}{y(1-y)}-\frac{\eps}{2\alpha_\gamma}\right|\leq C\eps^2.
$$
\end{lemma}
\begin{proof} 
 From \eqref{small_shock1} we have
\[
p(\tilde v_\eps)'=\sigma_\eps (\tilde v_{\eps}-v_-) + \frac{p(\tilde v_\eps)-p(v_-)}{\sigma_\eps},
\]
therefore
\[
\eps\frac{dy}{d\xi}= p(\tilde v_\eps)'=\frac{1}{\sigma_\eps}\Big(\sigma_\eps^2 (\tilde v_{\eps}-v_-) + p(\tilde v_\eps)-p(v_-)\Big),
\]
with 
$$
\sigma_\eps^2=\frac{p(v_+)-p(v_-)}{v_--v_+}.
$$
Plugging the expression of $\sigma_\eps^2$ into the one of $\eps\frac{dy}{d\xi}$ and writing the result with respect to differences of values of functions at $\tilde{v}_\eps$ and at the end points $v_\pm$, we find
\begin{eqnarray*}
\eps\frac{dy}{d\xi}&=&\frac{1}{\sigma_\eps(v_--v_+)}\Big(       (p(v_+)-p(v_-)) (\tilde v_{\eps}-v_-) +  (p(\tilde v_\eps)-p(v_-))(v_--v_+)  \Big)\\
&=&\frac{1}{\sigma_\eps(v_--v_+)} \Big(       (p(v_+)-p(\tilde v_\eps)) (\tilde v_{\eps}-v_-)+  (p(\tilde v_\eps)-p(v_-))(\tilde v_{\eps}-v_-) \\
&&\qquad\qquad+ (p(\tilde v_\eps)-p(v_-))(v_--\tilde v_{\eps})  +(p(\tilde v_\eps)-p(v_-))(\tilde v_{\eps}-v_+)\Big)\\
&=&\frac{1}{\sigma_\eps(v_--v_+)}\Big((p(v_+)-p(\tilde v_\eps)) (\tilde v_{\eps}-v_-)+(p(\tilde v_\eps)-p(v_-))(\tilde v_{\eps}-v_+)\Big).
\end{eqnarray*}
Hence
$$
\eps\frac{dy}{d\xi}=\frac{ (p(v_+)-p(\tilde v_\eps)) (p(\tilde v_\eps)-p(v_-))}{\sigma_\eps(v_--v_+)}\left( \frac{\tilde v_{\eps}-v_-}{p(\tilde v_\eps)-p(v_-)}+\frac{\tilde v_{\eps}-v_+}{p(v_+)-p(\tilde v_\eps)} \right).
$$
Then, using
$$
 y=\frac{p(\tilde v_\eps)-p(v_-)}{\eps}, \qquad 1-y=\frac{p(v_+)-p(\tilde v_\eps)}{\eps},
$$
we have
\[
\frac{dy/d\xi}{y(1-y)}=\frac{\eps}{\sigma_\eps(v_--v_+)}\left( \frac{\tilde v_{\eps}-v_-}{p(\tilde v_\eps)-p(v_-)}+\frac{\tilde v_{\eps}-v_+}{p(v_+)-p(\tilde v_\eps)} \right).
\]
Then
\begin{eqnarray*}
&&\left|\frac{dy/d\xi}{y(1-y)}-\eps \frac{p''(v_-)}{2p'(v_-)^2\sigma}\right|\\
&&\qquad \le \underbrace{\left| \frac{dy/d\xi}{y(1-y)}-\eps \frac{p''(v_-)}{2p'(v_-)^2\sigma_\eps} \right|}_{=:I_1} + \underbrace{\eps \frac{p''(v_-)}{2p'(v_-)^2}\Big|\frac{1}{\sigma_\eps}+\frac{1}{\sigma}\Big|}_{=:I_2}.
\end{eqnarray*}
We use Lemma \ref{lemma_pression} to have
\begin{align*}
\begin{aligned}
I_1=\frac{\eps}{|\sigma_\eps|(v_--v_+)}\left| \frac{\tilde v_{\eps}-v_-}{p(\tilde v_\eps)-p(v_-)}+\frac{\tilde v_{\eps}-v_+}{p(v_+)-p(\tilde v_\eps)}+\frac{p''(v_-)}{2p'(v_-)^2}(v_--v_+)  \right|\le C\eps^2.
\end{aligned}
\end{align*}
Since it follows from \eqref{sigma-f} that $I_2\le C\eps^2$, we get
\[
\left|\frac{dy/d\xi}{y(1-y)}-\eps \frac{p''(v_-)}{2p'(v_-)^2\sigma}\right|\le C\eps^2.
\]
Since $p(v)=v^{-\gamma}$, we have 
$$
\frac{p''(v_-)}{p'(v_-)^2\sigma}=\frac{\gamma+1}{\gamma \sigma p(v_-)}=\frac{1}{\alpha_\gamma}.
$$
This ends the proof of the lemma.
\end{proof}

The  diffusion term $\mathcal{D}$ is as follows:
\beq\label{st-d}
\mathcal{D}=\int_0^1 (1-\lambda y) |\partial_y w|^2 \Big(\frac{dy}{d\xi}\Big) dy.
\eeq
Thanks to the last lemma, we have
\begin{eqnarray*}
\mathcal{D}&\geq&(1-\lambda)\int_0^1  |\partial_y w|^2 \Big(\frac{dy}{d\xi}\Big) dy\\
&\geq&(1-\lambda) (\eps/(2\alpha_\gamma)-C\eps^2)   \int_0^1y(1-y)  |\partial_y w|^2 dy\\
&\geq&\frac{\eps}{2\alpha_\gamma}(1-C(\delta_3+\eps_0))  \int_0^1y(1-y)  |\partial_y w|^2 dy.
\end{eqnarray*}
After normalization, we obtain:
\begin{equation}\label{newD}
-2\alpha_\gamma \frac{\lambda^2}{\eps^3}\mathcal{D}\leq -(1-C(\eps_0+\delta_3))\int_0^1y(1-y) |\partial_y W|^2dy.
\end{equation}

$\bullet$ {\bf Control on $W$:}
Using \eqref{assYp} and \eqref{Ygc}, we find that 
$$
\int_0^1W^2\,dy-2\left|\int_0^1W\,dy\right|\leq C+C(\eps_0+\delta_3)\left(\int_0^1W^2\,dy+\int_0^1|W|\,dy\right).
$$
But 
$$
\left|\int_0^1W\,dy\right|\leq \int_0^1|W|\,dy\leq \frac{1}{8} \int_0^1W^2\,dy+8.
$$
Hence
$$
\int_0^1W^2\,dy\leq 2\left|\int_0^1W\,dy\right|+C+C(\eps_0+\delta_3)\left(\int_0^1W^2\,dy+\int_0^1|W|\,dy\right)\leq C+24+\frac{1}{2}\int_0^1W^2\,dy,
$$
if $\eps_0$ and $\delta_3$ are chosen small enough.
Hence there exists a  constant $C_1>0$  depending on $C_2$, but not on $\eps$ nor $\eps/\lambda$,  such that 
\begin{equation}\label{controlW}
\int_0^1W^2\,dy\leq C_1.
\end{equation}
 Note that  we cannot expect any smallness on this constant. 
 
 $\bullet$ {\bf Control on the $|Y_g|^2$ term:} We have 
 $$
 -2\alpha_\gamma\left(\frac{\lambda^2}{\eps^3}\right) \frac{|Y_g|^2}{\eps \delta_3}
 =-\frac{2\alpha_\gamma}{\delta_3\sigma^4}\left| \frac{\sigma^2\lambda}{\eps^2}Y_g\right|^2.
 $$
 For any $a,b\in \bbr$, we have 
$$
%=(b-a)((a-b)+2b)=
 b^2-a^2=-(b-a)^2+2b(b-a)=-(b-a)^2+2\frac{b}{\sqrt{2}}\sqrt{2}(b-a)\leq (b-a)^2+\frac{b^2}{2}.
$$
 So 
 $$
 -a^2\leq -\frac{b^2}{2}+|b-a|^2.
 $$
 Using this inequality with 
 $$
 a=\frac{\sigma^2\lambda}{\eps^2}Y_g,\qquad b=\int_0^1W^2\,dy+2\int_0^1 W\,dy,
 $$
 and using \eqref{Ygc}, we find
 \begin{eqnarray*}
 && -2\alpha_\gamma\left(\frac{\lambda^2}{\eps^3}\right) \frac{|Y_g|^2}{\eps \delta_3}\leq
  -\frac{\alpha_\gamma}{\delta_3\sigma^4}\left| \int_0^1W^2\,dy+2\int_0^1 W\,dy\right|^2\\
  &&\qquad \qquad +\frac{C}{\delta_3}(\eps_0+\delta_3)^2\left(\int_0^1W^2\,dy+\int_0^1|W|\,dy \right)^2.
 \end{eqnarray*}
 Using \eqref{controlW}, we have 
 $$
 \left(\int_0^1W^2\,dy+\int_0^1|W|\,dy \right)^2\leq \left(\int_0^1W^2\,dy+\sqrt{\int_0^1|W|^2\,dy }\right)^2\leq C\int_0^1W^2\,dy.
 $$
 So, restricting $\eps_0$ such that $\eps_0\leq \delta_3$, we have
\begin{align}
\begin{aligned}\label{Y^2}
 & -2\alpha_\gamma\left(\frac{\lambda^2}{\eps^3}\right) \frac{|Y_g|^2}{\eps \delta_3}\leq
  -\frac{\alpha_\gamma}{\delta_3\sigma^4}\left| \int_0^1W^2\,dy+2\int_0^1 W\,dy\right|^2+C\delta_3\int_0^1W^2\,dy.
 \end{aligned}
\end{align}

 $\bullet$ {\bf Conclusion:} %For any $\delta<\delta_3\sigma^4/\alpha_\gamma$, we have 
 For any $\delta<\delta_3$, we have 
 \begin{eqnarray*}
 &&\mathcal{R}_{\eps,\delta}(v)\leq -\frac{1}{\eps\delta_3}|Y_g(v)|^2 +(1+\delta_3)|\mathcal{B}_1(v)|\\
 &&\qquad\qquad\qquad +\left(1+\delta_3\left(\frac{\eps}{\lambda}\right)\right)|\mathcal{B}_2(v)|-\left(1-\delta_3\left(\frac{\eps}{\lambda}\right)\right)\mathcal{G}_2(v)-(1-\delta_3)\mathcal{D}(v). %-(1-\delta)\mathcal{G}_2(v)-(1-\delta_3)\mathcal{D}(v).
 \end{eqnarray*}
Multiplying \eqref{newD} by $(1-\delta_3)$, \eqref{H_2} by $1-\delta_3(\eps/\lambda)$, \eqref{newb1} by $1+\delta_3$, \eqref{newb2} by $1+\delta_3(\eps/\lambda)$,  
and summing all these terms together with \eqref{Y^2}, we find (remember that $\eps_0\leq \delta_3$ and $\eps/\lambda\leq1$):% and using that $\delta\leq \delta_3$ and $\eps_0\leq \delta_3$, we find 
\begin{eqnarray*}
&&\qquad\qquad\qquad\qquad 2\alpha_\gamma\left(\frac{\lambda^2}{\eps^3}\right)\mathcal{R}_{\eps,\delta}(v)\\
&& \leq -
\frac{1}{C_\gamma\delta_3}\left(\int_0^1W^2\,dy+2\int_0^1 W\,dy\right)^2+(1+C \delta_3)\int_0^1 W^2\,dy\\%+2\frac{\alpha_\gamma}{\sigma}\delta \left(\frac{\lambda}{\eps}\right)\int_0^1 W^2\,dy\\
&&\qquad\qquad+\frac{2}{3}\int_0^1 W^3\,dy +C\delta_3 \int_0^1 |W|^3\,dy -(1-C \delta_3)\int_0^1 y(1-y)|\partial_y W|^2\,dy.%\leq R_{\bar{\delta}}, 
\end{eqnarray*}
Let us fix the value of the $\delta_2$ of Proposition \ref{prop:W} corresponding to  the constant $C_1=C$ of \eqref{controlW}. %We define now the last constant $c_0$ as
%$$
%c_0=\delta_2 \frac{\sigma}{4\alpha_\gamma}.
%$$
%Note that if $\delta\leq c_0 (\eps/\lambda)$, then $2\frac{\alpha_\gamma}{\sigma}\delta \left(\frac{\lambda}{\eps}\right)\leq \delta_2/2$.

Consider  $\bar{\delta}=\max(C_\gamma, C) \delta_3$, and  choose $\delta_3$ small enough, such that 
$\bar{\delta}$ is smaller than $\delta_2$. Then we have 
\begin{eqnarray*}
&&\qquad\qquad\qquad\qquad 2\alpha_\gamma\left(\frac{\lambda^2}{\eps^3}\right)\mathcal{R}_{\eps,\delta}(v)\\
&& \leq - \frac{1}{\delta_2}\left(\int_0^1W^2\,dy+2\int_0^1 W\,dy\right)^2+(1+\delta_2)\int_0^1 W^2\,dy\\
&&\qquad\qquad+\frac{2}{3}\int_0^1 W^3\,dy +\delta_2\int_0^1 |W|^3\,dy  -(1- \delta_2)\int_0^1 y(1-y)|\partial_y W|^2\,dy= R_{\delta_2}(W). 
\end{eqnarray*}
%$$
%2\alpha_\gamma\left(\frac{\lambda^2}{\eps^3}\right)\mathcal{R}_{\eps,\delta}(v)\leq \leq R_{\bar{\delta}}.
%$$
Then from Proposition  \ref{prop:W}, we have $ R_{\delta_2}(W)\leq0$. Therefore  $\mathcal{R}_{\eps,\delta}(v)\leq 0$, for any $\lambda, \delta\leq \delta_3$, $\eps\leq \eps_0$ with $\eps\leq \lambda$, and any $v$ such that  $\mathcal{D}(v)+\mathcal{G}_2(v)$ is finite, and verifying \eqref{assYp}.
 \end{proof}

%\vspace{1cm}
\subsection{Truncation of the big values of $|p(v)-p(\tilde v_\eps)|$}\label{section-finale}
In order to use Proposition \ref{prop:main3} for proof of Proposition \ref{prop:main}, we need to show that  the values for $p(v)$ such that $|p(v)-p(\tilde{v}_\eps)|\geq\delta_3$ have a small effect. However, the value of $\delta_3$ is itself conditioned to the constant $C_2$ in the proposition. Therefore, we need first to find a uniform bound on $Y_g$ which is not yet conditioned on the level of truncation $k$.

We consider a truncation on $|p(v)-p(\tilde v_\eps)|$ with a  constant $k>0$. Later we will consider the case  $k=\delta_3$ as in Proposition \ref{prop:main3}. But for now, we consider the general case $k$ to estimate the constant $C_2$.  For that, let $\psi_k$ be a continuous function defined by
\beq\label{psi}
 \psi_k(y)=\inf\left(k,\sup(-k,y)\right).
\eeq
We then define the function $\bar{v}_k$ uniquely (since the function $p$ is one to one) as
$$
p(\bar v_k)-\pt=\psi_k(p(v)-\pt).
$$
We have the following lemma.
\begin{lemma}\label{lemmeC2}
For a fixed $v_- > 0$, $u_-\in\bbr$, there exists $C_2, k_0, \eps_0, \delta_0>0$ such that for any $\eps\leq \eps_0$, $\eps/\lambda\leq \delta_0$ with $\lambda<1/2$, the following is true whenever $|Y(U)|\leq \eps^2$:
\begin{eqnarray}
\label{l1}
&& \int_\bbr|a'||h-\tilde{h}_\eps|^2\,d\xi + \int_\bbr|a'| Q(v|\tilde{v}_\eps)\,d\xi\leq C\frac{\eps^2}{\lambda}, \\
\label{lbis}
&& |Y_g(\bar v_k)|\leq C_2\frac{\eps^2}{\lambda}, \qquad \mathrm{for \ every \ } k\leq k_0.
\end{eqnarray}
\end{lemma}
\begin{proof}
$\bullet$ {\it Proof of \eqref{l1}:}
% Let $v^*$ be a constant satisfying $v^*>v_-$. 
We first use \eqref{rel_Q} to estimate 
\begin{align}
\begin{aligned}\label{est-1}
&\int_\bbr  |a' |\eta(U|\tilde U_\eps) d\xi \ge  \int_\bbr  |a'| \frac{|h-\tilde h_\eps |^2}{2} d\xi
 +c_1\int_{v\le 3v_-}|a'|  |v-\tilde v_\eps|^2 + c_2\int_{v> 3v_-}|a'|  |v-\tilde v_\eps|.
\end{aligned}
\end{align}
On the other hand, using $\int_\bbr a'\eta(U|\tilde U_\eps) d\xi=-Y+\int_\bbr a\partial_\xi\nabla\eta(\tilde U_\eps) (U-\tilde U_\eps) d\xi$ in \eqref{badgood}, and $|Y|\le\eps^2$, we have 
\[
\int_\bbr  |a'|\eta(U|\tilde U_\eps) d\xi \le \eps^2 + \int_\bbr |\partial_\xi\nabla\eta(\tilde U_\eps)| |U-\tilde U_\eps| d\xi.
\]
Then, since $ |\partial_\xi\nabla\eta(\tilde U_\eps)| \leq C |\partial_\xi  p(\tilde v_\eps)| =C\frac{\eps}{\lambda} |a'|$ by \eqref{w-real}, we have
\begin{align*}
\begin{aligned}
&\int_\bbr  |a'|\eta(U|\tilde U_\eps) d\xi \\
&\quad\le  \eps^2 +C\frac{\eps}{\lambda}\int_\bbr |a'| |U-\tilde U_\eps| d\xi\\
&\quad\le \eps^2 +C\frac{\eps}{\lambda}\int_{v> 3v_-} |a'| |v-\tilde v_\eps| d\xi\\
&\qquad+C\frac{\eps}{\lambda}\Big(\int_{v\le 3v_-} |a'| |v-\tilde v_\eps|^2 d\xi+\int_{\bbr} |a'| |h-\tilde h_\eps|^2 d\xi\Big)^{1/2}\Big(\int_\bbr |a'| d\xi \Big)^{1/2}.
\end{aligned}
\end{align*}
Since it follows from \eqref{tail} that
\[
\int_\bbr |a'| d\xi=\frac{\lambda}{\eps}\int_\bbr |\partial_\xi  p(\tilde v_\eps)| d\xi \le \frac{\lambda}{\eps} |p'(v_+)| \int_\bbr  |\tilde v_\eps'| d\xi\leq C \lambda,
\]
using Young's inequality, we get  that
\begin{align}
\begin{aligned}\label{est-2}
&\int_\bbr  |a'|\eta(U|\tilde U_\eps) d\xi \\
&\quad\le \eps^2 +C\frac{\eps}{\lambda}\int_{v> 3v_-} |a'| |v-\tilde v_\eps| d\xi+\frac{c_1}{2}\int_{v\le 3v_-} |a'| |v-\tilde v_\eps|^2 d\xi+\frac{1}{2}\int_{\bbr} |a'| |h-\tilde h_\eps|^2 d\xi +C \frac{\eps^2}{\lambda}.
\end{aligned}
\end{align}
Now, taking $\deo\leq1/2$ such that $\frac{\eps}{\lambda}<\deo\leq 1/2$, and then combining the two estimates \eqref{est-1} and \eqref{est-2} together with $\eps^2<C\frac{\eps^2}{\lambda}$, we have 
\begin{align}
\begin{aligned}\label{esth}
\int_\bbr  |a'| \frac{|h-\tilde h_\eps |^2}{2} d\xi
 +\int_{v\le 3v_-}|a'|  |v-\tilde v_\eps|^2 + \int_{v>3 v_-}|a'|  |v-\tilde v_\eps| \leq C \frac{\eps^2}{\lambda}.
\end{aligned}
\end{align}
Applying the above estimate to \eqref{est-2}, we complete \eqref{l1}. 
\vskip0.3cm
$\bullet$ {\it Proof of \eqref{lbis}:} 
First of all, we have
\begin{eqnarray*}
&&|Y_g(\bar v_k)|= \left|-\frac{1}{2\sigma_\eps^2}\int_\bbr a' |p(\bar v_k)-p(\tilde v_\eps)|^2 d\xi -\int_\bbr a' Q(\bar v_k|\tilde v_\eps) d\xi \right.\\
&&\qquad\qquad\qquad\qquad\left.-\int_\bbr a \partial_\xi p(\tilde v_\eps)(\bar v_k-\tilde v_\eps)d\xi+\frac{1}{\sigma_\eps}\int_\bbr a \partial_\xi \tilde h_\eps\big(p(\bar v_k)-p(\tilde v_\eps)\big)d\xi\right |\\
&&\qquad\quad\quad \leq C  \underbrace{\int_\bbr|a'| |p(\bar v_k)-p(\tilde v_\eps)|^2\,d\xi}_{=:I_1} +   \int_\bbr|a'| Q(\bar v_k|\tilde{v}_\eps)\,d\xi\\
&&\qquad\qquad\qquad\qquad +C\underbrace{ \int_\bbr \frac{\eps}{\lambda}|a'|\Big(|\bar v_k-\tilde v_\eps| + |p(\bar v_k)-p(\tilde v_\eps)| \Big)d\xi}_{=:I_2}.
\end{eqnarray*}
Let us fix $k_0=\delta_*/2$ of Lemma \ref{lem:local}. Then, for any $k\leq k_0$, we have $|p(\bar v_k)-p(\tilde{v}_\eps)|\leq k\le \frac{\delta_*}{2}$.\\
 Thus using \eqref{pQ-equi0} with $\eps_0\ll 1$, we have
\[
I_1\le C \int_\bbr|a'| Q(\bar v_k|\tilde{v}_\eps)\,d\xi.
\]
Using \eqref{rel_Q} and \eqref{pQ-equi0}, we have
\begin{align*}
\begin{aligned}
I_2&\le \sqrt{ \int_\bbr\left(\frac{\eps}{\lambda}\right)^2|a'|\,d\xi} \sqrt{ \int_\bbr|a'| \Big(|\bar v_k-\tilde v_\eps|^2 + |p(\bar v_k)-p(\tilde v_\eps)|^2 \Big)\,d\xi}\\
&\le C \sqrt{\frac{\eps^2}{\lambda}}\sqrt{ \int_\bbr|a'| Q(\bar v_k|\tilde{v}_\eps)\,d\xi}.
\end{aligned}
\end{align*}
Notice that since the definition of $\bar v_k$ implies either $\tilde v_\eps\le \bar v_k\le v$ or $v\le\bar v_k\le \tilde v_\eps$, it follows from \eqref{Q-sim} that
\[
Q(v|\tilde{v}_\eps)\ge Q(\bar v_k|\tilde{v}_\eps).
\]
Therefore, using \eqref{l1}, there exists a constant $C_2>0$ such that
\[
|Y_g(\bar v_k)|\le C \int_\bbr|a'| Q(v|\tilde{v}_\eps)\,d\xi +  C \sqrt{\frac{\eps^2}{\lambda}}\sqrt{ \int_\bbr|a'| Q(v|\tilde{v}_\eps)\,d\xi} \le C_2\frac{\eps^2}{\lambda}.
\]

\end{proof}

We now fix the constant $\delta_3$ of Proposition \ref{prop:main3} associated to the constant $C_2$ of Lemma \ref{lemmeC2}. Without loss of generality, we can assume that $\delta_3< k_0$ (since Proposition \ref{prop:main3} is valid for any smaller $\delta_3$). From now on, we set
$$
\bar v:=\bar v_{\delta_3}, \qquad \qquad \bar{U}:=(\bar v, h).
$$
Note that from Lemma \ref{lemmeC2}, we have 
\begin{equation}\label{YC2}
|Y_g(\bar v)|\leq C_2 \frac{\eps^2}{\lambda}.
\end{equation}
We will  use the notations $\mathcal{G}_1, \mathcal{G}_2, \mathcal{D}$ to denote three good terms of $\mathcal{G}$, that is $\mathcal{G}=\mathcal{G}_1+\mathcal{G}_2+\mathcal{D}$ where 
\begin{align}
\begin{aligned}\label{ggd}
&\mathcal{G}_1(U):=\frac{\sigma_\eps}{2}\int_\bbr a'\Big(h-\tilde h_\eps -\frac{p(v)-p(\tilde v_\eps)}{\sigma_\eps}\Big)^2 d\xi,\\
&\mathcal{G}_2(U)=\sigma_\eps  \int_\bbr  a' Q(v|\tilde v_\eps) d\xi,\\
&\mathcal{D}(U)= \int_\bbr a |\partial_\xi (p(v)-p(\tilde v_\eps))|^2 d\xi.
\end{aligned}
\end{align}
We first notice that since $p(\bar v)-p(\tilde v_\eps)$ is constant for $v$ satisfying either $p(v)-p(\tilde v_\eps)<-\delta_3$ or $p( v)-p(\tilde v_\eps)>\delta_3$, we have
$$
\mathcal{D}(\bar U)= \int_\bbr a |\partial_\xi (p(v)-p(\tilde v_\eps))|^2 {\mathbf 1}_{\{|p(v)-\pt |\leq\delta_3\}} d\xi.
$$
We also note that
\begin{align}
\begin{aligned}\label{def-bar}
|p(v)-p(\bar v)|&= |(p(v)-\pt)+(\pt-p(\bar v))|\\[0.2cm]
&=|(\psi_{\delta_3}-I)(p(v)-\pt)|\\[0.2cm]
&= (|p(v)-\pt|-\delta_3)_+.
\end{aligned}
\end{align}
Therefore,
\begin{align}
\begin{aligned}\label{D_bar}
\mathcal{D}(U)&=\int_\bbr a  |\partial_\xi (p(v)-p(\tilde v_\eps))|^2 d\xi\\
&=\int_\bbr a  |\partial_\xi (p(v)-p(\tilde v_\eps))|^2 ( {\mathbf 1}_{\{|p(v)-\pt |\leq\delta_3\}} + {\mathbf 1}_{\{|p(v)-\pt |>\delta_3\}} )d\xi\\
&=\mathcal{D}(\bar U)+\int_\bbr a  |\partial_\xi (p(v)-p(\bar v))|^2 d\xi\\
&\ge \int_\bbr a  |\partial_\xi (p(v)-p(\bar v))|^2 d\xi,
\end{aligned}
\end{align}
which also yields
\begin{equation}\label{eq_D}
\mathcal{D}(U)-\mathcal{D}(\bar U)=\int_\bbr a  |\partial_\xi (p(v)-p(\bar v))|^2 d\xi\geq0.
\end{equation}
On the other hand, since $Q(v|\tilde{v}_\eps)\geq  Q(\bar v|\tilde{v}_\eps)$,
we have
\begin{equation}\label{eq_G}
|\sigma_\eps| \int_\bbr |a'|Q(v|\tilde{v}_\eps)\,d\xi\geq \mathcal{G}_2(U)-\mathcal{G}_2(\bar U)=|\sigma_\eps|\int_\bbr |a'|\left(Q(v|\tilde{v}_\eps)-Q(\bar v|\tilde{v}_\eps)\right)\,d\xi\geq 0.
\end{equation}

We will first show the following lemma.
\begin{lemma}\label{lemma_oublie}
There exist $C, \eps_0, \deo>0$, such that for any $\eps<\eps_0$, $\eps/\lambda<\deo$, and $\lambda<1/2$, the following is true whenever   $|Y(U)|\leq \eps^2$:
%\begin{itemize}
%\item[1]
\begin{eqnarray}
\label{l2}
&&0\leq \mathcal{G}_2(U)-\mathcal{G}_2(\bar U)\leq  \mathcal{G}_2(U)\leq C\frac{\eps^2}{\lambda}, \\
\label{l3}
&&  \int_\bbr|a'| |p(v)-p(\bar v)|^2 \,d\xi+ \int_\bbr|a'| |p(v)-p(\bar v)|\,d\xi\leq C\sqrt{\frac{\eps}{\lambda}}\mathcal{D}(U),\\
\label{l4}
&&  \int_\bbr|a'| \big ||p(v)-\pt|^2-|p(\bar v)-\pt|^2\big| \,d\xi \leq C\sqrt{\frac{\eps}{\lambda}}\mathcal{D}(U),\\
\label{l5}
&&  \int_\bbr|a'| \left|p(v|\tilde{v}_\eps)-p(\bar v|\tilde{v}_\eps)\right | \,d\xi+  \int_\bbr|a'| \left|Q(v|\tilde{v}_\eps)-Q(\bar v|\tilde{v}_\eps)\right | \,d\xi+\int_\bbr|a'| |v-\bar v| \,d\xi\\\
\nonumber
&&\qquad\qquad \leq  C\sqrt{\frac{\eps}{\lambda}}\mathcal{D}(U)
+C(\mathcal{G}_2(U)-\mathcal{G}_2(\bar U)).
\end{eqnarray}
%\item[2] If both $|Y(U)|\leq \eps^2$ and $\mathcal{D}(\bar U)\leq C\eps^2/\lambda$, then:
%\end{itemize}
\end{lemma}
\begin{proof}
We split the proof into several steps.

\noindent{\it Step 1:} The estimate \eqref{l1}  with \eqref{eq_G} gives  \eqref{l2}.

\vskip0.2cm
\noindent{\it Step 2:} Note first that since $(y-\delta_3/2)_+\geq \delta_3/2$ whenever $(y-\delta_3)_+>0$, we have
\beq\label{y-identity}
(y-\delta_3)_+\leq (y-\delta_3/2)_+{\mathbf 1}_{\{y-\delta_3>0\}}\leq (y-\delta_3/2)_+\left(\frac{(y-\delta_3/2)_+}{\delta_3/2}\right)\leq \frac{2}{\delta_3} (y-\delta_3/2)_+^2.
\eeq
Hence, to show \eqref{l3}, it is enough to show it only for the quadratic part, with $\bar v$ defined with $\delta_3/2$ instead of $\delta_3$. We will keep the notation $\bar v$ for this case below.

\vskip0.2cm
\noindent{\it Step 3:} Since $|a'|=(\lambda/\eps)|\tilde{v}'_\eps|$, thanks to \eqref{lower-v} and \eqref{l1}, we get
\begin{eqnarray*}
2\eps\int_{-1/\eps}^{1/\eps} Q(v|\tilde{v}_\eps)\, d\xi&\leq& \frac{2\eps}{\inf_{[-1/\eps,1/\eps]}|a'|}\int_\bbr|a'|Q(v|\tilde{v}_\eps)\,d\xi\\
&\leq& C \frac{\eps}{\lambda\eps}\frac{\eps^2}{\lambda}=C\left(\frac{\eps}{\lambda}\right)^2.
\end{eqnarray*}
Hence, there exists $\xi_0\in [-1/\eps,1/\eps]$ such that $Q(v(\xi_0),\tilde{v}_\eps(\xi_0))\leq C(\eps/\lambda)^2$. For $\deo$ small enough, and using  \eqref{pQ-equi0}, we have 
$$
|(p(v)-\pt)(\xi_0)|\leq C\frac{\eps}{\lambda}.
$$
Thus, if $\deo$ is small enough such that $C\eps/\lambda\leq \delta_3/2$, then we have from \eqref{def-bar} that
$$
(p(v)-p(\bar v))(\xi_0)=0.
$$
Therefore using \eqref{D_bar}, we obtain that for any $\xi\in \bbr$,
\begin{align}
\begin{aligned}\label{beta1}
|(p(v)-p(\bar v))(\xi)|&=\left|\int_{\xi_0}^\xi \partial_\zeta (p(v)-p(\bar v))\,d\zeta\right|\\
&\le \sqrt{|\xi|+|\xi_0|} \sqrt{  \int_\bbr \Big|\partial_\zeta (p(v)-p(\bar v)) \Big|^2 \,d\zeta }\\
&\leq C\sqrt{|\xi|+\frac{1}{\eps}}\sqrt{\mathcal{D}(U)}.
\end{aligned}
\end{align}
For any $\xi$ such that $|(p(v)-p(\bar v))(\xi)|>0$, we have from  \eqref{def-bar} that $|(p(v)-\pt)(\xi)|> \delta_3$. Thus using \eqref{pressure2} and \eqref{rel_Q}, we have $Q(v(\xi)|\tilde{v}_\eps(\xi))\geq \alpha$, for some constant $\alpha>0$ depending only on $\delta_3$. Hence
\begin{equation}\label{beta2}
{\mathbf 1}_{\{|p(v)-p(\bar v)|>0\}}\leq \frac{Q(v|\tilde{v}_\eps)}{\alpha}.
\end{equation}
In the next computation, we split the integral in two parts, and use \eqref{beta1}-\eqref{beta2} to have 
\begin{eqnarray*}
&& \int_\bbr|a'||p(v)-p(\bar v)|^2\,d\xi\leq \int_{-\frac{1}{\eps}\sqrt{\frac{\lambda}{\eps}}}^{\frac{1}{\eps}\sqrt{\frac{\lambda}{\eps}}}|a'||p(v)-p(\bar v)|^2\,d\xi+ \int_{|\xi|\geq\frac{1}{\eps}\sqrt{\frac{\lambda}{\eps}}}|a'||p(v)-p(\bar v)|^2\,d\xi\\
&&\qquad\leq\left(\sup_{\left[-\sqrt{\frac{\lambda}{\eps^3}},\sqrt{\frac{\lambda}{\eps^3}}\right]}|p(v)-p(\bar v)|^2\right)
\int_{-\frac{1}{\eps}\sqrt{\frac{\lambda}{\eps}}}^{\frac{1}{\eps}\sqrt{\frac{\lambda}{\eps}}}|a'|{\mathbf 1}_{\{|p(v)-p(\bar v)|>0\}}\,d\xi\\
&&\qquad\qquad\qquad+C\mathcal{D}(U) \int_{|\xi|\geq\frac{1}{\eps}\sqrt{\frac{\lambda}{\eps}}}|a'| \left(|\xi|+\frac{1}{\eps}\right)\,d\xi\\
&& \qquad \leq C \mathcal{D}(U) \left(\sqrt{\frac{\lambda}{\eps^3}}
\int_\bbr|a'|\frac{Q(v|\tilde{v}_\eps)}{\alpha}\,d\xi+2\int_{|\xi|\geq\frac{1}{\eps}\sqrt{\frac{\lambda}{\eps}}}|a'| |\xi|\,d\xi\right).
\end{eqnarray*}
Therefore we have
$$
\int_\bbr|a'||p(v)-p(\bar v)|^2\,d\xi\leq C\sqrt{\frac{\eps}{\lambda}} \mathcal{D}(U).
$$
Indeed, using \eqref{l1} and \eqref{tail} (recalling $|a'|=(\lambda/\eps)|\tilde{v}'_\eps|$), we have
\[
 \int_\bbr|a'||p(v)-p(\bar v)|^2\,d\xi \leq  C \mathcal{D}(U)\left(\sqrt{\frac{\eps}{\lambda}} +\lambda \eps
\int_{|\xi|\geq\frac{1}{\eps}\sqrt{\frac{\lambda}{\eps}}}e^{-c\eps|\xi|} |\xi|\,d\xi \right),
\]
and for the last term, we take $\deo$ small enough such that for any $\eps/\lambda\leq \deo$,
$$
\lambda \eps\int_{|\xi|\geq\frac{1}{\eps}\sqrt{\frac{\lambda}{\eps}}}e^{-c\eps|\xi|} |\xi|\,d\xi =\frac{\lambda}{\eps} \int_{|\xi|\geq \sqrt{\frac{\lambda}{\eps}}}  e^{-c|\xi|} |\xi| d\xi \le \frac{\lambda}{\eps} \int_{|\xi|\geq \sqrt{\frac{\lambda}{\eps}}} e^{-\frac{c}{2}|\xi|}d\xi=\frac{2\lambda}{c\eps}e^{-\frac{c}{2}\sqrt{\frac{\lambda}{\eps}}}\leq \sqrt{\frac{\eps}{\lambda}}.
$$

As mentioned in Step 2, recall that $\bar v=\bar v_{\delta_3/2}$ in the above estimate. Then using \eqref{def-bar}, we have
\begin{align*}
\begin{aligned}
\int_\bbr|a'||p(v)-p(\bar v_{\delta_3})|^2\,d\xi &=\int_\bbr|a'|(|p(v)-\pt|-\delta_3)_+^2\,d\xi \\
&\le \int_\bbr|a'|(|p(v)-\pt|-\delta_3/2)_+^2\,d\xi\\
&=\int_\bbr|a'||p(v)-p(\bar v_{\delta_3/2})|^2\,d\xi \leq C  \mathcal{D}(U)\sqrt{\frac{\eps}{\lambda}}.
\end{aligned}
\end{align*}
Likewise, using \eqref{def-bar} and \eqref{y-identity} with $y:=|p(v)-p(\tilde v_\eps)|$, we have
\[
\int_\bbr|a'||p(v)-p(\bar v_{\delta_3})|\,d\xi \le \frac{2}{\delta_3}\int_\bbr|a'||p(v)-p(\bar v_{\delta_3/2})|^2 \,d\xi \leq C \mathcal{D}(U) \sqrt{\frac{\eps}{\lambda}}.
\]
Hence, we obtain \eqref{l3}.

\vskip0.2cm
\noindent{\it Step 4:} %Whenever $|p(\bar v)-p(v)|>0$ we have $|p(\xi)-\pt(\xi)|\geq \delta_3$,
%Hence
%$$
%{\mathbf 1}_{\{|p(v)-p(\bar v)|>0\}}\leq \frac{|p(v)-\pt|}{\delta_3}.
%$$
We use $|p(\bar v)-\pt|\leq \delta_3$ and  \eqref{l3} to show
\begin{eqnarray*}
&&\int_\bbr |a'| \left| |p(v)-\pt|^2 -|p(\bar v) -\pt|^2  \right|\,d\xi\\
&&\qquad \leq \int_\bbr |a'| |p( v)-p(\bar v)| |p(v)+p(\bar v)-2p(\tilde v_\eps)|\,d\xi\\
&&\qquad\leq \int_\bbr |a'| |p(v)-p(\bar v)|
\left( |p(v)-p(\bar v)|+2|p(\bar v) -\pt|    \right)\,d\xi\\
&&\qquad\leq  \int_\bbr |a'| \left( |p(v)-p(\bar v)|^2+2\delta_3 |p(v)-p(\bar v)|             \right)\,d\xi\\
&&\qquad\leq C \mathcal{D}(U) \sqrt{\frac{\eps}{\lambda}},
\end{eqnarray*}
which gives \eqref{l4}.

\vskip0.2cm
\noindent{\it Step 5:} 
First, since $\tilde{v}_\eps\in [v_-/2, v_-]$ for $\eps_0$ small enough, it follows from the definition of the relative pressure \eqref{pressure-relative} that
$$
|p(v|\tilde{v}_\eps)-p(\bar v|\tilde{v}_\eps)|=|(p(v)-p(\bar v))-p'(\tilde{v}_\eps)(v-\bar v)|\leq |p(v)-p(\bar v)|+C|v-\bar v|.
$$
Thus,
\begin{eqnarray*}
&& \int_\bbr|a'| \left|p(v|\tilde{v}_\eps)-p(\bar v|\tilde{v}_\eps)\right | \,d\xi+\int_\bbr|a'| |v-\bar v| \,d\xi\\
&&\leq C\int_\bbr |a'| |p(v)-p(\bar v) |\,d\xi+C\int_\bbr|a'| |v-\bar v| \,d\xi.
\end{eqnarray*}
To control the last term above, we use \eqref{rel_Q1} as follows: If $|v-\bar v|>0$, we have from the definition of $\bar v$ that $|p(\bar v)-p(\tilde v_\eps)|=\delta_3$. Then using \eqref{pressure2}, we find 
\[|\bar v-\tilde v_\eps|\ge \min(c_3^{-1} \delta_3, v_-/2-\eps_0).\]
Taking $\delta_*$ in $2)$ of Lemma \ref{lem-pro} such that $\eps_0\leq \delta_*/2$ and $\min(c_3^{-1} \delta_3, v_-/2-\eps_0)\ge \delta_*$, we use \eqref{rel_Q1} with $w=\tilde{v}_\eps$, $u=\bar v$ and $v=v$ to find that there exists a constant $C>0$ such that 
$$        
C|v-\bar v|\leq Q(v|\tilde{v}_\eps)- Q(\bar v|\tilde{v}_\eps).
$$
Therefore, using \eqref{l3} and \eqref{eq_G} , we find
\begin{eqnarray*}
&& \int_\bbr|a'| \left|p(v|\tilde{v}_\eps)-p(\bar v|\tilde{v}_\eps)\right | \,d\xi+\int_\bbr|a'| |v-\bar v| \,d\xi\\
&&\qquad \leq C\int_\bbr |a'| |p(v)-p(\bar v) |\,d\xi+C\int_\bbr|a'| \big(Q(v|\tilde{v}_\eps)-Q(\bar v|\tilde{v}_\eps)\big) \,d\xi\\
&&\qquad\leq C \mathcal{D}(U) \sqrt{\frac{\eps}{\lambda}}+C[\mathcal{G}_2(U)-\mathcal{G}_2(\bar U)].
\end{eqnarray*}
This together with \eqref{eq_G} completes the proof of \eqref{l5}. 
\end{proof}

\vspace{0.5cm}
We first recall $Y$ in \eqref{badgood} as
\begin{align*}
\begin{aligned}
Y= -\int_\bbr a' \frac{|h-\tilde h_\eps|^2}{2} d\xi -\int_\bbr a' Q(v|\tilde v_\eps) d\xi +\int_\bbr a \Big(-\partial_\xi p(\tilde v_\eps)(v-\tilde v_\eps) +\partial_\xi \tilde h_\eps(h-\tilde h_\eps) \Big) d\xi.
\end{aligned}
\end{align*}
We split $Y$ into three parts $Y_g$,  $Y_b$ and $Y_l$ as below: $Y_g$ consists of the terms related to $v-\tilde v_\eps$,  while $Y_b$ and $Y_l$ consist of   terms related to $h-\tilde h_\eps$. While $Y_b$ is quadratic in $U$, the term $Y_l$ is linear in $h-\tilde h_\eps$. More precisely, $Y$ can be decomposed as
\[
Y=Y_g+Y_b+Y_l,
\]
where
\begin{align*}
\begin{aligned}
Y_g&:= -\frac{1}{2\sigma_\eps^2}\int_\bbr a' |p(v)-p(\tilde v_\eps)|^2 d\xi -\int_\bbr a' Q(v|\tilde v_\eps) d\xi -\int_\bbr a \partial_\xi p(\tilde v_\eps)(v-\tilde v_\eps)d\xi\\
&\quad+\frac{1}{\sigma_\eps}\int_\bbr a \partial_\xi \tilde h_\eps\big(p(v)-p(\tilde v_\eps)\big)d\xi,
\end{aligned}
\end{align*}
\begin{align*}
\begin{aligned}
Y_b&:= -\frac{1}{2}\int_\bbr a' \Big(h-\tilde h_\eps-\frac{p(v)-p(\tilde v_\eps)}{\sigma_\eps}\Big)^2 d\xi -\frac{1}{\sigma_\eps}\int_\bbr a' \big(p(v)-p(\tilde v_\eps)\big)\Big(h-\tilde h_\eps-\frac{p(v)-p(\tilde v_\eps)}{\sigma_\eps}\Big) d\xi.
\end{aligned}
\end{align*}
\begin{align*}
\begin{aligned}
Y_l&=\int_\bbr a \partial_\xi \tilde h_\eps\Big(h-\tilde h_\eps-\frac{p(v)-p(\tilde v_\eps)}{\sigma_\eps}\Big)d\xi.
\end{aligned}
\end{align*}
Notice that the first part $Y_g$ is independent of $h$, and $Y_g(\bar U)$ was used to absorb the bad term $\mathcal{B}$ in Proposition \ref{prop:main3}, while $Y_b$ and $Y_l$ are useless because $\mathcal{B}$ does not depend on $h-\tilde h_\eps$. Therefore we need show that $Y_g(U)-Y_g(\bar U)$, $Y_b(U)$ and $Y_l(U)$ are negligible by other terms. 
We now prove the following lemma.

\begin{lemma}\label{lemma3}
There exist constants $\deu, \eps_0, C, C^*>0$ such that for any $\eps<\eps_0$, and any $\lambda<1/2$ with $\eps/\lambda<\deu$, the following statements hold true.
\begin{itemize}
\item[1] For any $U$ such that $|Y(U)|\leq \eps^2$,
\begin{eqnarray}
\label{n1}
&&|\mathcal{B}(U)-\mathcal{B}(\bar U)|\leq C \sqrt{\frac{\eps}{\lambda}}  \mathcal{D}(U)+C\frac{\eps}{\lambda}[\mathcal{G}_2(U)-\mathcal{G}_2(\bar U)],\\
\label{n2}
&&|\mathcal{B}(U)|\leq C^*\frac{\eps^2}{\lambda} +C\sqrt{\frac{\eps}{\lambda}} \mathcal{D}(U).
\end{eqnarray}
\item[2] For any $U$ such that $|Y(U)|\leq \eps^2$ and $\mathcal{D}(U)\leq \frac{C^*}{4}\frac{\eps^2}{\lambda}$,
\begin{eqnarray}
\label{m1}
&&|Y_g(U)-Y_g(\bar U)|+|Y_b(U)|\leq C\frac{\eps^2}{\lambda},\\
\label{m2}
&&|Y_g(U)-Y_g(\bar U)|+|Y_b(U)|\leq C \sqrt{\frac{\eps}{\lambda}} \mathcal{D}(U) +C [\mathcal{G}_2(U)-\mathcal{G}_2(\bar U)]\\
%\label{m3}
\nonumber
&&\qquad\qquad\qquad\qquad\qquad\qquad+\left(\frac{\lambda}{\eps}\right)^{1/4}\mathcal{G}_1(U)+C\left(\frac{\eps}{\lambda}\right)^{1/4}\mathcal{G}_2(\bar U),\\
\label{m4}
&&|Y_l(U)|^2\leq \frac{\eps^2}{\lambda}\mathcal{G}_1(U).
\end{eqnarray}
\end{itemize}
\end{lemma} 
\begin{proof}
We split the proof in several steps.
\vskip0.2cm
\noindent{\it Step 1:} Recall the bad term $\mathcal B$ in \eqref{badgood}. Using \eqref{p-est1}, \eqref{pQ-equi0} and $|\tilde{v}_\eps'|+|a''|\leq C|a'|$, and then \eqref{l1}, we have
\beq\label{B-bar}
|\mathcal{B}(\bar U)| \leq C\int_\bbr |a'| Q(\bar v|\tilde{v}_\eps)\,d\xi\leq C\int_\bbr |a'| Q(v|\tilde{v}_\eps)\,d\xi\leq C^*\frac{\eps^2}{\lambda}.
\eeq
Moreover, using \eqref{l4} and \eqref{l5} together with $|\tilde{v}_\eps'|\le C\frac{\eps}{\lambda} |a'|$, we have 
\[
|\mathcal{B}(U)-\mathcal{B}(\bar U) |\leq  C \mathcal{D}(U)\sqrt{\frac{\eps}{\lambda}}+C\frac{\eps}{\lambda}[\mathcal{G}_2(U)-\mathcal{G}_2(\bar U)].
\]
Combining the above two estimates together with \eqref{l2}, we have \eqref{n2}.
\vskip0.2cm

\noindent{\it Step 2:} We show \eqref{m1} as follows: Using \eqref{l3}-\eqref{l5}, we have
\begin{align}
\begin{aligned}\label{Yg1}
|Y_g(U)-Y_g(\bar U)| &\leq C \int_\bbr |a'| \Big(\big| |p(v)-\pt|^2-|p(\bar v)-\pt|^2\big| +\big|Q(v|\tilde{v}_\eps)-Q(\bar v|\tilde{v}_\eps)\big| \\
&\qquad\quad +|v-\bar v|+ |p(v)-p(\bar v)|\Big)\,d\xi\\
&\leq  C\sqrt{\frac{\eps}{\lambda}}\mathcal{D}(U)
+C(\mathcal{G}_2(U)-\mathcal{G}_2(\bar U)).
\end{aligned}
\end{align}
Then using \eqref{l2} and $\mathcal{D}(U)-\mathcal{D}(\bar U)\leq \mathcal{D}(U)\leq C\eps^2/\lambda$, we have $|Y_g(U)-Y_g(\bar U)| \le C\eps^2/\lambda$.\\
Next, recalling $\mathcal{G}_1$ in \eqref{ggd}, we have
\[
|Y_b(U)|\le C \mathcal{G}_1(U)+ C\int_\bbr |a'| |p(v)-\pt|^2 \,d\xi \le C(\mathcal{G}_1(U)+|\mathcal{B}(U)|).
\]
Since
$$
 \mathcal{G}_1(U)\leq \int_\bbr |a'| \left(|h-\tilde h_\eps|^2+|p(v)-\pt|^2\right)\,d\xi,
$$
using \eqref{l1} and \eqref{n2}, we have $|Y_b(U)|\le C\eps^2/\lambda$. 
\vskip0.2cm
\noindent{\it Step 3:} We first estimate the term $\int a' (p(v)-\pt)(h-\tilde{h}_\eps+(p(v)-\pt)/\sigma) d\xi$ in $Y_b$ using Young's inequality with $\eps/\lambda$ as follows: 
\begin{eqnarray*}
&&\Big|\int a' (p(v)-\pt)\left(h-\tilde{h}_\eps+(p(v)-\pt)/\sigma\right) d\xi \Big|\\
&&\qquad\le\left(\frac{\lambda}{\eps}\right)^{1/4}\mathcal{G}_1(U)+C\left(\frac{\eps}{\lambda}\right)^{1/4}\int_\bbr|a'| |p(v)-\pt|^2\,d\xi.
\end{eqnarray*}
Since \eqref{n1} and the first inequality in \eqref{B-bar} yield
\begin{align*}
\begin{aligned}
\int_\bbr|a'| |p(v)-\pt|^2\,d\xi &\leq |\mathcal{B}(U)-\mathcal{B}(\bar U) |+|\mathcal{B}(\bar U) |\\
&\le C\sqrt{\frac{\eps}{\lambda}} \mathcal{D}(U)+C\frac{\eps}{\lambda}[\mathcal{G}_2(U)-\mathcal{G}_2(\bar U)] +C\mathcal{G}_2(\bar U),
\end{aligned}
\end{align*}
we have
\begin{eqnarray*}
&&\Big|\int a' (p(v)-\pt)\left(h-\tilde{h}_\eps+(p(v)-\pt)/\sigma\right) d\xi \Big|\\
&&\qquad\le\left(\frac{\lambda}{\eps}\right)^{1/4}\mathcal{G}_1(U)+C\sqrt{\frac{\eps}{\lambda}} \mathcal{D}(U)+C\frac{\eps}{\lambda}[\mathcal{G}_2(U)-\mathcal{G}_2(\bar U)] +C\left(\frac{\eps}{\lambda}\right)^{1/4}\mathcal{G}_2(\bar U).
\end{eqnarray*}
Therefore, this estimate together with $\lambda/\eps>\delta_0^{-1}\gg1$ and \eqref{Yg1} implies
\begin{eqnarray*}
&&\qquad\qquad\qquad|Y_g(U)-Y_g(\bar U)|+|Y_b(U)|\\
&&\leq  C\sqrt{\frac{\eps}{\lambda}}  \mathcal{D}(U)+C[\mathcal{G}_2(U)-\mathcal{G}_2(\bar U)]+\left(\frac{\lambda}{\eps}\right)^{1/4}\mathcal{G}_1(U)+C\left(\frac{\eps}{\lambda}\right)^{1/4}\mathcal{G}_2(\bar U),
\end{eqnarray*}
which proves \eqref{m2}.

\vskip0.2cm
\noindent{\it Step 4:} Using Cauchy-Schwartz inequality together with $|\tilde{h}_\eps'|\le C\frac{\eps}{\lambda} |a'|$, we find
$$
|Y_l(U)|^2\leq \left(\frac{\eps}{\lambda}\right)^2 \left[\int_\bbr|a'|\,d\xi\right]\int_\bbr|a'| |h-\tilde{h}_\eps+(p(v)-\pt)/\sigma|^2\,d\xi \leq C\frac{\eps^2}{\lambda} \mathcal{G}_1(U),
$$ 
which gives \eqref{m4}.
\end{proof}

\subsection{Proof of Proposition \ref{prop:main}}
We now prove the main Proposition of the paper. We split the proof into two steps, depending on the strength of the dissipation term $\mathcal{D}(U)$.

\vskip0.2cm
\noindent{\it Step 1:} We first consider the case where 
$
\mathcal{D}(U)\geq 4 C^* \frac{\eps^2}{\lambda}, $  where the constant $C^*$ is defined as in Lemma \ref{lemma3}.
Then using $\eqref{n2}$, we find that for $\deu$ small enough,
\begin{align*}
\begin{aligned}
\mathcal{R}(U):=-\frac{|Y(U)|^2}{\eps^4}+\left(1+\delta_0\frac{\eps}{\lambda}\right)|\mathcal{B}(U)|-\mathcal{G}(U)&\leq 2|\mathcal{B}(U)|-\mathcal{D}(U)\\
& \leq 2C^*\frac{\eps^2}{\lambda}+\left(2C\sqrt{\frac{\eps}{\lambda}}-1\right)\mathcal{D}(U)\\
& \leq 2 C^* \frac{\eps^2}{\lambda}-\frac{1}{2}\mathcal{D}(U)\leq 0,
\end{aligned}
\end{align*}
which gives the desired result.

\vskip0.2cm
\noindent{\it Step 2:}  We now assume the other alternative, i.e., $\mathcal{D}(U)\leq 4 C^* \frac{\eps^2}{\lambda}.$ \\
We will use Proposition \ref{prop:main3} to get the desired result. First of all, we have \eqref{YC2}, and for the small constant $\delta_3$ of Proposition \ref{prop:main3} associated to the constant $C_2$ of \eqref{YC2}, we have $|p(\bar v)-\pt|\leq \delta_3.$\\
Let us take $\delta_0$ small enough such that $\delta_0\le\delta_3^8$. Using
$$
Y_g(\bar U)=Y(U)-(Y_g(U)-Y_g(\bar U))-Y_b(U)-Y_l(U),
$$
we have
$$
|Y_g(\bar U)|^2\leq 4|Y(U)|^2+4|Y_g(U)-Y_g(\bar U)|^2+4 |Y_b(U)|^2+4|Y_l(U)|^2,
$$
which can be written as
$$
-4|Y(U)|^2\leq -|Y_g(\bar U)|^2+4|Y_g(U)-Y_g(\bar U)|^2+ 4|Y_b(U)|^2+4|Y_l(U)|^2.
$$
Thus we find that for any $\eps<\eps_0(\le\delta_3)$ and $\eps/\lambda<\delta_0$,
\begin{eqnarray*}
&&\mathcal{R}(U)\le-\frac{4|Y(U)|^2}{\eps\delta_3}+\left(1+\delta_0\frac{\eps}{\lambda}\right)|\mathcal{B}(U)|-\mathcal{G}(U)\\
&&\quad\leq -\frac{|Y_g(\bar U)|^2}{\eps\delta_3}+\left(1+\delta_0\frac{\eps}{\lambda}\right)|\mathcal{B}(\bar U)|-\mathcal{G}_2(\bar U)-(1-\delta_3)\mathcal{D} (U)\\
&&\qquad\qquad +\frac{4}{\eps\delta_3}|Y_g(U)-Y_g(\bar U)|^2+\frac{4}{\eps\delta_3}|Y_b(U)|^2+\frac{4}{\eps\delta_3}|Y_l(U)|^2\\
&&\qquad\qquad+\left(1+\delta_0\frac{\eps}{\lambda}\right)|\mathcal{B}(U)-\mathcal{B}(\bar U)|-(\mathcal{G}_2(U)-\mathcal{G}_2(\bar U)) -\mathcal{G}_1(U)- \delta_3\mathcal{D}(U).
\end{eqnarray*}
To control the square of $|Y_g(U)-Y_g(\bar U)|+|Y_b(U)|$, we multiply the bound of \eqref{m1} and the bound of \eqref{m2} to find
\begin{eqnarray*}
&&\qquad\qquad\qquad\frac{1}{\eps\delta_3}|Y_g(U)-Y_g(\bar U)|^2+\frac{1}{\eps\delta_3}|Y_b(U))|^2\\
&&\leq \frac{C}{\delta_3}\bigg[\left(\frac{\eps}{\lambda}\right)^{3/2} \mathcal{D}(U)+ \frac{\eps}{\lambda}(\mathcal{G}_2(U)-\mathcal{G}_2(\bar U))+ \left(\frac{\eps}{\lambda}\right)^{3/4}\mathcal{G}_1(U)+  \left(\frac{\eps}{\lambda}\right)^{1/4}\frac{\eps}{\lambda}\mathcal{G}_2(\bar U)\bigg]\\
&&\leq C\delta_0^{1/8}\bigg[  \mathcal{D}(U) +(\mathcal{G}_2(U)-\mathcal{G}_2(\bar U))+  \mathcal{G}_1(U)+\frac{\eps}{\lambda}\mathcal{G}_2(\bar U)\bigg].
\end{eqnarray*}
%In the same way, we find
%$$
%\frac{1}{\eps\delta_0}|Y_b(U))|^2\leq C\left(\frac{\eps}{\lambda}\right)^{3/4}\mathcal{G}_1(U)+C \left(\frac{\eps}{\lambda}\right)^{1/4}\frac{\eps}{\lambda}\mathcal{G}_2(\bar U)%+C\left(\frac{\eps}{\lambda}\right)^{3/2}\mathcal{D}(U)
%+\left(\frac{\eps}{\lambda}\right)^{3/2}[\mathcal{G}_2(U)-\mathcal{G}_2(\bar U)].
%$$
Using also \eqref{m4} and \eqref{n1} together with \eqref{eq_D}, therefore we find that for $\deu$ small enough with $\deu\le\delta_3^8$, 
\beq\label{superfinal}
\mathcal{R}(U)\leq  -\frac{|Y_g(\bar U)|^2}{\eps\delta_3}+\left(1+\delta_3\frac{\eps}{\lambda}\right)|\mathcal{B}(\bar U)|-\left(1-\delta_3\frac{\eps}{\lambda}\right)\mathcal{G}_2(\bar U)-(1-\delta_3)\mathcal{D}(\bar U).
\eeq
Since the above quantities $Y_g(\bar U)$, $\mathcal{B}(\bar U)=\mathcal{B}_1(\bar U)+\mathcal{B}_2(\bar U)$, $\mathcal{G}_2(\bar U)$ and $\mathcal{D}(\bar U)$ depends only on $\bar v$ through $\bar U$, it follows from Proposition \ref{prop:main3} that $\mathcal{R}(U)\le 0$. Hence we complete the proof of Proposition \ref{prop:main}.\\

%12

\begin{appendix}
\setcounter{equation}{0}

\section{Proof of Lemma \ref{lem-alge}}\label{app-1}
We show the following lemma which contains Lemma \ref{lem-alge}.
\begin{lemma}
Let 
\[
g(x):=2x-2x^2-\frac{4}{3}x^3 +\frac{4\theta}{3} \Big(-x^2 -2x \Big)^{3/2},
\]
where $\theta =\sqrt{5-\frac{\pi^2}{3}}\approx 1.308$. The following statements are true.
\begin{itemize}
\item[1]  For any $x\in  [-2,-\frac{1+\sqrt{3}}{2}]$, $g''(x)>0$.
\item[2] For any $x\in (-\frac{1+\sqrt{3}}{2},-1]$, $g'(x)>0$. 
\item[3] The function $g'$ has exactly two roots $x_1$ and $x_2$ on $[-1,0]$. The smaller one $x_1$ belongs to $(-1+\sqrt{2}/2,-1+\sqrt{3}/2)$, and is the only local maximum of $g$ on $(-1,0)$.
%\item[4] The function $g$ does not have any nonnegative local maximum on $(-1,0)$. 
\item[4] The function $g$ is negative on $(-2,0)$.  
\end{itemize}
\end{lemma}
The point 4 is the result of Lemma  \ref{lem-alge}.
\begin{proof}
\noindent{\bf Step 1.}  Note that 
$$
-x^2-2x=1-(1+x)^2.
$$
This function is increasing on $(-2,-1)$. So,
for  $-2\leq x\le -\frac{1+\sqrt{3}}{2}$ we have 
\begin{equation}\label{eq:calc}
1-(1+x)^2\le 1-\frac{1}{4}(1-\sqrt{3})^2=1-\frac{1}{4}(1+3-2\sqrt{3}) =\frac{\sqrt{3}}{2}.
\end{equation}
Then we have
\begin{align*}
\begin{aligned}
\frac{g'(x)}{2} &= 1-2x-2x^2-2\theta(1+x)\sqrt{1-(1+x)^2},\\
\frac{g''(x)}{2} &= -2-4x -4\theta \sqrt{1-(x+1)^2}+\frac{2\theta}{\sqrt{1-(x+1)^2}}.
%&\ge -2-4x -4\theta \sqrt{ \frac{\sqrt{3}}{2}} + 2\theta \sqrt{\frac{2}{\sqrt{3}}}.
\end{aligned}
\end{align*}
So, thanks to (\ref{eq:calc}), if $-2\leq x\le -\frac{1+\sqrt{3}}{2}$:
$$
\frac{g''(x)}{2}\ge -2-4x -4\theta \sqrt{ \frac{\sqrt{3}}{2}} + 2\theta \sqrt{\frac{2}{\sqrt{3}}}.
$$
But we have
\[
-4\theta \sqrt{ \frac{\sqrt{3}}{2}} + 2\theta \sqrt{\frac{2}{\sqrt{3}}} \approx -2.06>-2.1,\qquad \mathrm{and \ \ } -\frac{1+\sqrt{3}}{2}<-\frac{4.1}{4}.
\]
%and $-\frac{1+\sqrt{3}}{2}<-\frac{4.1}{4}$, we have
Therefore:
$$
\frac{g''(x)}{2} > -4.1-4x>0,\quad\mathrm{whenever \ \  } -2\leq x\le -\frac{1+\sqrt{3}}{2}.
$$
This proves the point 1 of the lemma.
\vskip0.2cm 
\noindent{\bf Step 2.}  We have 
\beq\label{gh}
g'(x)= \underbrace{2-4x-4x^2}_{=:h_1(x)} -\underbrace{4\theta (x+1)\sqrt{1-(1+x)^2}}_{=:h_2(x)}.
\eeq
Note that $-\frac{1+\sqrt{3}}{2}$  and  $-\frac{1-\sqrt{3}}{2} (>-1)$ are the two roots of $h_1$. Therefore $h_1>0$ on $(-\frac{1+\sqrt{3}}{2},-1]$.  The function $x+1$ is non-positive on this interval, so we have also $h_2\le 0$ on the same interval. Therfore  $g'>0$ on that interval. This proves the point 2.
\vskip0.2cm 
\noindent{\bf Step 3.} For any root $x$ of $g'$, 
$$
P(x):=(h_1(x))^2-(h_2(x))^2=0.
$$
Note that $P$ is a polynomial of order 4, so it has at most 4 roots. 
Using special roots of $h_1$ and $h_2$, we find that 
$$
P(-2)=(h_1(-2))^2>0, \qquad P\left(\frac{-1+\sqrt{3}}{2}\right)=-(h_2)^2<0,\qquad P(-1)=(h_1(-1))^2>0.
$$
Hence $P$ has at least two roots on $(-2,-1)$. Therefore $P$ (and $g'$) cannot have more than 2 roots on $[-1,0]$. However:
$$
g'(-1+\sqrt{2}/2)=2(\sqrt{2}-\theta)>0, \qquad g'(-1+\sqrt{3}/2)=(2-3\theta)\sqrt{3}-1<0, \qquad g'(0)=2>0.
$$
So $g'$ has exactly two roots in $[-1,0]$. One  root $x_1$ is in $(-1+\sqrt{2}/2, -1+\sqrt{3}/2)$ and the other  root $x_2$ is in $(-1+\sqrt{3}/2,0)$. Moreover, $g$ is increasing on $(-1,x_1)$ and on $(x_2,0)$, and decreasing on $(x_1,x_2)$. Hence, $g$ has  a local maximum at $x_1$ and a  a  local minimum at $x_2$. 
\vskip0.2cm 
\noindent{\bf Step 4.}  The function $g$ is continuous on $[-2,0]$, so it attains its maximum on this interval. Assume that this maximum is reached at $x_*\in(-2,0)$. At this point it verifies both 
$g'(x_*)=0$ and $g''(x_*)\leq0$. From Steps 1 and 2, we have $x_*\in (-1,0)$. But from Step 3, we have $x_*=x_1\in (-1+\sqrt{2}/2,-1+\sqrt{3}/2)$. 

Let us consider
\begin{eqnarray*}
&&h_1'(x)=4-8(1+x),\\
&&\sqrt{1-(1+x)^2}h_2'(x)=4\theta(1-2(1+x)^2).
\end{eqnarray*}
We see that these functions are decreasing on $(-1+\sqrt{2}/2,-1+\sqrt{3}/2)$, and non-positive  at $-1+\sqrt{2}/2$, 
that is, 
$$
\mathrm{for \ \ }  x\in(-1+\sqrt{2}/2,-1+\sqrt{3}/2), \ \ h_i'(x)\leq 0 \quad \mathrm{and}\quad h''_i(x)\leq0,\quad\mbox{for }\\ i=1,2. 
$$
Since $g(0)=0$, and $g(x_1)$ is supposed to be a global maximum, we have $g(x_1)\geq0$ and 
\[
I=\int_{-1+\frac{\sqrt{2}}{2}}^{x_1} g'(y)\,dy =g(x_1)-g(-1+\frac{\sqrt{2}}{2})\geq -g(-1+\frac{\sqrt{2}}{2})>0.107.
\]
But using the monotonicity of $h'_1$ and $h'_2$,  and $h'(x_1)=h_2'(x_1)$ (since $g'(x_1)=0$), we have 
\begin{eqnarray*}
&&I=\int_{-1+\frac{\sqrt{2}}{2}}^{x_1} (h'_1(y)-h'_2(y))\,dy\leq (x_1-(-1+\sqrt{2}/2))(h'_1(-1+\sqrt{2}/2)-h'_2(x_1))\\
&&\qquad =(x_1-(-1+\sqrt{2}/2))(h'_1(-1+\sqrt{2}/2)-h'_1(x_1))\\
&&\qquad \leq \frac{\sqrt{3}-\sqrt{2}}{2} (h'_1(-1+\sqrt{2}/2)-h'_1(-1+\sqrt{3}/2)). 
\end{eqnarray*}
 Since
\[
\frac{\sqrt{3}-\sqrt{2}}{2}<0.2,\quad h_1(-1+\frac{\sqrt{2}}{2})-h_1(-2+2\frac{\sqrt{3}}{2})=2\sqrt{2}-2(\sqrt{3}-\frac{1}{2})<0.4,
\]
we have
\[
I\le 0.08,
\]
which contradicts with $I>0.107$. Hence $g$ reaches his maximum only at 0 or -2. Since $g(-2)=-4/3$, and $g(0)=0$, therefore
$$
g(x)<0 \qquad \mathrm{for \ every } \ x\in [-2,0).
$$
\end{proof}

\section{Proof of Lemma \ref{lem-poin}}\label{app-2}
Let $\{P_n:[-1,1]\to\bbr\}_{n\ge0}$ be an orthonormal basis of the Legendre polynomials, that are solutions to Legendre's differential equations:
\beq\label{Leq}
\frac{d}{dx}\Big((1-x^2)\frac{d}{dx}P_n(x)\Big)=-n(n+1)P_n(x),
\eeq
and satisfy the orthonormality in $L^2[-1,1]$, i.e., $\int_{-1}^1 P_i P_j=\delta_{ij}$ and $\int_{-1}^1P_i^2=1$.\\
Then, for any $w\in L^2[-1,1]$, we have $w=\sum_{i=0}^{\infty}c_iP_i$, $c_i=\int_{-1}^1 w(x)P_i(x)dx$.\\
In particular, we see that $P_0(x)=\frac{1}{\sqrt{2}}$, thus $c_0P_0=\frac{1}{2}\int_{-1}^1 w dx=:\bar w$, which is an average of $w$ over $[-1,1]$.
Then, since $w-\bar w=\sum_{i=1}^{\infty}c_iP_i$, using \eqref{Leq}, we have
\begin{align*}
\begin{aligned}
\int_{-1}^1 (1-x^2)|w'|^2 dx &= -\int_{-1}^1\Big((1-x^2)w'\Big)'w dx = -\int_{-1}^1 \Big((1-x^2)w'\Big)'(w-\bar w)dx \\
&=-\sum_{i\ge1}\sum_{j\ge1}\int_{-1}^1 c_i\Big((1-x^2)P_i'\Big)'c_jP_jdx \\
&=\sum_{i\ge1}\sum_{j\ge1}\int_{-1}^1 c_ic_j i(i+1)P_iP_jdx \\
&=\sum_{i\ge1}\int_{-1}^1 i(i+1)c_i^2 P_i^2 dx \ge 2 \sum_{i\ge1}\int_{-1}^1c_i^2P_i^2 dx =2\int_{-1}^1(w-\bar w)^2 dx.
\end{aligned}
\end{align*}
Therefore, we have
\[
\int_{-1}^1(w-\bar w)^2 dx \le \frac{1}{2}\int_{-1}^1 (1-x^2)|w'|^2 dx.
\]
By a change of variable as $W(x):=w(2x-1)$, we have
\[
\int_{0}^1(W-\bar W)^2 dx \le \frac{1}{2}\int_{0}^1 x(1-x)|W'|^2 dx,
\]
where $\bar W=\int_0^1 W dx$.

\end{appendix}

\bibliography{NS_Kang_Vasseur_revision}
\end{document}